\documentclass[11pt]{article}
\usepackage{amsthm}
\usepackage{mathrsfs}
\makeatletter
\def\th@plain{%
  \thm@notefont{}% same as heading font
  \itshape % body font
}
\def\th@definition{%
  \thm@notefont{}% same as heading font
  \normalfont % body font
}
\makeatother
\usepackage{amsthm}
\usepackage{mathrsfs}
\usepackage{amsmath}
\usepackage{amsfonts}
\usepackage{latexsym}
\usepackage{amssymb}
\usepackage{graphicx}
\usepackage{srcltx,graphicx,amsmath,amssymb,amsthm,mathrsfs,bm}
\usepackage{multirow}
\usepackage{soul}
\usepackage{tabularx,ragged2e}
\newcolumntype{C}{>{\Centering\arraybackslash}X} % centered "X" column
\usepackage{caption}
\usepackage{subcaption}
\usepackage{float}
\usepackage{xcolor}
\usepackage{pdfpages}

\newtheorem{theorem}{Theorem}[section]
\newtheorem{assumption}[theorem]{Assumption}
\newtheorem{corollary}[theorem]{Corollary}
\newtheorem{definition}[theorem]{Definition}
\newtheorem{notation}[theorem]{Notation}
\newtheorem{example}[theorem]{Example}
\newtheorem{lemma}[theorem]{Lemma}

\newtheorem{proposition}[theorem]{Proposition}
\newtheorem{remark}[theorem]{Remark}

\numberwithin{equation}{section}
\numberwithin{figure}{section}
\usepackage[a4paper]{geometry}
 \newcommand{\wOmega}{\widehat{\Omega}}
 \newcommand{\pwOmega}{\partial \widehat{\Omega}}

\newcommand{\Imap}[4]{\cI_{\Gamma_{#1.#2} \ra \Gamma_{#3,#4}}}
\newcommand{\Imaps}[4]{\cI_{\Gamma_{#1}^{#2} \ra \Gamma_{#3}^{#4}}}
\newcommand{\imp}{\mathrm{imp}} 
\newcommand{\hu}{{u}}
\newcommand{\hL}{{L}}
\newcommand{\hn}{{n}}
\newcommand{\hk}{{k}}
\newcommand{\hgamma}{{\gamma}}
\newcommand{\hrho}{{\rho}}
\newcommand{\hGamma}{{\Gamma}}
\newcommand{\hOmega}{{\Omega}}
\newcommand{\hdelta}{{\delta}}

\newcommand{\beq}{\begin{equation}}
\newcommand{\eeq}{\end{equation}}

\newcommand{\cA}{{\mathcal A}}

\newcommand{\cI}{{\mathcal I}}
\newcommand{\cL}{{\mathcal L}}
\newcommand{\cP}{{\mathcal P}}

\newcommand{\bcT}{\boldsymbol{\mathcal{T}}}
\newcommand{\bcL}{\boldsymbol{\mathcal{L}}}

\newcommand{\bcU}{\boldsymbol{\mathcal{U}}}
\newcommand{\cR}{\mathcal{R}}  
\newcommand{\tcR}{\widetilde{\mathcal{R}}}

\newcommand{\cT}{{\mathcal T}}

\newcommand{\cU}{{\mathcal U}}

\newcommand{\bx}{\boldsymbol{x}}

\newcommand{\bu}{\boldsymbol{u}}
\newcommand{\be}{\boldsymbol{e}}
\newcommand{\bv}{\boldsymbol{v}}

\newcommand{\supp}{\mathrm{supp}}

\newcommand{\ra}{\rightarrow}

\newcommand{\bbU}{\mathbb{U}}

\newcommand{\eps}{\varepsilon}
\newcommand{\re}{{\rm e}}
\newcommand{\ri}{{\rm i}}

\usepackage{color}
\usepackage{hyperref}
\definecolor{myblue}{rgb}{0,0,0.6}
\definecolor{darkgreen}{rgb}{0,0.5,0}
\hypersetup{colorlinks=true,
linkcolor=myblue,citecolor=myblue,filecolor=myblue,urlcolor=myblue}

\definecolor{escol}{rgb}{0,0,0.6}
\definecolor{sgcol}{rgb}{0,0,0.7}
\definecolor{estcol}{rgb}{0.5,0,0}
\definecolor{esnewcol}{rgb}{0,0.5,0}

\newcommand{\igg}[1]{{\color{black}{#1}}}

\newcommand{\iggg}[1]{{\color{black}{#1}}}
\newcommand{\ig}[1]{{\color{black}{#1}}}

\newcommand{\beqs}{\begin{equation*}}
\newcommand{\eeqs}{\end{equation*}}
\newcommand{\bit}{\begin{itemize}}
\newcommand{\eit}{\end{itemize}}
\newcommand{\ben}{\begin{enumerate}}
\newcommand{\een}{\end{enumerate}}
\newcommand{\bal}{\begin{align}}
\newcommand{\eal}{\end{align}}
\newcommand{\bals}{\begin{align*}}
\newcommand{\eals}{\end{align*}}
\newcommand{\bre}{\begin{remark}}
\newcommand{\ere}{\end{remark}}
\newcommand{\bpf}{\begin{proof}}
\newcommand{\epf}{\end{proof}}
\newcommand{\ble}{\begin{lemma}}
\newcommand{\ele}{\end{lemma}}
\newcommand{\bco}{\begin{corollary}}
\newcommand{\eco}{\end{corollary}}
\newcommand{\bex}{\begin{example}}
\newcommand{\eex}{\end{example}}
\newcommand{\bth}{\begin{theorem}}
\newcommand{\enth}{\end{theorem}}

\newcommand{\tfa}{\text{ for all }}

\newcommand{\tin}{\text{ in }}
\newcommand{\ton}{\text{ on }}

\newcommand{\tand}{\text{ and }}

\newcommand*{\N}[1]{\left\|#1\right\|}

\newcommand{\noi}{\noindent}
\newcommand{\tendi}{\rightarrow\infty}
\newcommand{\rI}{\mathrm{I}}

\newcommand{\mythmname}[1]{\textbf{\emph{(#1.)}}}

\newcommand{\pdiff}[2]{\frac{\partial #1}{\partial #2}}
\newtheorem{experiment}[theorem]{Experiment}
\usepackage[outdir=./]{epstopdf}
\usepackage{verbatim}

\title{Convergence of parallel overlapping domain decomposition methods  for the Helmholtz  equation }

\author{Shihua Gong$^{*}$,  Martin J.  Gander$^{\dagger }$, Ivan G.~Graham$^{**}$, \\ David Lafontaine$^{\dagger \dagger} $,  and Euan A.~Spence$^{**}$, 
  \\[2ex]
$^{*}$ School of Science and Engineering, The Chinese University of Hong Kong,\\ Shenzhen, Guangdong 518172, China \\ $^{**}$ Department of Mathematical Sciences, University of Bath, Bath BA2
7AY, UK.\\ 
 $^{\dagger }$ 
 Section de Math\'ematiques,
Universit\'e de Gen\`eve,
Suisse\\
$^{\dagger \dagger}$ 
CNRS and Institut de Math\'ematiques de Toulouse, UMR5219; Universit\'e de Toulouse, \\
CNRS; UPS, F-31062 Toulouse Cedex 9, France
}
\date{\today}

\begin{document}
\maketitle

\begin{abstract}

We analyse  parallel  overlapping Schwarz  domain decomposition
  methods for the Helmholtz equation, where the exchange of information between subdomains is achieved using first-order absorbing (impedance) transmission conditions,  together with 
   a partition of unity.  We  provide a novel analysis of  this method at the PDE level (without   discretization).
  First,  we formulate the method as a fixed point iteration, and show (in dimensions 1,2,3) that it is well-defined in a tensor product of appropriate local function spaces, each  with $L^2$ impedance boundary data.
We then  obtain a  bound on the norm of the fixed point  operator in terms of the local norms of  certain impedance-to-impedance maps arising from local interactions between subdomains. These bounds  provide  conditions under which (some power of)  the fixed point operator is a  contraction.
 In 2-d, for rectangular domains and strip-wise  domain decompositions (with each subdomain only overlapping its immediate neighbours),   we present  two techniques for  verifying the assumptions on the impedance-to-impedance maps that ensure power contractivity of the fixed point operator.  The first is  through semiclassical analysis, which gives rigorous estimates valid as the frequency tends to infinity. At least for a model case with two subdomains, these results verify the required assumptions for sufficiently large overlap. For more realistic domain decompositions,
we  directly  compute the norms of the  impedance-to-impedance maps by solving  certain canonical (local)  eigenvalue problems.   
  We give  numerical experiments that illustrate the theory.  These also show  that the iterative method remains convergent and/or provides a good preconditioner in cases not covered by the theory, including for general domain decompositions, such as those obtained via automatic graph-partitioning software.

\bigskip

\noi{\bf MSC2010 classification}: 65N22, 65N55, 65F08, 65F10, 35J05\\
\noi{\bf Keywords}: Helmholtz equation, High frequency, Domain decomposition, Overlapping subdomains,  Schwarz method, Impedance transmission condition  

\end{abstract}

\section{Introduction}  \label{sec:Intro} 

\subsection{The Helmholtz problem.}
\label{subsec:Helm} 
Motivated by  the large range of applications,  there is currently great interest in designing and analysing 
{domain decomposition methods}
 for discretisations of the Helmholtz equation
\begin{align} \label{eq:Helm}
 \Delta u + k^2 \, u =-f,  \quad \text{on} \quad \Omega,  
  \end{align}  
  on a $d-$dimensional bounded  domain $\Omega$  ($d = 2,3$),  with $k$ the (\iggg{spatially constant,  but possibly large}) angular frequency.
  \iggg{While the algorithm considered here is easily applicable to \eqref{eq:Helm} with very general boundary condition, geometry and variable $k$,  our theory is restricted here to the homogeneous interior impedance problem with  $k$  constant,  and the boundary condition}

  \begin{align}
\frac{    \partial u} {\partial n} - \ri \igg{k}  u = g \quad \text{on} \quad
     \partial \Omega, \label{eq:BC}
  \end{align}
  where $\partial u/\partial n$ \iggg{is the normal derivative, outward from   $\Omega$.}

    \subsection{Parallel domain decomposition method} \label{sec:parallel_algorithm}
To solve  \eqref{eq:Helm}, \eqref{eq:BC},  we  use a parallel overlapping Schwarz method with impedance transmission condition,    based on a set of Lipschitz polyhedral    
subdomains   $\{\Omega_j\}_{j = 1}^N$,  
forming  an overlapping cover of $\Omega$. \iggg{To derive this, note that if  $u$ solves  \eqref{eq:Helm}, \eqref{eq:BC},   then,   the restriction of $u$ to $\Omega_j$:}  
\begin{align} \label{eq:urest} u_j : =  u\vert_{\Omega_j}, \quad \text{for} \quad j \in \{ 1, \ldots, N\},\end{align}
satisfies 
\begin{align}
    (\Delta + k^2)u_j & = - f  \quad  &\text{in } \  \Omega_j, \label{eq11} \\
  \left(\frac{\partial }{\partial n_j} - \ri k \right) u_j & = \left(\frac{\partial }{\partial n_j} - \ri k \right) u \quad  &\text{on }  \ \partial \Omega_j\backslash \partial \Omega,  \label{eq12} \\
  \left(\frac{\partial }{\partial n_j} - \ri k \right) u_j & = g  \quad  &\text{on } \  \partial \Omega_j  \cap \partial \Omega,\label{eq13}
\end{align}
\igg{where $\partial /\partial n_j$ denotes the outward normal derivative on $\partial \Omega_j$.} 
To iteratively solve \eqref{eq11} -- \eqref{eq13}, we  introduce  a partition of unity (POU), 
 $\{\chi_j\}_{j = 1}^N$, with the properties 
 \beq
 \left.
 \begin{array}{ll} 
  \text{for each} \quad j: & \  \supp \chi_j \subset \overline{\Omega_j}, \quad
   0 \leq  \chi_j(\bx) \leq 1\,\,  \text{when } \bx \in \overline{\Omega_j},\\
   & \\
   \quad\tand\quad  & 
  \sum_j  \chi_j(\bx) = 1 \,\tfa \bx \in \overline{\Omega}.
 \end{array}
 \right\} \label{POUstar}
  \eeq
Then, the parallel Schwarz method reads as follows: given an iterate $u^n$ defined on $\Omega$,  we solve each  local problem on $\Omega_j$ for  $u^{n+1}_j$ ,
\begin{align}
  (\Delta + k^2)u_j^{n+1}  & = - f  \quad  &\text{in } \  \Omega_j,  \label{eq21}\\
  \left(\frac{\partial }{\partial n_j} - \ri k \right) u_j^{n+1}  & = \left(\frac{\partial }{\partial n_j} - \ri k \right) u^n \quad  &\text{on }  \ \partial \Omega_j\backslash \partial \Omega, \label{eq22}  \\
    \left(\frac{\partial }{\partial n_j} - \ri k \right) u_j^{n+1}  & = g  \quad  &\text{on } \  \partial \Omega_j  \cap \partial \Omega,  \label{eq23} 
\end{align}
and the new iterate is the weighted sum of the local solutions:
\begin{align} \label{star}
  u^{n+1} := \sum_\ell \chi_\ell  u_\ell ^{n+1}.
\end{align}
This  well-known method  can be thought of as a generalization of the classical  algorithm of Despr\'{e}s \cite{De91},  \cite{BeDe:97} to the case of overlapping subdomains. The form of the algorithm stated above can be found  in \cite[\S2.3]{DoJoNa:15}. The novel contribution of this paper is the convergence analysis of the  method. 

\subsection{The main results and structure of the paper} \label{subsec:main_results}

The main results of this paper are as follows.
  \begin{enumerate}
  \item A proof that the iterative method \eqref{eq21}-\eqref{star} is well-defined for general Lipschitz subdomains (Theorem \ref{thm:algorithm}) using well-posedness results for the Helmholtz equation on Lipschitz domains and the harmonic-analysis results of \cite{JeKe:95}.     
  \item The formulation of the fixed-point iteration \eqref{eq:errit}  for the \iggg{error vector $\be^n$, where
  $e_j^ n = u_j - u_j^n,  \ j  = 1, \ldots, N$, and the expression of powers of the fixed-point operator  $\bcT$} in terms of ``impedance-to-impedance maps'' linking \iggg{pairs of subdomains}  with non-trivial intersection (Theorem \ref{thm:main_T2}).
  \item For 2-d rectangular domains  covered by overlapping strips,  with each subdomain only overlapping its immediate neighbours, sufficient conditions for 
  $\bcT^N$ to be a contraction, where $N$ is the number of subdomains (Theorem \ref{thm:TLU} \iggg{and its corollaries}). These conditions are formulated in terms of norms of impedance-to-impedance maps and compositions of such maps.
\item A summary and explanation of the results from the companion paper \cite{LaSp:21} that bound impedance-to-impedance maps using rigorous high-frequency asymptotic analysis (a.k.a., semiclassical analysis). In particular, these results indicate that Theorem \ref{thm:TLU} implies contractivity of $\bcT^N$ {for a model case with two subdomains and} provided that the overlap is sufficiently large (see \S\ref{subsec:SCA}).
    \item Finite element  experiments (\S\ref{sec:numerical}) that both back up the theory and investigate scenarios out of the theory's current reach.
{Since the theory in Points 1-4 is at the continuous level without discretization, \S \ref{sec:fem}  first gives a description of the finite element algorithms used, along with justification that the results in \S\ref{sec:numerical} are reliable.}

The experiments related to the theory illustrate how the good behaviour of the relevant impedance-to-impedance maps induces good convergence of the iterative method. Those beyond  the theory show, for square domains and general domain decompositions, that the fixed point operator still enjoys a power contractivity property
(\S \ref{sec:numerical}).
  \end{enumerate}    

Regarding 2 and 3:~to our knowledge this is the first time that overlapping DD methods for Helmholtz have been analysed in terms of 
impedance-to-impedance maps. This analysis therefore gives a route to analyse overlapping DD methods for Helmholtz using the PDE theory of the Helmholtz equation, \iggg{which will then be} the focus of \cite{LaSp:21}.
Interest in  impedance-to-impedance (a.k.a., Robin-to-Robin)  maps can be widely  found in  the non-overlapping literature - see the references {in the literature review} below.
These   maps also arise  in the {formulation} of fast direct methods (e.g. \cite{GiBaMa:15}, \cite{PeTuBo:17}) and the recent work \cite{BeCaMa:21} analyses 
these maps in this setting (using complementary techniques to those in \cite{LaSp:21}). 
Previous work of some of the authors (e.g., \cite{GrSpZo:20, GoGrSp:20})
  also used  PDE theory to analyse  overlapping DD preconditioners; while this work was able to  cover very general geometries,  it was limited to the case when $\Im k>0$, corresponding to media with some absorptive properties. In the present paper we consider only the propagative case $\Im k=0$. 
  
Regarding 3:~In 1-d we recover the
  classical  result (a special case of \cite{NaRoSt:94}) that, with $N$ subdomains, the $N$th power of the fixed point operator is zero, and so the  iterative method converges in $N$ iterations.

The structure of the paper is as follows. 
\S\ref{sec:assumptions} contains the necessary well-posedness and regularity theory for the Helmholtz equation needed to prove the results described in Point 1 above. 
\S\ref{sec:fixedpoint} formulates the fixed-point iteration as described in Point 2.
\S\ref{sec:conv} focuses on 2-d strip decompositions as described in Point 3. 
\S\ref{sec:fem} describes the set-up for the finite-element computations used to illustrate the results, with the results of these computations given in \S\ref{sec:numerical}.

\subsection{Literature review }\label{sec:literature}

In the last decade there has been an  explosion of progress in the construction and analysis of 
solvers for 
  frequency-domain wave problems. 
Here we highlight those parts of the literature most related to the present paper;  
more substantial recent reviews can  be found, for example in \cite{GaZh:19} and in the introductions to \cite{GrSpZo:20} and \cite{TaZeHeDe:19}. 

The method \eqref{eq21}-\eqref{star} can be thought of  as a straightforward extension  of the parallel non-overlapping method of Despr\'{e}s \cite{De91},  \cite{BeDe:97} to the case of overlapping subdomains.  In \cite{De91}, \cite{BeDe:97} the 
coupling between subdomains at each iteration is achieved by feeding  to each subdomain  impedance data from its neighbours
at the previous iteration. In the overlapping case considered here,  the boundary impedance data for the next iterate
on a given subdomain is a weighted average of data coming from all subdomains overlapping it  (see \eqref{eq22}  and  \eqref{star}).  

The results of \cite{De91},  \cite{BeDe:97} 
   proved  convergence of the iterative algorithm via an energy argument. Although a rate of convergence was not provided, when it first appeared  this work   
   inspired huge interest in non-overlapping methods which {continues} today and a recent  review can be found in the introduction to
     \cite{ClPa:20}.
Indeed the results of \cite{De91},  \cite{BeDe:97} have recently been extended to higher-order boundary conditions in \cite{DeNiTh:20}.
Furthermore, there has been much interest in handling cross points in non-overlapping domain decomposition methods, e.g., at the PDE level in \cite{Cl:21}  and at the discrete level in
\cite{ClPa:20} and \cite{MoRoAnGe:20}. In \cite{Cl:21,ClPa:20} the ``multitrace'' theory, \iggg{originally introduced for  boundary integral equations}, was  used  to prove the contractivity of a certain non-overlapping domain decomposition method, even in the presence of   cross points (where more than two subdomains meet), 
albeit at the cost of inverting a global operator coupling the subdomains.

An early paper on transmission conditions for the overlapping case \cite{NaRoSt:94} showed that if the impedance
transmission condition was replaced by a transparent condition (constructed using the appropriate
Dirichlet to Neumann (DtN) map), then for a one dimensional sequence of $N$ subdomains with a first and last subdomain, the iterative method converges in $N$ steps; see also \cite{EnZh:98} for complementary results on the optimal choice of boundary condition. Since DtN maps are  not practical to compute,  a great deal of interest then focussed on effective approximations of them. For example, second order impedance operators were introduced in   \cite{GaMaNa:02} and discussed in many subsequent papers.  Pad\'{e} approximations  of the DtN map were investigated in \cite{Boubendir} and non-local integral equation-based techniques were  proposed in \cite{CoJoLe:20}, 
\iggg{although again \cite{Boubendir,CoJoLe:20} concern the non-overlapping case.}

The above ``Helmholtz-specific'' algorithms can also be thought of as examples of the more general class
of Optimized Schwarz methods,  a concept that is applicable to a wide range of PDEs, in which
transparent boundary conditions on subdomain boundaries are approximated by Robin or higher order
transmission conditions, 
optimized for fastest convergence  -- see, for example,  \cite{gander2000optimized, GaMaNa:02,GaZh:16a} and \cite{GanderOSM} and the references therein. For example,  \cite{LiuXu:2019} studies the (dis)advantages of large  overlap for a particular Schwarz method for the Helmholtz equation using two-sided Robin transmission conditions and an additional (optimized)
  relaxation parameter. However this particular method is
  somewhat different from the (essentially classical) general  method
  analysed  in the present paper. 
A historical perspective on Schwarz methods in general is given in \cite{gander2008schwarz}. 

The methods described above aim at maximising parallelism by solving independent subproblems at each iteration.  Since wave problems are fundamentally propagative there is  also great potential for alternating (or `multiplicative')  methods, in which solutions of subproblems are passed from subdomain to subdomain in the iterative process.
While these are less inherently parallel they can potentially gain in the number of iterations needed for convergence.
Algorithms that  can be classified as inherently multiplicative  include  the sweeping
preconditioner \cite{EY2}, the source transfer domain
decomposition \cite{Chen13a,DuWu:20}, the single-layer-potential
method \cite{Stolk}, the method of polarized traces
\cite{ZD}. 
All these methods are very much related, and can also be
understood in the context of optimized Schwarz methods -- see \cite{GaZh:19}. \ig{A related multiplicative method is the double sweep method,  introduced in  \cite{NN97,Vion} and partly analysed in \cite{NN97,BoCaNa:20}.}
Mostly these algorithms do not allow cross-points, although extensions to regular decompositions with cross points are proposed in  \cite{TaZeHeDe:19} and  \cite{leng2019additive}.

Related to the question of convergence  of Schwarz methods for Helmholtz problems, we remark that in the recent paper \cite{GaZh:22}, Fourier semi-analytical
techniques (different from, but complementary to,  the methods used in this paper)
are used
to study contractivity for strip-type domain decompositions for many different interior transmission and 
outer boundary conditions.

\subsection{Discussion of our results in the context of the literature}

While the majority of the theory discussed  above concerns  non-overlapping DD for Helmholtz, the present  paper develops
  a new  convergence analysis  in the overlapping case.
  As explained in  \S \ref{sec:fem}, the corresponding  solver  is  closely related  to the Optimized Restricted Additive Schwarz (ORAS) preconditioner, which  provides  the  foundational  one-level component for several  large-scale wave propagation solvers  e.g., \cite{MEDIMAX,BoDoGrSpTo:17c,HaJoNa:17,BoDoJoTo:20}. Thus  {the} theory {in the present paper} underpins several existing successful algorithms.  Unlike previous work (e.g., \cite{GoGrSp:20}) that aimed at proving that the field of values of the preconditioned operator did not include the origin - an extremely strong requirement -  our analysis here has  the more modest aim   of proving  power contractivity of the fixed point operator. This turns out to be not only provable {for a model problem} in simple-enough geometry but also observable in more general situations, giving hope that the present theory can be generalised. Power contractivity of the fixed point operator also ensures convergence rates for preconditioned GMRES.

  The estimates  ensuring power contractivity  in  \S \ref{sec:conv}, involving the norms of impedance-to-impedance maps  (which can be computed by solving canonical eigenvalue problems) are in some sense analogous to  condition number estimates in the positive definite case: both estimates provide   upper bounds that  can be controlled in certain parameter ranges, but the actual  value of the bound is rarely computed  when  solving a particular problem (so the bounds are ``descriptive'' rather than ``prescriptive'').        

\subsection{The discrete analogue of the results of this paper}

In  \cite{GoGaGrSp:21} it was  shown  (see also \S \ref{subsec:iterative}) that  a natural 
  finite element counterpart  of the  parallel iterative method considered in the current 
  paper is the Restricted Additive Schwarz method with impedance transmission condition (often called
  the Optimized Restricted Additive Schwarz (ORAS) method).

 Since this paper was written, three of the current authors developed
  a convergence theory for ORAS, thus providing  a discrete version of the theory given here.
  These results  are presented  in \cite{GoGrSp:21}.

   The theory in  \cite{GoGrSp:21} applies to discrete Helmholtz systems arising from     
 conforming nodal finite elements  of any polynomial
order and a general theory for discrete fixed point iteration analogous to \S\S  \ref{sec:assumptions}, \ref{sec:fixedpoint} on general Lipschitz domains and partitions of unity is presented.    For domain decompositions in strips in 2-d, we show that, when the mesh size is small enough, ORAS inherits the
convergence properties of the parallel iterative method at the PDE level which are proved here, 
independent of the  polynomial order of the elements. The proof relies on characterising the ORAS iteration in terms of discrete ‘impedance-to-impedance maps’ on local discrete Helmholtz-harmonic spaces, which we prove (via a novel weighted finite-element error analysis) converge as $h \rightarrow 0$ in the
operator norm to their non-discrete counterparts (i.e., the operators analysed here).
This discrete theory thus justifies the use of the finite element method to illustrate the properties of the iterative method at the PDE level, as we have done in {\S\ref{sec:numerical} of} this paper.
   
    \section{Helmholtz well-posedness and regularity theory}
 \label{sec:assumptions}
 \subsection{Basic notation and assumptions} 
\paragraph{Standard norms.} 

Let $(\cdot, \cdot)_\Omega$ denote   the usual $L^2(\Omega)$ inner product with induced norm $\Vert \cdot \Vert_\Omega$ and 
          denote the standard weighted $H^1$ norm by
         \begin{align} \label{weighted}  \Vert v \Vert_{1,k,\Omega}^2  =  \Vert \nabla v \Vert_\Omega^2 + k^2 \Vert v \Vert_\Omega^2 .  \end{align} 
  Let $\langle \cdot, \cdot \rangle_{\partial \Omega}$ denote the $ L^2(\partial \Omega )$ inner product,  with induced norm $\Vert \cdot \Vert_{\partial \Omega}$. For inner products over measurable   subsets
  $\widetilde{\Omega} \subset \Omega$ and $\widetilde{\Gamma} \subset \Gamma$, we
  write    $(\cdot, \cdot)_{\widetilde{\Omega}}$ and $\langle \cdot,\cdot\rangle_{\widetilde{\Gamma}}$.

\paragraph{Sesquilinear forms.} We define the global and local sesquilinear forms by
       \begin{align}\label{eq:sesq1}
      a(u,v) & := (\nabla u, \nabla v)_\Omega - k^2(u,v)_\Omega - \ri k \langle u,v\rangle_{\partial \Omega} \quad \text{for} \ \ u,v \in H^1(\Omega),      \\
\text{and} \quad       a_\ell(u,v)      &   := (\nabla u, \nabla v)_{\Omega_\ell} - k^2(u,v)_{\Omega_\ell} - \ri k \langle u,v\rangle_{\partial \Omega_\ell} \quad \text{for} \ \  u,v \in H^1(\Omega_\ell).  
       \label{eq:sesq2} \end{align}

\paragraph{Prolongation and restriction.}
     We  build a prolongation  $\tcR_\ell^\top: H^1(\Omega_\ell) \rightarrow H^1(\Omega)$ by multiplication  by the POU, i.e., for each $v_\ell \in H^1(\Omega_\ell)$, 
     \begin{align} \label{RtildeT} \tcR_\ell^\top  v_\ell  = \left\{ \begin{array}{ll} \chi_\ell v_\ell \quad & \text{on} \quad \Omega_\ell , \\  0 \quad &  \text{elsewhere} . \end{array}\right.  
     \end{align}
     Recalling that $u_\ell$ denotes the  restriction of $u$ to $\Omega_\ell$ (see \eqref{eq:urest}), we have  the important property 
                           \begin{align} \label{eq:RR1} \sum_{\ell = 1}^N \tcR_\ell^\top u\vert_{\Omega_\ell}= \sum_{\ell = 1}^N \chi_\ell u\vert_{\Omega_\ell}  = u, \quad \text{for all} \ \ u \in H^1(\Omega)  . \end{align}  

The main purpose of this section is to justify step \eqref{eq22} in the domain decomposition algorithm, by showing that for each $n$, the impedance trace of $u^n$ is in $L^2(\partial \Omega_\ell)$, for each $\ell$. This then ensures that, for each $\ell$,   $u_\ell^{n+1}$ is  well-defined in the space $U(\Omega_\ell)$ defined below and hence $u^{n+1} \in U(\Omega) $, so that $u^{n+1}$ in turn   provides suitable   $L^2$ impedance traces on each $\partial \Omega_\ell$ for the next iteration.
The  main result of this section is Theorem \ref{thm:algorithm}. 
To prove it  we need to  analyse  \eqref{eq21}-\eqref{star} 
in a space of higher regularity than $H^1$.  
In what follows we need the following further property of the partition of unity $\{ \chi_\ell \}$.  
\begin{assumption} \label{ass:POU} 
In addition to satisfying \eqref{POUstar}, $\chi_\ell \in C^{1,1} (\Omega_\ell)$. 
 \end{assumption} 

Such a partition of unity exists by, e.g., \cite[Theorem 3.21 and Corollary 3.22]{Mc:00}.
We note for later that Assumption \ref{ass:POU} implies that 
$\partial \chi_\ell/\partial n_\ell = 0$ on $\partial \Omega_\ell\setminus \partial\Omega$, and thus,
for any $v_\ell \in H^1(\Omega_\ell)$,
   \begin{align} \frac{\partial }{\partial n_\ell} (\chi_\ell v_\ell) = \frac{\partial \chi_\ell }{\partial n_\ell} v_\ell +
     \chi_\ell \frac{\partial v_\ell  }{\partial n_\ell} \ =\ 0 \quad \text{on} \quad \partial \Omega_\ell\backslash \partial\Omega.  \label{eq:diffprod} \end{align}

\begin{notation} \label{not:twiddles}
Where possible,  we explicitly indicate the dependence of our estimates on  the wavenumber $k$.  In this context, we always assume $k \geq k_0$ where $k_0>0$ is chosen a priori and we use the notation   $A\lesssim B$ to mean $A\leq CB$ with a constant $C$  independent of $k$ (but possibly depending on $k_0$) and $A \sim B$ to mean $A\lesssim B$ and $B\lesssim A$. 
  \end{notation} 
\subsection{The Helmholtz problem in spaces $U(D)$ and $U_0(D)$}\label{sec:spaces}

In this subsection $D$ denotes  a general Lipschitz domain, with boundary $\partial D$. 
\begin{definition}\label{def:U} 
Let
   $$ U (D) := \big\{ u \in H^1(D): \,  \Delta u  \in L^2(D), \, \partial u / \partial n  
   \in L^2(\partial D) \big\},  $$
with norm
   \begin{equation}\label{eq:normU}
   \|u\|_{U(D)} ^2: = k^{-2} \N{\Delta u}^2_{L^2(D)}  + \N{\nabla u }^2_{L^2(D)}  +  k^2 \N{ u}_{L^2(D)}^2
   + \N{\partial u/\partial n}^2_{L^2(\partial D)} + k^2\N{u}^2_{L^2(\partial D)}. 
      \end{equation}
Since the trace operator maps $H^1(D)$ to $H^{1/2}(\partial D)\subset L^2(\partial D)$, an equivalent definition of $U(D)$ is  
   $$ U (D) = \big\{ u \in H^1(D): \,  \Delta u +k^2 u \in L^2(D), \, \partial u / \partial n - \ri k u 
   \in L^2(\partial D) \big\}.  $$
Let
   $$ U_0 (D) := \big\{ u \in H^1(D): \,  \Delta u + k^2 u = 0
   \ \text{in} \  D, \ \partial u / \partial n - \ri k u \in L^2(
   \partial D) \big\} \ \subset U(D);  $$
in the rest of the paper we refer to $U_0(D)$ as the space of \emph{Helmholtz-harmonic  functions}  on $D$.
\end{definition} 

\ble[Well-posedness of the Interior Impedance Problem in  $U(D)$]\label{lem:IIPU}
Suppose  $D$ is Lipschitz and consider the problem  \begin{align} \label{IIP}\Delta u +k^2 u =-f  \quad \text{in} \quad  D \quad \text{ and}
                                                       \quad \partial u/\partial n - \ri k u= g \quad \text{on} \ \ \partial D, \end{align}
  with $f\in L^2(D)$ and $g\in L^2(\partial D)$. Then there exists a unique solution $u \in U(D)$ and $C_j= C_j(k)$, $j=1,2,$ such that 
\beq\label{eq:Saturday1}
\N{u}_{U(D)} \lesssim C_1(k) \N{f}_{L^2(D)} + C_2(k)\N{g}_{L^2(\partial D)}.
\eeq
\end{lemma}

\bpf[Proof of Lemma \ref{lem:IIPU}]
By the standard result about well-posedness of the interior impedance problem for Lipschitz $D$ (see, e.g., \cite[\S\S6.1.3, 6.1.6]{Sp:15}), a unique solution $u$ exists and there exist 
$C_j= C_j(k)$, $j=1,2,$ such that 
\beq\label{eq:Friday3}
\N{\nabla u}_{L^2(D)} + k \N{u}_{L^2(D)}\leq C_1(k)\N{f}_{L^2(D)} + C_2(k)\N{g}_{L^2(\partial D)}.
\eeq
Without loss of generality we can assume  that  $C_j(k) \gtrsim 1$, for $j = 1,2$. 
Then, multiplying the PDE in \eqref{IIP} by $\overline{u}$ and using Green's identity (see, e.g., \cite[Lemma 4.3]{Mc:00}), we find that 
\beqs
\langle \partial u/\partial n, u\rangle_{\partial D} -\N{\nabla u}^2_{L^2(D)} + k^2 \N{u}^2_{L^2(D)} = -\int_D f\, \overline{u}.
\eeqs
Inserting the boundary condition from \eqref{IIP}, taking the imaginary part, and using the Cauchy--Schwarz inequality gives
\beq
\label{eq:inequality24} 
k \N{u}^2_{L^2(\partial D)} \lesssim \N{f}_{L^2(D)}\N{u}_{L^2(D)} + \N{g}_{L^2(\partial D)}\N{u}_{L^2(\partial D)} .
\eeq
Now, multiplying \eqref{eq:inequality24} through by $k$ and using the inequality $2ab \leq  {\eps} a^2 + {\eps^{-1}}b^2$, {for any $a,b,\eps>0$} to estimate {both terms on} the right-hand side,  
we have  
\beqs
k^2  \N{u}^2_{L^2(\partial D)} \lesssim  \N{f}_{L^2(D)}^2 + \N{g}^2 _{L^2(\partial D)} +  k^2\N{u}_{L^2(D)}^2.
\eeqs
Combining this with \eqref{eq:Friday3}, we obtain 
\beq\label{eq:Friday3a}
\N{\nabla u}_{L^2(D)} + k \N{u}_{L^2(D)}
+k  \N{u}_{L^2(\partial D)}\lesssim(2 C_1(k) +1)\N{f}_{L^2(D)} + (2 C_2(k) + 1) \N{g}_{L^2(\partial D)},
\eeq
which is in the required form \eqref{eq:Saturday1}.

To complete the bound on $\|u\|_{U(D)}$, we therefore only need to bound $k^{-1}\|\Delta u\|_{L^2(D)}$ and $\|\partial u/\partial n\|_{L^2(\partial D)}$ by the right-hand side of \eqref{eq:Saturday1}; a bound on $k^{-1}\|\Delta u\|_{L^2(D)}$ follows from the PDE in \eqref{IIP} together with \eqref{eq:Friday3a}. The required bound on $\|\partial u/\partial n\|_{L^2(\partial D)}$ follows from the boundary condition in \eqref{IIP} together with  \eqref{eq:Friday3a}.
\epf

 \bre[The $k$-dependence of $C_1$ and $C_2$ in Lemma \ref{lem:IIPU}]  \label{rem:kdep}
For any Lipschitz domain $D$, $C_1(k)\gtrsim 1$ by \cite[Lemma 4.10]{Sp:14}, and when $D$ is a ball, 
$C_2(k)\gtrsim 1$ by \cite[Lemma 5.5]{BaSpWu:16}.

 If $D$ is \emph{either} Lipschitz star-shaped \emph{or} $C^\infty$, then \eqref{eq:Saturday1} holds with $C_1(k)\sim 1$ and $C_2(k)\sim 1$ by \cite[Equation 3.12]{MoSp:14} and \cite[Theorem 1.8 and Corollary 1.9]{BaSpWu:16} respectively. 

If $D$ is only assumed to be Lipschitz, then the sharpest existing result about the $k$-dependence of $C_1$ and $C_2$ is that $C_1(k)\sim k$ and $C_2(k)\sim k^{1/2}$ \cite[Theorem 1.6]{Sp:14}. If $\partial D$ is \ig{merely} piecewise smooth, then $C_1(k)\sim k^{3/4}$ and $C_2(k)\sim k^{1/4}$  \cite[Theorem 1.6]{Sp:14}. 
\ere

We now use results from the harmonic-analysis literature about the Laplacian on Lipschitz domains to give an alternative characterisation of the space $U(D)$. Let
\beqs
H^{3/2}(D;\Delta):= \Big\{ v \in H^{3/2}(D) : \Delta v \in L^2(D)\Big\},
\eeqs 
with norm
\beq\label{eq:H32norm}
\N{u}_{H^{3/2}(D;\Delta)}^2:= \N{u}_{H^{3/2}(D)}^2 + \N{\Delta u}^2_{L^2(D)}.
\eeq
The following theorem is a consequence of \cite[Corollary 5.7]{JeKe:95}; see \cite[Lemma 2]{CoDa:98}.
\begin{theorem}
  \label{thm:H32}
If $D$ is Lipschitz, then $U(D)=H^{3/2}(D;\Delta) $.
\end{theorem}

Since the $H^{3/2}(D;\Delta)$ norm is characterised only through norms on $D$ (as opposed to the norm on $U(D)$, which is characterised through norms on both $D$ and $\partial D$),
the following corollary holds. 
\begin{corollary}\label{cor:restrict}
If $v \in U(D)$ and $D'$ \igg{is a Lipschitz subdomain of $D$},  then the restriction of $v$ to $D'$ is in $U(D')$.
\end{corollary}

By Theorem \ref{thm:H32} and the definition of $U(D)$ in \S\ref{sec:spaces}, the norms $\N{\cdot}_{H^{3/2}(D;\Delta)}$ and 
$|||\cdot|||_{U(D)}$ defined by 
\beq\label{eq:H32norm2}
|||u|||_{U(D)} ^2: =
\N{\Delta u}_{L^2(D)}^2  + \N{\nabla u }_{L^2(D)}^2  +\N{ u}_{L^2(D)}^2
   + \N{\partial u/\partial n}_{L^2(\partial D)}^2+\N{ u}_{L^2(\partial D)}^2
\eeq
are equivalent. Moreover the equivalence constants are  $k-$independent  (since $k$ does not feature in the definition of either norm.). We therefore have the following corollary. 

\begin{corollary}[Neumann trace for functions in $H^{3/2}(\Omega;\Delta)$]
\beq\label{eq:Ntrace}
\N{\partial v/\partial n}_{L^2(\partial D)}\lesssim \N{v}_{H^{3/2}(D;\Delta)} \quad\tfa v\in H^{3/2}(D;\Delta),
\eeq
(i.e.,  the omitted constant is independent of $k$).
\end{corollary}

Because of the oscillatory character of $u$, one expects its $H^{3/2}$ norm to be $k^{1/2}$ times its $H^1$ norm; i.e.,
from \eqref{eq:Saturday1}, that $\N{u}_{H^{3/2}(\Omega)} \lesssim k^{1/2}\big( C_1(k) \N{f}_{L^2(D)} + C_2(k)\N{g}_{L^2(\partial D)}
\big)$. The following result almost proves this.

\begin{theorem}\label{cor:27}
Let $u$ be the solution of \eqref{IIP} with $f \in L^2(D)$ and $g \in L^2(\partial D)$. 
Then, for \iggg{any $\beta>1/2$}, there exists $C(\beta)>0$ such that
\beq\label{eq:H32bound}
\N{u}_{H^{3/2}(D)}\leq C(\beta) k^{\beta} \Big(C_1(k) \N{f}_{L^2(D)} + C_2(k)\N{g}_{L^2(\partial D)}\Big).
\eeq
\end{theorem}

\bpf
The combination of \cite[Theorem 3.2, Theorem 3.5, and Remark 3.3]{AmAm:21} implies that, if $\lambda \in \mathbb{C}$  with $\Re \lambda \geq \lambda_0$,
\beq\label{eq:AA1}
(\Delta-\lambda^2 )v =F \quad\tin D \qquad \tand \qquad\partial v/\partial n = G \quad\ton \partial D,
\eeq
with $F\in L^2(D)$ and \iggg{$G \in L^2(\partial D)$}, then, for all $0<r<1/2$ there exists $C_r>0$ such that
\beq\label{eq:AA2}
\N{v}_{H^{3/2}(D)}\leq C_r |\lambda|^{1-2r} \Big( |\lambda|^{1/2} \N{F}_{L^2(D)} + \iggg{\N{G}_{L^2(\partial D )}}\Big).
\eeq
Let $\lambda := \ri k +1$, then \eqref{eq:AA1} is satisfied with $v=u$, $G= \ri k u +g$, and \iggg{$F = -f- (2\ri k +1)u$}. Applying the bound \eqref{eq:AA2} we obtain that 
\beqs
\N{u}_{H^{3/2}(D)}\leq C_r k^{1-2r} \Big(k^{1/2}\big(\N{f}_{L^2(D)} + (k+1)\N{u}_{L^2(D)}\big)+ k\N{u}_{L^2(\partial D)}
+\N{g}_{L^2(\partial D)}
\Big).
\eeqs
The result \eqref{eq:H32bound} then follows from using \eqref{eq:Saturday1}, 
and the facts that $C_j(k)\gtrsim 1$, $j=1,2,$ by \cite[Lemma 4.10]{Sp:14}, \cite[Lemma 5.5]{BaSpWu:16} (as discussed in Remark \ref{rem:kdep}).
\epf

The following lemma studies the behaviour of the impedance trace of a function $u \in U(D)$ on an interface interior to $D$. This plays a key role in the analysis of the iterative method \eqref{eq21}-\eqref{star}. 

\begin{lemma}\label{lem:Friday} Suppose  $D,D'$ are both  Lipschitz domains  and $D' \subset D$.\\  
(i) If $u\in U(D)$, 
then
\beq\label{eq:Friday1}
\N{(\partial /\partial n - \ri k) u}_{L^2(\partial D')} \lesssim k \N{u}_{H^{3/2}(D';\Delta)}\leq  k \N{u}_{H^{3/2}(D;\Delta)}.
\eeq
(ii) If  $u \in U_0(D)$,  
then, 
\beq\label{eq:Friday2}
\N{(\partial /\partial n - \ri k) u}_{L^2(\partial D')} \lesssim \ig{k^2} \, C_2(k) \N{(\partial /\partial n - \ri k) u}_{L^2(\partial D)};
\eeq
i.e., ~the impedance-to-impedance map (defined more precisely  in \S \ref{subsec:relation}) is bounded as an operator from $ L^2(\partial D)$  to $L^2(\partial D')$.
\end{lemma}

\bpf[Proof of Lemma \ref{lem:Friday}]
The first inequality in \eqref{eq:Friday1} follows from \eqref{eq:Ntrace} and $$\|u\|_{L^2(\partial D')}\lesssim \|u\|_{H^1(D')} \leq  \|u\|_{H^{3/2}(D';\Delta)}.$$ The second inequality in \eqref{eq:Friday1} follows from the definition \eqref{eq:H32norm} of $\|\cdot\|_{H^{3/2}(D)}$ and the inclusion $D'\subset D$.
By \eqref{eq:Friday1}, to prove \eqref{eq:Friday2} we only need to prove that, \ig{for $ u \in U_0(D)$,} 
\beq\label{eq:Friday4}
\N{u}_{H^{3/2}(D;\Delta)}\ig{\ \lesssim  \ k } \,C_2(k) \N{(\partial /\partial n - \ri k) u}_{L^2(\partial D)}. 
\eeq
However, since 
$$ \N{u}_{H^{3/2}(D;\Delta)} \lesssim \Vert u \Vert_{H^{3/2} (D)} + \Vert \Delta u \Vert_{L^2(D)} = \Vert u \Vert_{H^{3/2} (D)} + k^2 \Vert  u \Vert_{L^2(D)},$$ \eqref{eq:Friday4} follows by using  \eqref{eq:H32bound} and  \eqref{eq:Saturday1}.
\epf

We make two remarks: 

(i) Lemma \ref{lem:Friday} makes no assumptions about the geometries of $D'$ and $D$, other than that they are both Lipschitz and $D'\subset D$.

(ii) The powers of $k$ in \eqref{eq:Friday1} and \eqref{eq:Friday2} are almost-certainly not sharp; this is because the right-hand side of the trace result \eqref{eq:Ntrace} involves a norm on $H^{3/2}(D;\Delta)$ that does not weight derivatives with the appropriate powers of $k$ (in contrast to, e.g., \eqref{eq:normU}). As far as we are aware, the result analogous to \eqref{eq:Ntrace} with 
an $H^{3/2}(D;\Delta)$ norm weighted with $k$ and a potentially-$k$-dependent constant has not yet been proved for general Lipschitz domains.

\subsection{Well-posedness of the domain decomposition  algorithm}
\label{subsec:well-posed}

 We now discuss the behaviour of the  algorithm \eqref{eq21}-\eqref{star} in the  product spaces:
\beq\label{eq:product}
 \bbU : = \prod_{\ell = 1} ^N U(\Omega_\ell)
 \quad\tand\quad
 \bbU_0 : = \prod_{\ell = 1} ^N U_0(\Omega_\ell).
  \eeq
   In this section we show the well-posedness of \eqref{eq21}-\eqref{star}   in $\bbU$.
The convergence analysis of \eqref{eq21}-\eqref{star} in the following section is set in $\bbU_0$.
First
we need the following lemma, which exploits the smoothness requirement on the partition of unity (Assumption \ref{ass:POU}).

\ble\label{lem:map} Suppose Assumption \ref{ass:POU} holds. 
(i) If $v_\ell \in U(\Omega_\ell)$ then $\chi_\ell v_\ell  \in U(\Omega_\ell)$.\ \ 
(ii) $\tcR_\ell^\top: U(\Omega_\ell) \rightarrow U(\Omega)$.
\ele

\bpf
(i) By Assumption \ref{ass:POU}, $\chi_\ell \in C^{1,\alpha}$ for $\alpha>1/2$. Therefore, by \cite[Theorem 1.4.1.1]{Gr:85},   $\chi_\ell v_\ell \in H^{3/2}(\Omega_\ell)$.
Since $\chi_\ell \in C^{1,1}$, Rademacher's theorem \cite[Page 96]{Mc:00} implies that $\Delta \chi_\ell$ exists almost everywhere as an $L^\infty$ function on $\Omega_\ell$, and thus 
\beqs
\Delta( \chi_\ell v_\ell)= \chi_\ell \Delta v_\ell + 2 \nabla\chi_\ell \cdot \nabla v_\ell + v_\ell \Delta \chi_\ell  \in L^2(\Omega_\ell);
\eeqs
therefore $\chi_\ell v_\ell \in H^{3/2}(\Omega_\ell; \Delta)$. Hence,  by Theorem \ref{thm:H32}, $\chi_\ell v_\ell \in U(\Omega_\ell)$.

To prove (ii),  observe that, by Assumption \ref{ass:POU},  \eqref{eq:diffprod} and the definition of $U(\Omega_\ell)$, $\partial (\chi_\ell v_\ell)/\partial n_\ell \in L^2(\partial \Omega_\ell)$ and
\beq
\N{\pdiff{\big(\widetilde{\mathcal{R}}_\ell^\top v_\ell\big) }{n} }_{L^2(\partial \Omega)} =
\N{\pdiff{\big(\chi_\ell v_\ell\big)}{n_\ell} }_{L^2(\partial \Omega\cap \partial \Omega_\ell)}  
\leq \N{\pdiff{\big(\chi_\ell v_\ell\big)}{n_\ell} }_{L^2(\partial \Omega_\ell )} <\infty. \label{eq:L2star} 
\eeq
Because $\partial (\widetilde{\mathcal{R}}_\ell^\top v_\ell)/\partial n\in L^2(\partial \Omega)$,
  to finish the proof that  $\tcR^\top_\ell v_\ell \in U(\Omega)$,  we need to show that $\tcR^\top_\ell v_\ell  \in   H^1(\Omega;\Delta)$. 
To do this, recall that, by the definition of the weak derivative and the divergence theorem, a piecewise $H^1$ function is globally $H^1$ if it is continuous across the interface between the pieces.
   Therefore, since $\chi_\ell v_\ell=0$ on $\partial \Omega_\ell$, 
  $\tcR_\ell^\top v_\ell \in H^1(\Omega)$. 
Also, since $\partial (\chi_\ell v_\ell)/\partial n$ also vanishes on  $\partial \Omega_\ell\backslash \partial\Omega$ 
(recall \eqref{eq:diffprod}),  a similar argument shows that the Laplacian of $\tcR_\ell^\top v_\ell$ is in $L^2(\Omega)$. 
\epf

\begin{theorem}\mythmname{The algorithm in \S\ref{sec:parallel_algorithm} is well-defined in both $U(\Omega)$ and  $\bbU$}
  \label{thm:algorithm}
  Suppose  $\Omega$ and $\Omega_\ell, \ \ell = 1,\ldots, N$ are  Lipschitz and let  Assumption \ref{ass:POU} hold.\\
  (i) Suppose $u^n \in U(\Omega)$ and define  $u^{n+1}$ by  \eqref{eq21}-\eqref{star}. Then  
  $u^{n+1}\in U(\Omega)$.\\
  (ii) Suppose $\bu^n \in \bbU$. Define  $u^n$  by \eqref{star} (with $n+1$ replaced by $n$) and then  $\bu^{n+1}$ by \eqref{eq21}-\eqref{eq23}. Then  $\bu^{n+1} \in \bbU$. 
\end{theorem}

\bpf
We prove (i) only; the proof of (ii) is similar.
Part (i) of Lemma \ref{lem:Friday}  implies that\linebreak  $(\partial /\partial n_j - \ri k) u^n\in L^2(\partial\Omega_j)$.
Therefore, by Lemma \ref{lem:IIPU}, $u_j^{n+1}$ (defined by \eqref{eq21}-\eqref{eq23}) is in
$U(\Omega_j)$. 
Then \eqref{star} and Lemma \ref{lem:map} imply that
 $u^{n+1}\in U(\Omega)$.
\epf

\section{Framework for the  convergence analysis }
\label{sec:fixedpoint}

In this section we develop the tools needed to analyse   the algorithm   \eqref{eq21}--\eqref{star} in the space
$\bbU_0$. 
\subsection{The error propagation operator $\cT$}
  To begin, recalling \eqref{eq:urest}, we introduce the error   
 \begin{align} \label{31} \mathbf{e}^n = (e_1^n, e_2^n, \ldots e_N^n)^\top , \quad \text{where} \quad  e_\ell^n := u_\ell - u_\ell^n  = u\vert_{\Omega_\ell}  - u_\ell^n,  \quad \ell = 1, \ldots, N. \end{align}   Then, using    \eqref{star} and \eqref{eq:RR1},  
 the global error $e^n := u - u^n $  can be written 
\begin{align} \label{eq:global} 
e^n =\sum_\ell \chi_\ell   u\vert_{\Omega_\ell }  -\sum_\ell    \chi_\ell u_\ell^n = \sum_\ell \chi_\ell  e_\ell^n . 
\end{align}
Thus, subtracting  \eqref{eq21}-\eqref{eq23} from  \eqref{eq11}-\eqref{eq13}, we obtain 
\begin{align}
  (\Delta + k^2)e_j^{n+1}  & = 0  \quad  \text{in } \  \Omega_j,  \label{eq31}\\
  \left(\frac{\partial }{\partial n_j} - \ri k \right) e_j^{n+1}  & = \left(\frac{\partial }{\partial n_j} - \ri k \right) e^n \ = \ \sum_\ell \left(\frac{\partial }{\partial n_j} - \ri k \right) \chi_\ell e_\ell^n,   \quad  \text{on }  \ \partial \Omega_j\backslash \partial \Omega, \label{eq32}  \\
    \left(\frac{\partial }{\partial n_j} - \ri k \right) e_j^{n+1}  & = 0  \quad  \text{on } \  \partial \Omega_j  \cap \partial \Omega.   \label{eq33} 
\end{align} 
This motivates the introduction of the operator-valued matrix $\bcT = (\cT_{j,\ell})_{j,\ell = 1}^N$, defined as follows.  For $v_\ell \in U(\Omega_\ell)$,  and any $j$, 
\begin{align}
  (\Delta + k^2)(\cT_{j,\ell} v_\ell)  & = 0  \quad  \text{in } \  \Omega_j,  \label{eq41}\\
  \left(\frac{\partial }{\partial n_j} - \ri k \right) (\cT_{j,\ell} v_\ell) & = \left(\frac{\partial }{\partial n_j} - \ri k \right) (\chi_\ell v_\ell) \quad  \text{on }  \ \partial \Omega_j\backslash \partial \Omega, \label{eq42}  \\
    \left(\frac{\partial }{\partial n_j} - \ri k \right) ({\cT_{j,\ell} v_\ell}) & = 0  \quad  \text{on } \  \partial \Omega_j \cap \partial \Omega.   \label{eq43} 
\end{align}
Therefore,
 \begin{align}\label{eq:errit} 
  e_j^{n+1} = \sum_\ell \cT_{j,\ell} e_\ell^n, \quad \text{and so} \quad  \be^{n+1} = \bcT \be^{n}.   
 \end{align}

\begin{remark} \label{rem:null}
     (i) By  Assumption \ref{ass:POU},   $(\partial /\partial n_\ell - \ri k )(\chi_\ell v_\ell)$ vanishes on $\partial \Omega_\ell$,  and so  $\cT_{\ell, \ell} \equiv 0$ for all $\ell$.\\
(ii) If $\Omega_j \cap\Omega_\ell = \emptyset$, then $\cT_{j,\ell}=0$. \\
     (ii) It is convenient here to introduce  the notation
     \begin{align} \Gamma_{j,\ell} = \partial \Omega_j \cap \Omega_\ell,\label{defGamma}\end{align} so that  \eqref{eq42} holds on $\Gamma_{j,\ell}$ and \eqref{eq43} holds on $\partial \Omega_j \backslash \Gamma_{j,\ell}$.   
     \end{remark}

 Since $e_j^n $ is Helmholtz-harmonic in $\Omega_j$ for each $j$,  we aim to  analyse convergence of \eqref{eq:errit} in the space $\bbU_0$ defined in \eqref{eq:product}. 
\ig{For the rest of this section we restrict to 2-d and 3-d, using  the following norm, previously introduced in \cite[Equation 12]{BeDe:97}. (The  1-d case is discussed brielfy  in \S \ref{subsec:1d}, where the norm on the boundary data is trivially the modulus on $\mathbb{C}$.) }

 \begin{definition}[Norm on $U_0(D)$]
For $D$ a bounded Lipschitz domain and $v \in U_0(D)$, let
   \begin{align}\label{eq:pseudo-energy}
\|v\|_{1,k,\partial D}^2 \ := \  
\N{\pdiff{v}{n}}^2_{L^2(\partial D)} + k^2 \N{v}^2_{L^2(\partial D)},
   \end{align}
   where  $\partial/\partial n$ denotes the outward normal derivative on $\partial D$.
\end{definition}

The next lemma guarantees that this is indeed a norm on $U_0(D)$, {with the relation \eqref{eq:norm2} a well-known ``isometry" result about impedance traces; see, e.g., \cite[Lemma 6.37]{Sp:15}, \cite[Equation 3]{ClCoJoPa:21}.}

\begin{lemma}[Equivalent formula for $\Vert\cdot\Vert_{1,k,\partial D}$]\label{lm:orth-trace-helm}
  For all  $v \in U_0(D)$ and $k > 0$, 
  \begin{align} \label{eq:norm2}   \|v\|_{1,k,\partial D}^2 \ = \ 
\N{\pdiff{v}{n}-\ri k v }^2_{L^2(\partial D)}\ =\  \N{-\pdiff{v}{n}-\ri k v }^2_{L^2(\partial D)},
     \end{align} 
and so   $\| \cdot \|_{1,k,\partial D}$ is a norm on $U_0(D)$. Furthermore, if $D$ is either Lipschitz star-shaped or $C^\infty$, then $\| \cdot \|_{1,k,\partial D}$
      is equivalent to $\Vert \cdot \Vert_{U(D)}$,  with  equivalence constants independent of $k$. 
\end{lemma}
\begin{proof}  If $v \in U_0(D)$, then by Green's first identity (see, e.g., \cite[Lemma 4.3]{Mc:00}),
  $$ 0 = - \int_{D} (\Delta v + k^2 v) \overline{v} = \int_{D} \left(\vert \nabla v \vert^2 - k^2 \vert v \vert^2\right) \  -\  \int_{\partial D} \frac{\partial v}{\partial n} \overline{v}. $$
 Taking the imaginary part,  we have 
  \begin{align} \label{eq:norm1}
    \Im  \left( \int_{\partial D} \frac{\partial v}{\partial n}  \overline{ v} \right) = 0.  
  \end{align}
  Thus, 
  $$ \int_{\partial D} \left \vert \pm \frac{\partial v }{\partial n} - \ri k v \right \vert^2 = \int_{\partial D}
  \left(\left\vert \frac{\partial v }{\partial n}\right \vert^2 + k^2 \vert  v  \vert^2\right)    \mp 2k \,   \Im \left( \int_{\partial D}  \frac{\partial v }{\partial n} \overline{ v}\right) \ = \    \|v\|_{1,k,\partial D}^2, $$
  yielding \eqref{eq:norm2}.

To show \eqref{eq:norm2} is a norm, suppose  $\Vert v \Vert_{1,k,\partial D} = 0$. Then, by \eqref{eq:norm2},  $(\partial /\partial n - \ri k) v = 0$ on $\partial D$.
  Since $v\in U_0(D)$,  Lemma \ref{lem:IIPU} ensures that $v= 0$.  The other norm axioms  follow from the definition \eqref{eq:pseudo-energy}.

  To obtain the norm equivalence, observe that, for $v \in U_0(D)$,
  $$ \Vert v \Vert_{1,k,\partial D} = \left\Vert {\partial v }/{\partial n} - \ri k v \right \Vert_{L^2(\partial D)} \leq \left\Vert {\partial v }/{\partial n}\right \Vert_{L^2(\partial D)} + k \Vert v  \Vert_{L^2(\partial D)}  
  \leq \Vert v \Vert_{U(D)}. $$
  Moreover since $\partial v /\partial n$ and $v$ both belong to $L^2(\partial D)$,  Lemma \ref{lem:IIPU} implies that
  \begin{align*} 
    \Vert v \Vert_{U(D)} \ \leq \ {C}_2(k)\N{\partial v/\partial v - \ri k v }_{L^2(\partial D)} = C_2(k)
    \Vert v \Vert_{1,k,\partial D}.
  \end{align*}
The stated $k-$independence follows from  Remark \ref{rem:kdep}. 
   \end{proof}

 Using \eqref{eq:pseudo-energy}, we define the norm on  $\bbU_0$:   
   \begin{align}\label{eq:pseudo-energy1}
\|\bv\|_{1,k,\partial }^2 \ := \sum_{\ell=1}^N \|v_\ell \|_{1,k,\partial \Omega_\ell}^2 \quad \text{for} \quad \bv \in \bbU_0.
\end{align}

For simplicity, we now assume that each $\Omega_\ell$ is star-shaped Lipschitz, so that 
the norm equivalence in Lemma \ref{lm:orth-trace-helm} holds with constants independent of $k$.
 Analogues of the following results for general Lipschitz $\Omega_\ell$ hold, but with different $k$-dependence.

\begin{assumption}\label{ass:star}
 $\Omega_\ell$ is star-shaped Lipschitz.
\end{assumption}
Furthermore, to simplify the notation, we define the operator
     $$
      \imp_\ell \ := \ \left(\frac{\partial}{\partial n_\ell} - \ri k \right).
     $$
 The next theorem summarises some basic properties of the  operator $\cT_{j,\ell}$ on the space $U_0(\Omega_\ell)$.    
 \begin{theorem}[Properties of $\cT$]\label{main_T}   
If $v_\ell \in U_0(\Omega_\ell)$ then $    \imp_j (\cT_{j,\ell} v_\ell) $ \iggg{vanishes on  $\partial \Omega_j \backslash \Gamma_{j,\ell}$,  and} 
   \begin{align}\label{main_T1a}
    \imp_j (\cT_{j,\ell} v_\ell) \ = \ \chi_\ell \, \imp_j (v_\ell) \ +\  \pdiff{\chi_\ell}{n_j} v_\ell \quad \text{ on} \quad   \Gamma_{j,\ell}.  \ \end{align}
Also, 
   \begin{align} \label{main_T2}  \Vert \cT_{j,\ell} v_\ell \Vert_{1, k, \partial \Omega_j} \ & = \
     \Vert \imp_j (\cT_{j,\ell} v_\ell) \Vert_{L^2(\Gamma_{j,\ell})}  \nonumber \\ & \leq \  \left\Vert \imp_j (v_\ell)  \right\Vert_{L^2(\Gamma_{j,\ell})} \ + \ 
     k^{-1/2}  \Vert \nabla \chi_\ell \Vert_{L^\infty(\Gamma_{j,\ell})}\Vert v_\ell \Vert_{1, k, \partial \Omega_\ell} ,
     \end{align}
and $\cT_{j,\ell}: U_0(\Omega_\ell) \rightarrow U_0(\Omega_j)$ is a bounded operator.
\end{theorem}

\begin{proof}   By  its definition \eqref{eq41}--\eqref{eq43},
  $\cT_{j,\ell} v_\ell \in U_0(\Omega_j)$ \iggg{and, on $\partial \Omega_j$,} 
  \begin{align} \label{eq:partial1}   \imp_j (\cT_{j,\ell} v_\ell ) 
  \ =\ \imp_j (\chi_\ell v_\ell )
   \  =\  \left(\frac{\partial} { \partial n_j}  - \ri k \right) (\chi_\ell v_\ell ),\end{align}    % \\
which, recalling \eqref{defGamma},   vanishes on   $\partial \Omega_j \backslash \Gamma_{j,\ell}$.  
 Differentiating the product on the right-hand side of  \eqref{eq:partial1}  yields \eqref{main_T1a}.
  Then, taking norms of both sides of \eqref{main_T1a} and using Assumption \ref{ass:POU} and the fact that $0 \leq \chi_\ell \leq 1$, 
we obtain
  \begin{align}\label{main_T4} \Vert \imp_j (\cT_{j,\ell} v_\ell)\Vert_{L^2(\Gamma_{j,\ell}) }  \ \leq  \ \Vert \imp_j (v_\ell)
    \Vert_{L^2(\Gamma_{j,\ell})} \ +\  {\Vert \nabla \chi_\ell \Vert_{L^\infty(\Gamma_{j,\ell})}} \Vert v_\ell \Vert_{L^2(\Gamma_{j,\ell})}.   \ \end{align}
  Then, using the fact that $\Gamma_{j,\ell}\subset \Omega_\ell$ \ig{is an interface in $\Omega_\ell$},
using the multiplicative trace theorem and then Lemma \ref{lm:orth-trace-helm}, we obtain 
 \begin{align} \nonumber  k^{1/2} \Vert v_\ell \Vert_{L^2(\Gamma_{j,\ell})} \ & \lesssim \  
                k^{1/2} \Vert v_\ell \Vert_{L^2(\Omega_\ell)}^{1/2}\Vert v_\ell\Vert_{H^1(\Omega_\ell)}^{1/2} \\  & \lesssim\ 
                k  \Vert v_\ell \Vert_{L^2(\Omega_\ell)} + \Vert v_\ell \Vert_{H^1(\Omega_\ell)}
                \lesssim  \Vert v_\ell \Vert_{1,k,\Omega_\ell}  \lesssim  \Vert v_\ell \Vert_{1,k,\partial \Omega_\ell} \ . \end{align}
  Combining this with \eqref{main_T4} yields \eqref{main_T2}.
Finally, to obtain the boundedness of $\cT_{j,\ell}: U_0(\Omega_\ell) \to U_0(\Omega_j)$, we use Lemma \ref{lem:Friday} (ii),  to obtain
  $$ \Vert \imp_j (v_\ell) \Vert_{L^2(\Gamma_{j,\ell})} 
  \ {\lesssim \ k^2 \Vert \imp_\ell (v_\ell)\Vert_{L^2(\partial \Omega_\ell)}}
  \ {=}\  k^2   \Vert v_\ell \Vert_{1,k,\partial \Omega_\ell},$$
  and we then combine this with  \eqref{main_T2}.
\end{proof}

\begin{remark}  {The same arguments}  show that $\cT_{j,\ell}: U(\Omega_\ell) \rightarrow U_0(\Omega_j)$ is bounded. \end{remark} 

In the following section we are interested in proving power contractivity of the error propagation operator $\cT$. 
This motivates us to study the composition $\cT_{j,\ell} \cT_{\ell, j'}$; indeed,
\begin{align} \label{eq:square} (\cT^2)_{j,j'} \ = \ \sum_\ell \cT_{j,\ell} \cT_{\ell, j'},  \end{align}
where the sum is over all $\ell \in  \{1, 2, \ldots , N \} \backslash \{ j,j'\}$, with  $\Gamma_{j,\ell} \not = \emptyset \not= \Gamma_{\ell,j'}$. \iggg{A useful expression  for the action of   \eqref{eq:square} can be obtained by  inserting $v_\ell = \cT_{\ell,j'} z_{j'}$,  with $z_{j'} \in U(\Omega_{j'})$,  into \eqref{main_T1a},  to obtain}      
   \begin{align}\label{main_T1}
    \imp_j (\cT_{j,\ell} \cT_{\ell,j'} z_{j'}) \ = \ \chi_\ell \, \imp_j (\cT_{\ell,j'} z_{j'}) \ +\  \left(\pdiff{\chi_\ell}{n_j}\right) \, \left(\cT_{\ell,j'} z_{j'}\right) \quad \text{ on} \quad   \Gamma_{j,\ell}.
    \end{align}
The first term on the right-hand side of \eqref{main_T1} is of key interest in this paper. We note that its value is obtained by (i) finding $\cT_{\ell, j'} {z_{j'}}$, i.e.,  the unique function  in $U_0(\Omega_\ell)$ with impedance data on $\Gamma_{\ell, j'}$ given by $\imp_\ell(\chi_{j'} z_{j'})$ ;  (ii) evaluating $\imp_j(\cT_{\ell,j'}z_{j'})$    on $\Gamma_{j, \ell}$  and (iii) then multiplying the result by $\chi_\ell$. Combining  steps (i) and (ii) leads us to the following key definition.
    
      \subsection{ The impedance-to-impedance map}
     \label{subsec:relation}

\begin{definition}[Impedance map] \label{def:impmap}
Let $\ell, j, j' \in \{ 1, \ldots, N\}$
  be such that $\Gamma_{\ell, j'}  \not =\emptyset$  and $\Gamma_{j,\ell}  \not =\emptyset  $ (or, equivalently, $\Omega_\ell\cap\Omega_{j'} \neq \emptyset$ and $\Omega_\ell\cap\Omega_{j} \neq \emptyset$).
    Given $g \in L^2(\Gamma_{\ell,j'})$, let $v_\ell$ be the unique element of
    $U_0(\Omega_\ell)$ with impedance data 
\begin{align}  
   \imp_\ell (v_\ell) & = \left\{ \begin{array}{ll} g & \text{on} \quad \Gamma_{\ell,j'} \\  0 & 
  \text{on} \quad \partial \Omega_\ell \backslash \Gamma_{\ell,j'}  \end{array} \right. . \label{H2} \end{align}
Then the impedance-to-impedance map $ \Imap{\ell}{j'}{j}{\ell}: L^2(\Gamma_{\ell,j'})
    \rightarrow L^2(\Gamma_{j,\ell})$ is defined by 
\begin{align}   \Imap{\ell}{j'}{j}{\ell} g := \imp_j(v_\ell), \quad \text{on} \quad \Gamma_{j,\ell} \, ,     \label{H3}
 \end{align}   
i.e., $\Imap{\ell}{j'}{j}{\ell} g $ is the impedance data on $\Gamma_{j,\ell} = \partial \Omega_j \cap \Omega_\ell$ 
of the  Helmholtz-harmonic function  on $\Omega_\ell$   with given impedance data \eqref{H2}.
  \end{definition}
  \begin{figure}[H]
    \centering
    \begin{subfigure}[t]{0.33\textwidth}
        \centering
      \includegraphics[width=1.1\textwidth]{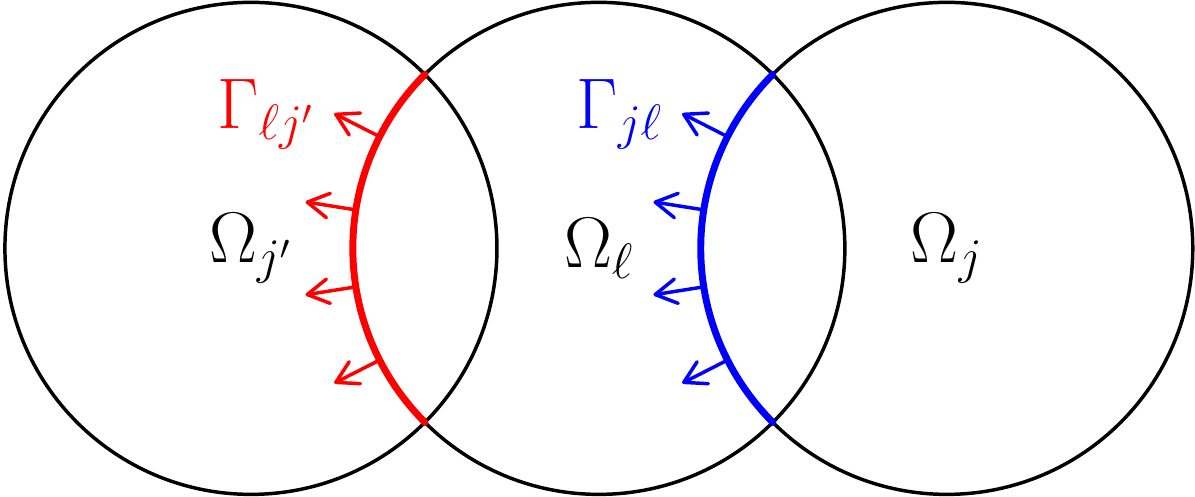}
        \caption{$\Omega_j \cap \Omega_{j'} = \emptyset$  }
         \end{subfigure}%
     \hfill
    \begin{subfigure}[t]{0.33\textwidth}
        \centering
        \includegraphics[width=.75\textwidth]{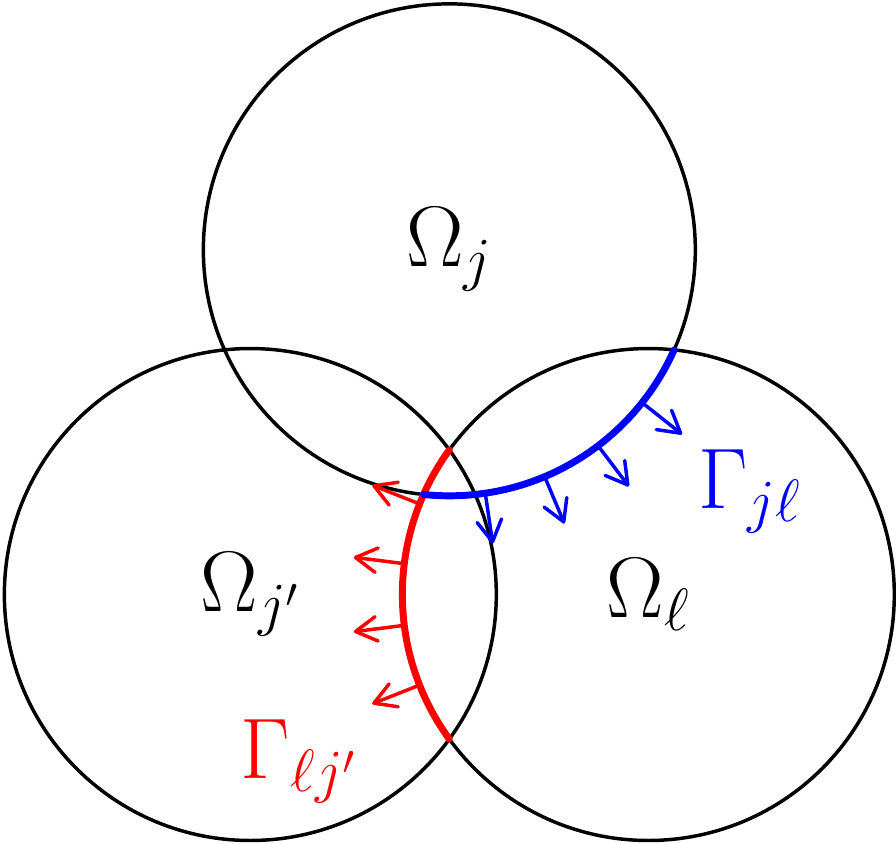}
        \caption{$\Omega_j \cap \Omega_{j'} \not = \emptyset$}
        \end{subfigure}%
           \begin{subfigure}[t]{0.33\textwidth}
        \centering
        \includegraphics[width=.95\textwidth]{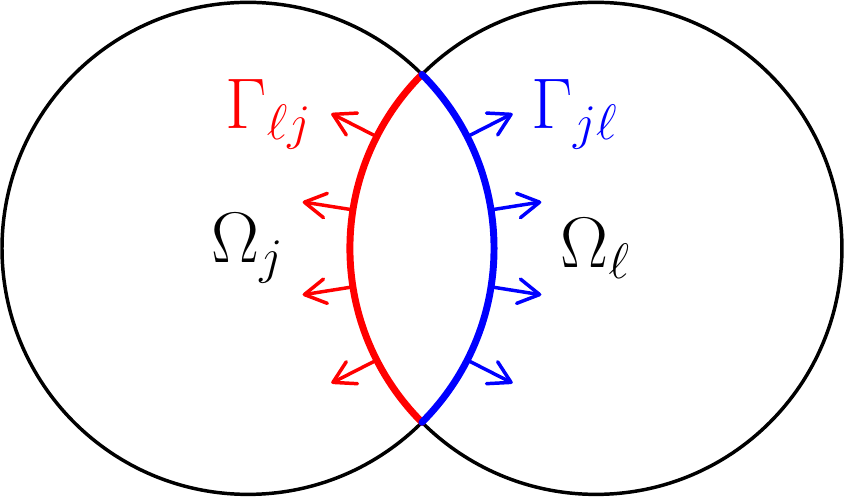}
        \caption{$\Omega_j = \Omega_{j'} $}
         \end{subfigure}%
           \caption{Illustrations of the domain (red) and co-domain (blue) of $\Imap{\ell}{j'}{j}{\ell}$   \ig{in 2d}}\label{fig:illustrations}
\end{figure}

\noindent Illustrations of the domain (in red) and co-domain (in blue) of the impedance-to-impedance map,  indicating the direction of the normal derivative,  are given in Fig.  \ref{fig:illustrations}. The next lemma shows its  $L^2-$boundedness. 
  \ble \label{lem:Ibounded}
$\Imap{\ell}{j'}{j}{\ell} : L^2(\Gamma_{\ell,j'})\rightarrow L^2(\Gamma_{j,\ell}) $ is bounded. 
\ele

\noindent \emph{Proof.} 
By \eqref{H3}, Lemma \ref{lem:Friday} (ii), Assumption \ref{ass:star}, and \eqref{H2},
$$
 \Vert  \Imap{\ell}{j'}{j}{\ell} g\Vert_{L^2(\Gamma_{j,\ell})} \ = \ \left\Vert \imp_j(v_\ell) \right\Vert_{L^2(\Gamma_{j,\ell})}  \lesssim  k^2  \left\Vert \imp_\ell( v_\ell) \right\Vert_{L^2(\partial \Omega_\ell )} = k^2 \Vert g\Vert_{L^2(\Gamma_{\ell,j'})}.   \quad \quad \quad \quad \quad \qed
$$

Although the proof of Lemma \ref{lem:Ibounded} gives a $k$-explicit bound on $\|\Imap{\ell}{j'}{j}{\ell}\|$, 
we obtain sharper $k$-explicit  information in certain set-ups later.
We now rewrite \eqref{main_T1} using this map.  

\begin{theorem}[Connection between $\cT^2$ and impedance-to-impedance map] \label{thm:main_T2} 
Let $\ell, j, j' \in \{ 1, \ldots, N\}$
  be such that $\Gamma_{\ell, j'}  \not =\emptyset$  and $\Gamma_{j,\ell}  \not =\emptyset  $ (or, equivalently, $\Omega_\ell\cap\Omega_{j'} \neq \emptyset$ and $\Omega_\ell\cap\Omega_{j} \neq \emptyset$).
If $z_{j'} \in U(\Omega_{j'})$, then
  \begin{align}
  \imp_j (\cT_{j,\ell} \cT_{\ell,j'} z_{j'} ) = \chi_\ell \, \Imap{\ell}{j'}{j}{\ell}   \left(\imp_\ell (\cT_{\ell,j'} z_{j'})\right) \ + \ \left(\pdiff{\chi_\ell}{n_j}\right)  \, \left(\cT_{\ell,j'} z_{j'} \right) \quad\ton \Gamma_{j,\ell}
\label{eq:impit}    \end{align}
  \begin{align} \label{cor_T21} \text{and} \quad \quad \Vert \cT_{j,\ell} \cT_{\ell, j'} z_{j'} \Vert_{1, k, \partial \Omega_j} \ \leq \ \left(  \Vert \Imap{\ell}{j'}{j}{\ell}   \Vert + 
     k^{-1/2} \Vert \nabla \chi_\ell \Vert_{L^\infty(\Gamma_{j,\ell})}\right) \Vert \cT_{\ell,j'} z_{j'}  \Vert_{1, k, \partial  \Omega_\ell} .
  \end{align}
\end{theorem}
\begin{proof}
  Since $\cT_{\ell, j'} z_{j'} \in U_0(\Omega_{\ell})$ and $\imp_\ell (\cT_{\ell,j'} z_{j'} ) $ vanishes on $\partial \Omega_\ell \backslash \Gamma_{\ell,j'}$,  we have
  $$ \chi_\ell \, \imp_j (\cT_{\ell, j'} z_{j'} ) = \ig{\chi_\ell}\,  \cI_{\Gamma_{\ell, j'} \rightarrow \Gamma_{j,\ell}} (\imp_\ell (\cT_{\ell,j'} z_{j'} ) ). $$ \ig{Substituting this in \eqref{main_T1}   gives \eqref{eq:impit}. The estimate \eqref{cor_T21} is obtained by following the proof of  \eqref{main_T2}, with $v_\ell = \cT_{\ell, j'} z_{j'}$.}  
 \end{proof} 
We now see from \eqref{cor_T21} that (at least for sufficiently large $k$ and/or sufficiently small $\Vert \nabla \chi_\ell \Vert_{L^\infty(\Gamma_{j,\ell})}$) 
the right-hand side of \eqref{eq:impit} is dominated by the first term. The norm of the impedance-to-impedance map  lies at the heart of the convergence theory in \S \ref{sec:conv}.  
 
     \section{Convergence  of the iterative method for strip decompositions}
\label{sec:conv}

In this section we obtain a convergence theory for the iterative method \eqref{eq21} -- \eqref{star} when the domain $\Omega$ is either an interval in 1-d or a rectangle in 2-d. In 1-d the subdomains are intervals and in 2-d the subdomains are sub-rectangles. 

\subsection{Notation common to both 1-d and 2-d}\label{sec:notation}

\begin{notation}[Strip decompositions in 1- and 2-d]\label{def:geom}
    In 1-d the subdomains are denoted by $\Omega_\ell = (\Gamma_\ell^-, \Gamma_\ell^+)$. In 2-d we assume the domain $\Omega$ is a rectangle of height $H$ and the  subdomains $\Omega_\ell$ also have height $H$ and are  bounded by vertical sides denoted $\Gamma_\ell^-$, $\Gamma_\ell^+$.
    In both 1-d and 2-d we  assume each  $\Omega_\ell$ is only overlapped by $\Omega_{\ell-1}$ and $\Omega_{\ell+1}$ (with   $\Omega_{-1} := \emptyset$ and $\Omega_{N+1}  :=\emptyset$).
    The   width of $\Omega_\ell$ is denoted   $L_\ell$.
    This  notation is illustrated  in Figures \ref{fig:overlap1d} and \ref{fig:overlap}.
\end{notation}

\begin{remark}\label{rem:specialPOU}      (i) The simpler notation in this section is linked to the general notation  \eqref{defGamma} via 
 \begin{align} \label{notstrip} \Gamma_\ell^-  = \Gamma_{\ell, \ell-1}  \quad \text{and}\quad \Gamma_\ell^+ = \Gamma_{\ell, \ell+1} .\end{align} 
  (ii) Under this set-up,  any  partititon of unity $\{\chi_\ell\}$
        defined in \eqref{POUstar} satisfies
        \begin{align} \label{eq:important} 
          \chi_{\ell}\vert_{\Gamma_{\ell-1}^+} = 1 = \chi_\ell\vert_{\Gamma_{\ell+1}^-} .  \end{align}
      \end{remark}

\begin{figure}[H]
  \begin{center}
    \includegraphics{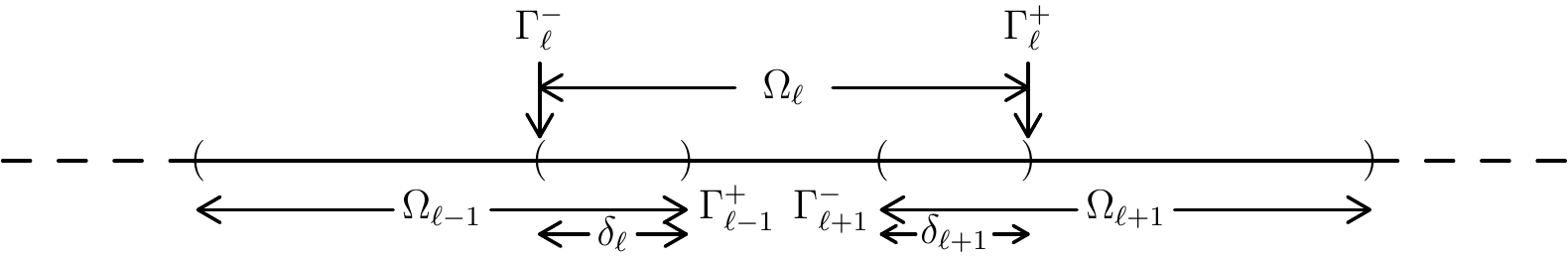}
\end{center}   \caption{Three overlapping subdomains in 1-d \label{fig:overlap1d}}
\end{figure}
\begin{figure}[H]
  \begin{center}
  \includegraphics{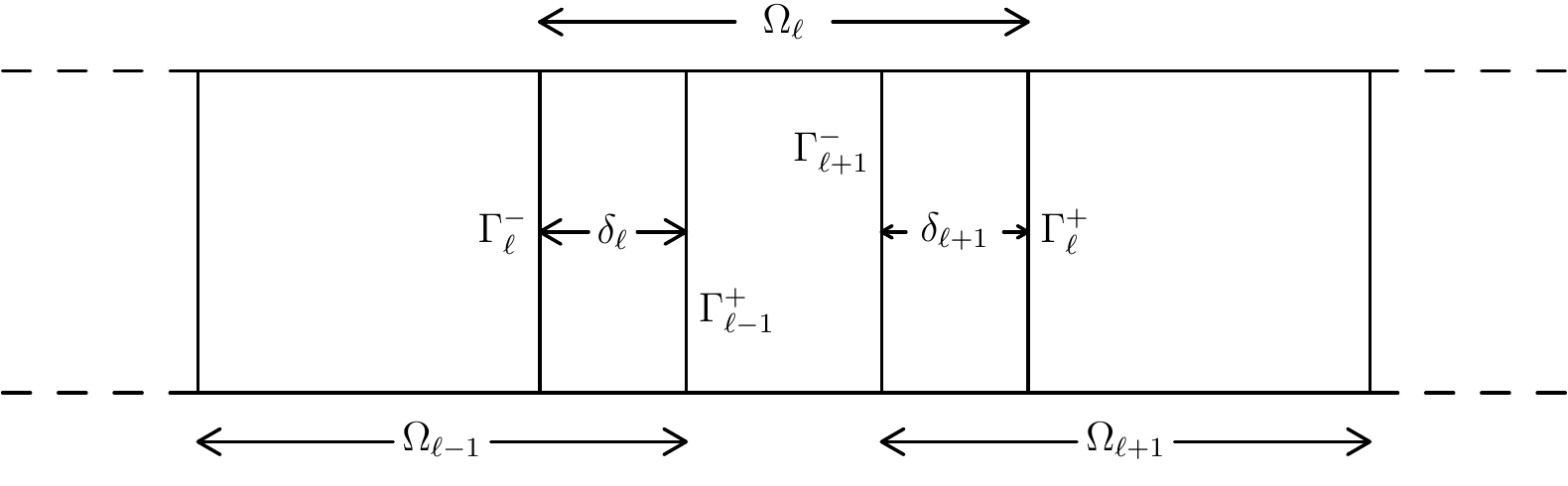}
\end{center}   \caption{Three overlapping subdomains in 2-d \label{fig:overlap}}
\end{figure}
To illustrate these definitions, 
    given $g \in L^2(\Gamma_\ell^-)$,  let     $u$ denote  the Helmholtz-harmonic function on $\Omega_\ell$ with left-facing impedance data $g$ on $\Gamma_\ell^-$ and zero impedance data, elsewhere. Then         $ \Imaps{\ell}{-}{\ell-1}{+}g $ is the right-facing impedance data of $u$  on $\Gamma_{\ell-1}^+$
    and      $ \Imaps{\ell}{-}{\ell+1}{-}g $  is the left-facing impedance data of $u$ on  $\Gamma_{\ell+1}^-$. Note that $\Gamma_{\ell-1}^+$ and $\Gamma_{\ell+1}^{-}$ are both interior interfaces in $\Omega_\ell$ (see Figure  \ref{fig:overlap}). 
     
Recall from Remark \ref{rem:null} that
$\cT_{\ell,\ell} = 0$ and 
$  \cT_{j,\ell} = 0$ {if } $\Omega_j \cap \Omega_\ell =\emptyset.
$
Therefore, $\bcT$ takes the block tridiagonal  form

\begin{equation}\label{eq:splitT}
\bcT =\begin{pmatrix}
0		&	\cT_{1,2} \\
\cT_{2,1}	&	0		&\cT_{2,3}\\
		&	\cT_{3,2}	&0		&\cT_{3,4}\\
		&			&	\ddots	&	\ddots&	\ddots\\
		&			&	\cT_{N-1,N-2}&		0	&\cT_{N-1,N}\\
		&			&			&	\cT_{N,N-1}&0	
\end{pmatrix} :  = \bcL + \bcU\, , 
\end{equation}
where $\bcL$ and $\bcU$ are the lower and upper triangular components of $\bcT$.
We record for later that, by the Cayley-Hamilton theorem, 
  \begin{align} \label{eq:nilpotent} \bcL^N \ = \ \mathbf{0} \ = \ \bcU^N. \end{align}
In what follows a crucial role is played by the products:
\begin{align} 
\resizebox{0.9\hsize}{!}{$
\bcL\bcU = \begin{pmatrix}
0 \\
&	\cT_{2,1}\cT_{1,2}\\
&	&	&\ddots\\
&	&	&	&\cT_{N-1,N-2}\cT_{N-2,N-1}\\
&	&	&	&	&\cT_{N,N-1}\cT_{N-1,N}
\end{pmatrix}
~~ \text{and} ~~~ 
\bcU\bcL = \begin{pmatrix}
\cT_{1,2}\cT_{2,1} \\
&	\cT_{2,3}\cT_{3,2}\\
&	&	&\ddots\\
&	&	&	&\cT_{N-1,N}\cT_{N,N-1}\\
&	&	&	&	&0
\end{pmatrix}
$}.
\label{products} 
\end{align}
 We remark  that in 2-d the structure \eqref{eq:splitT} remains the same
     if the vertical interfaces in $\Gamma_\ell^{\pm}$ are replaced by non-intersecting polygonal pieces; however, we do not pursue this generalisation here.    The  diagonal entries in \eqref{products} can be estimated in terms of impedance-to-impedance maps - see 
  Theorem \ref{thm:main_T2}.

\subsection{The results of \S\ref{sec:fixedpoint} specialised to strip decompositions}

Since strip decompositions have the property \eqref{eq:important}, Theorem \ref{thm:main_T2} simplifies to the following.

\begin{corollary}\label{cor:main_T2strip} 
Let $\ell \in \{ 1, \ldots, N\}$ and $j,j' \in \{\ell-1,\ell+1\}$. 
If $z_{j'} \in U(\Omega_{j'})$, then
  \begin{align}
\imp_j (\cT_{j,\ell} \cT_{\ell,j'} z_{j'} ) = \Imap{\ell}{j'}{j}{\ell} \,    \left(\imp_\ell (\cT_{\ell,j'} z_{j'})\right)  \quad\ton \Gamma_{j,\ell}.
  \label{eq:impit2new}    
\end{align}
  \begin{align} \label{cor_T21new} \text{and thus} \quad \quad \Vert \cT_{j,\ell} \cT_{\ell, j'} z_{j'} \Vert_{1, k, \partial \Omega_j} \ \leq \ 
   \Vert \Imap{\ell}{j'}{j}{\ell}   \Vert  \Vert \cT_{\ell,j'} z_{j'}  \Vert_{1, k, \partial  \Omega_\ell} .
  \end{align}
\end{corollary}

\bpf
To obtain \eqref{eq:impit2new}, without loss of generality,  consider  the case $j= \ell-1 = j'$. Then, for any $v_\ell \in U_0(\Omega_\ell)$,  \begin{align} \label{Martin}  \imp_{\ell-1} (\cT_{\ell-1,\ell} \, v_\ell) = \imp_{\ell-1} (\chi_\ell v_\ell) = \chi_\ell \imp_{\ell-1} (v_\ell)  = \imp_{\ell-1}(v_\ell) \quad \text{on} \quad \Gamma_{\ell-1,\ell} .
\end{align}
In \eqref{Martin}, the first equality comes from the definition of $\cT_{\ell-1,\ell}$, the second comes from the fact that (by Assumption \ref{ass:POU}),  $(\partial \chi_\ell / \partial n_{\ell-1}) = 0 $ on $\Gamma_\ell^+$  and   the final equality comes from the fact  (see \eqref{notstrip} and \eqref{eq:important}) that $\chi_\ell \equiv 1$ on $ \Gamma_{\ell-1,\ell} =
  \partial \Omega_{\ell-1} \cap \Omega_\ell = \Gamma_{\ell-1}^+ $.   Using this instead of  \eqref{main_T1a}  and propagating this simplification through the arguments using to prove Theorems \ref{main_T} and  \ref{thm:main_T2} gives the result.
\epf

\subsection{One dimension}\label{subsec:1d} 

The following result is known from \cite[Propositions 2.5 and 2.6]{NaRoSt:94} \ig{(restricted to 1-d)}, 
but we state it here in our notation, because it helps motivate our approach in the 2-d case.  

\begin{lemma}\label{lm:TtoI1D}
In 1-d,
\begin{equation}
\label{1d1} \Imaps{\ell}{-}{\ell-1}{+}=\Imaps{\ell}{+}{\ell+1}{-}  = 0,
\end{equation}
and
\ig{
\begin{equation}\label{1d2} 
  \vert\Imaps{\ell}{+}{\ell-1}{+} g \vert = \vert g \vert  \quad \text{and} \quad \vert
  \Imaps{\ell}{-}{\ell+1}{-} g  \vert   = \vert g \vert, \quad \text{for all} \quad  g \in \mathbb{C}.    
\end{equation}
}
Moreover
\begin{align}\label{1d3}
  \bcU \bcL = \boldsymbol{0} = \bcL \bcU. \end{align} \end{lemma}

\begin{proof}
These results are obtained from the explicit expression for the solution of the Helmholtz interior impedance problem in 1-d.
We consider only  $\Imaps{\ell}{-}{\ell-1}{+}$ and $\Imaps{\ell}{-}{\ell+1}{-}$;
the proofs for $\Imaps{\ell}{+}{\ell+1}{-}$ and $\Imaps{\ell}{+}{\ell-1}{+}$ are similar.

By Definition \ref{def:impmap}, the maps $\Imaps{\ell}{-}{\ell-1}{+}$ and $\Imaps{\ell}{-}{\ell+1}{-}$ can be written in terms of the solution of the following boundary value problem
\begin{align}
  v_\ell'' + k^2v_\ell & = 0  \quad  \text{in } \  \Omega_{\ell }, \label{eq51} \\
   -v_\ell'- \ri k v_\ell& = g \quad  \text{at }  \ {x=\Gamma_{\ell}^-}, \label{eq52} \\
 v_\ell' - \ri k v_\ell& = 0 \quad  \text{at }  \ {x=\Gamma_{\ell}^+} \label{eq53} ,  
\end{align}
for $g\in \mathbb{C}$.
The solution of \eqref{eq51}-\eqref{eq53} is 
\begin{align}\label{eq:exact}
 v_\ell(x) \  = \    \frac{\ri g}{2 k} \re^{\ri k (x-{\Gamma_{\ell}^-})}. 
 \end{align} 
Since 
\begin{align} \label{crucial} (v_\ell' - \ri k v_\ell)(x) = 0 \quad \text{and} \quad (-v_\ell'-  \ri k v_\ell)(x)  = g \re^{\ri k (x - \Gamma_\ell^-)} \quad \text{ for all }  x \in \Omega_\ell,   \end{align} 
 it follows immediately that $\Imaps{\ell}{-}{\ell-1}{+}g = 0$ and $\Imaps{\ell}{-}{\ell+1}{-}g = 
{ \re^{\ri k (\Gamma_{\ell+1}^- -{\Gamma_{\ell}^-})}}  g$. 
Then  \eqref{1d3} follows from using \eqref{1d1} in \eqref{products}, together with 
\eqref{cor_T21new}.
\end{proof}

\begin{proposition} \label{prop:1d}  
$$ \text{(i)} \quad   \bcT^n  \ =\ \bcL^n +\bcU^n, \quad \text{for all} \quad n \geq 1  
\quad \text{and} \quad \text{(ii)} \quad  
 \bcT^N  = 0 . $$
\end{proposition}
\begin{proof}
Part   (i) is proved by induction, starting from \eqref{eq:splitT} and using   \eqref{1d3}. Part (ii) uses part (i) with $n = N$ and 
\eqref{eq:nilpotent}.
\end{proof}

\subsection{Two dimensions}

In the rest of this section  our goal is to estimate $\bcT^n$, where $\bcT = \bcL+\bcU$ is given by \eqref{eq:splitT}. 
In 1-d, $\bcT^n$ takes the simple form given in Proposition \ref{prop:1d}, however this is not the case in 2-d. Our bounds in 2-d on $\bcT^n$ are therefore based on
the following elementary algebraic result.
  
 \subsubsection{An elementary algebraic result and its consequences}
\label{subsec:elem}
 
For integers $n \geq 1$ and $0 \leq j \leq n-1$, let $\cP(n,j)$ denote the set of monomials of order $n$ in the two variables $x,y$ that take  the form
\begin{align} \label{eq:monomx} p(x,y) & = x^{s_0} y^{s_1} x^{s_2} \ldots x^{s_j} 
\quad\text{($j$ even)}
\quad \text{or} \quad x^{s_0} y^{s_1} x^{s_2} \ldots y^{s_j}
\quad\text{($j$ odd)}
\\
  \label{eq:monomy}
  \text{or} \quad p(x,y) & =   y^{s_0} x^{s_1} y^{s_2} \ldots x^{s_j} \quad\text{($j$ odd)}
\quad \text{or}\quad  y^{s_0} x^{s_1} y^{s_2} \ldots y^{s_j} \quad\text{($j$ even)}, \end{align} 
with $1\leq s_\ell \leq n$ for all $\ell = 0, \ldots j$ and
 $s_0 + s_1  + \ldots + s_j = n$.  The terms in \eqref{eq:monomx}, \eqref{eq:monomy}  are monomials of order $n$ with $j$ transitions between the variables $x$ and $y$. 
Since we consider below operators $p(\cL, \cU)$ where $\cL$ and $\cU$ do not, in general, commute, all four of the expressions  in  
\eqref{eq:monomx}, \eqref{eq:monomy}  are considered to be distinct.
The proof of the following proposition is given in the appendix.     
\begin{proposition} \label{prop:sumsum}
  For all $n \geq 1$,  
  \begin{align} \label{eq:prod} (x+ y)^n \ = \ \sum_{j=0}^{n-1} \sum_{p \in \cP(n,j)} p(x,y) .\end{align} 
Moreover, for $0 \leq j \leq n-1$,
  \begin{align} \label{eq:count}
   \# \cP(n,j) \ := \ \text{cardinality of}  \  \cP(n,j) \ = \    \ 2 \left( \begin{array}{c} n-1 \\ j \end{array}\right). 
  \end{align} 
\end{proposition}

\begin{theorem}[General formula for $\cT^n$] \label{thm:sumsumLU}
  For all $n \geq 1$,
  \begin{align} \label{eq:dagger}  \cT^n \ = \ \sum_{j=0}^{n-1} \sum_{p \in \cP(n,j)} p (\cL, \cU) , \end{align}
\ig{ and the $j = 0$ term in \eqref{eq:dagger} vanishes when $n \geq N$.} 
      \end{theorem} 
      \begin{proof}  The formula \eqref{eq:dagger} follows directly  from Proposition \ref{prop:sumsum}. To obtain the final statement, note that $\cP(n,0) = \{ x^n, y^n\}$, so by \eqref{eq:nilpotent}, when $n \geq N$, $p(\cL, \cU) = 0$, for $p \in  \cP(n,0)$. 
        
\end{proof}      
\begin{corollary} [Estimate for $\cT^n$ in terms of composite maps] \label{cor:composite}
  \ig{Suppose {$n\geq N$}.} Then, 
  \begin{align} \label{eq:composite} \Vert \cT^n\bv\Vert_{1,k,\partial} \leq 2 \left(\sum_{j=1}^{n-1} \left( \begin{array}{c} n-1\\j \end{array} \right) \max_{p\in \cP(n,j)} \Vert p(\cL, \cU) \Vert_{1,k,\partial}
  \right) \, \Vert \bv \Vert_{1,k,\partial},  \quad \text{for any} \quad  \bv \in \bbU_0.
\end{align}  \end{corollary}

\subsubsection{The impedance-to-impedance map on a canonical domain}
\label{subsec:canonical}

The properties \eqref{1d1}, \eqref{1d2} of the impedance-to-impedance map in 1-d can be understood via the fact that in 1-d 
the exact solution to the Helmholtz equation is given by \eqref{eq:exact} and thus 
  the action of the Dirichlet-to-Neumann map for the Helmholtz problem on a domain exterior to an interval
  is multiplication by $\ri k$.
  Multiplication by $\ri k$ no longer has this property in higher dimensions, 
  but we  see that,  under certain conditions, the relations \eqref{1d1} and \eqref{1d2}  still hold `approximately';
in the sense that $ \| \Imaps{\ell}{-}{\ell-1}{+} \|$ and $\| \Imaps{\ell}{+}{\ell+1}{-}\|$ can be small, with 
$ \|\Imaps{\ell}{+}{\ell-1}{+} \| \approx 1 $ and $\|\Imaps{\ell}{-}{\ell+1}{-} \| \approx  1$.
  We use these properties to prove a 2-d analogue of Proposition \ref{prop:1d} (ii), namely conditions under which $\cT^N$ is a contraction.  To obtain these properties, we first introduce a canonical domain on which the 2-d impedance-to-impedance maps can be studied.

The impedance-to-impedance map in the geometry Fig. \ref{fig:overlap} can be expressed in terms of a Helmholtz-harmonic solution on a  `canonical'  domain  
  ${\wOmega} =  [0,\hL] \times [0,1]$, 
  via an affine transformation $x \rightarrow x/H$.

 \begin{definition}[Canonical impedance-to-impedance map] \label{def:imp_unit}   For $\hL > 0$, let ${\wOmega} =  [0,\hL] \times [0,1]$ with boundary $\pwOmega$ and let  $\hGamma^{-}, \hGamma^+$ denote, respectively,  the left and right vertical boundaries of $\wOmega$.
       For any $\hdelta \in (0,\hL)$,  let 
 $$
 \hGamma_{\hdelta} := \{(\hdelta, y): y \in [0,1]\}, 
 $$
 i.e., $\hGamma_{\hdelta}$ is an  interior interface; see Figure \ref{fig:hat}.
  Let  ${u}$ be the solution to
\begin{equation}\label{eq:plane-wave-scattering1}\left.
  \begin{aligned}
  \Delta {u} + {k}^2 {u} &=0 \quad \text{ in} \quad {\wOmega},\\
     \frac{\partial {u}}{\partial n} - \ri {k} {u} &= {{g}}\quad \text{on}\quad  {\Gamma}^-,\\
    \frac{\partial {u}}{\partial n}  - \ri {k} {u} &=0\quad \text{on}\quad  \pwOmega\backslash {\Gamma}^-. \\
  \end{aligned}\right\}
  \end{equation}
  Then define the canonical left-to-right and left-to-left impedance-to-impedance maps by
  \begin{align} \label{eq:maps} 
  \rI_{-+} g:= \partial_x u - \ri k u , \quad  \rI_{--} g := -\partial_x u - \ri k u \quad \text{on} \quad \Gamma_\delta , \end{align} 
and define the following norms of these maps
\begin{equation}\label{def:rho_gamma}
 {\rho}({k},{\delta}, L) = \sup_{{g}\in L^2({\Gamma}^-)}\frac{\|\rI_{-+} g  \|_{L^{2}({\Gamma}_{\hdelta})} }
 {   \|{g}\|_{L^2({\Gamma}^{-})}},
 \quad\quad
{\gamma}({k},{\delta}, L)  = \sup_{{g}\in L^2({\Gamma}^-)}\frac{\| \rI_{--}g  \|_{L^{2}({\Gamma}_{\hdelta})} }
{  \|{g}\|_{L^2({\Gamma}^{-})}}.
\end{equation}
By Part (ii) of Lemma \ref{lem:Friday}, $\hrho$ and $\hgamma$ are well-defined. 
\end{definition}

\begin{figure}[H]
  \begin{center}
  \includegraphics[scale=0.8]{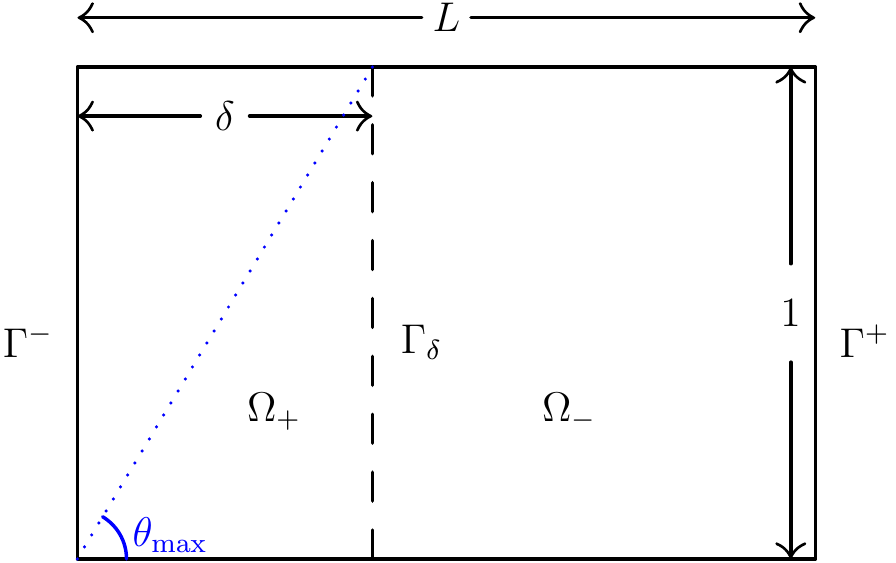}
\end{center}   
\caption{The canonical  domain ${\wOmega}$, composed of $\Omega_+$ (left) and $\Omega_-$ (right).  
The dotted diagonal with angle  $\theta_{\max}$ labelled in blue is used in \S\ref{subsec:SCA} below.
 \label{fig:hat}}
\end{figure}

We record the following simple relationship between $\gamma$ and $\rho$.

\begin{lemma}\label{lm:imp}
For ${\gamma} , {\rho}$ as defined in \eqref{def:rho_gamma}, 
\begin{align}\label{eq:triangle} 
\gamma(k,\delta, L) \le \sqrt{1+ {\rho}^2(k,\delta, L )}.
\end{align}
\end{lemma}

\begin{proof} 
Let $u\in U_0(\wOmega)$ be the solution to  \eqref{eq:plane-wave-scattering1}. 
 By Lemma \ref{lm:orth-trace-helm} 
\begin{align}
\int_{{\partial \wOmega}} \left| \frac{\partial {u}}{\partial {n}} -\ri \hk {u}\right|^2 ds \ =\ 
\int_{\partial {\wOmega}} \left( \left| \frac{\partial {u}}{\partial {n}}\right|^2   +  \hk^2  \left| {u}\right|^2 \right) ds \ =
\ \int_{\partial {\wOmega}} \left|- \frac{\partial {u}}{\partial {n}} -\ri \hk {u}\right|^2 ds. 
  \label{isom} \end{align}
Using the boundary conditions in \eqref{eq:plane-wave-scattering1} together with  \eqref{isom},  we obtain 
\begin{equation}\label{eq:isometry1}
  \int_{{\Gamma}^{-}} \left| -\partial_x {u}-\ri \hk {u}\right|^2 ds\  = \ \int_{{\Gamma}^{-}} \left| \frac{\partial {u}}{\partial \hn} -\ri \hk
    {u}\right|^2 ds \ =\  \int_{\igg{\partial} \wOmega} \left| \frac{\partial {u}}{\partial \hn} -\ri \hk {u}\right|^2 ds \ =\
  \int_{\partial {\wOmega}} \left|- \frac{\partial {u}}{\partial {n}} -\ri \hk {u}\right|^2 ds.
\end{equation}
Now let ${\Omega}_-$ denote the subdomain of $\wOmega$ with height $1$ and  vertical sides $\hGamma_{\hdelta}$ and  
${\Gamma}^+$ (see Figure \ref{fig:hat}).
Since $\hu \in U_0(\hOmega_-)$, repeating the argument above gives
\begin{align}
  \int_{{\Gamma}_{\hdelta}} \left| -\partial_x {u}-\ri \hk {u}\right|^2 ds &= \int_{\partial \hOmega_{-}}
                                                                             \left|  \frac{\partial {u}}{\partial \hn}  -\ri \hk {u}\right|^2 ds \ = \ \int_{\partial \hOmega_{-}} \left|  - \frac{\partial {u}}
                                                                             {\partial \hn}  -\ri \hk {u}\right|^2 ds \nonumber   
  \\ & = \int_{{\Gamma}_{\hdelta}} \left| \partial_x \hu -\ri \hk {u}\right|^2 ds +  \int_{ \partial {\Omega}_{-} \backslash {\Gamma}_{\hdelta}} \left|- \frac{\partial {u}}{\partial {n}} -\ri \hk {u}\right|^2 ds. \label{eq:isometry2}
\end{align}
Since $\partial {\Omega}_{-} \backslash {\Gamma}_{\hdelta}  \subset \partial {\Omega}$, we can use the definition \eqref{def:rho_gamma} of
$\hrho(\hk, \hdelta, \hL)$   to estimate the first term on the right-hand side of \eqref{eq:isometry2} and use   \eqref{eq:isometry1} to estimate the second term on the right-hand side of  \eqref{eq:isometry2}:
\begin{align*}
\int_{{\Gamma}_{\hdelta} } \left| -\partial_x {u}-\ri \hk {u}\right|^2
  &\ \le \ {\rho}^2 ({k}, {\delta}, \hL )  \int_{{\Gamma}^{-}} \left| -\partial_x {u}-\ri \hk {u}\right|^2 ds
    + \int_{{\Gamma}^{-}} \left| -\partial_x {u}-\ri \hk {u}\right|^2 ds\\
&\ =  \ \left(1+ {\rho}^2 ({k},{\delta}, \hL )\right)  \int_{{\Gamma}^{-}} \left| -\partial_x {u}-\ri \hk {u}\right|^2 ds.
\end{align*}
The result follows from the definition of $\hgamma(\hk, \hdelta, \hL)$ \eqref{def:rho_gamma}.
\end{proof}

\subsubsection{Main convergence results obtained by bounding the actions of $\cL$ and $\cU$ via single impedance-to-impedance maps}
\label{subsec:main} 

We now return to the physical domain, as depicted in Figure \ref{fig:overlap}. 

\begin{corollary}\label{coro:imp}
In 2-d, with $\Gamma_{\ell}^{\pm}$ as defined in Notation \ref{def:geom},
  \begin{align} 
    \|\Imaps{\ell}{-}{\ell-1}{+}\| &= \hrho(kH, \delta_\ell/H, L_\ell/H),   \label{rho1} \\
                                   \|\Imaps{\ell}{+}{\ell+1}{-}\| & =      \hrho(kH, \delta_{\ell+1}/H, L_\ell/H),  
\label{rho2} \end{align}
and 
\begin{align}
  \|\Imaps{\ell}{-}{\ell+1}{-}\| & =   \hgamma(kH,  (L_\ell-\delta_{\ell+1})/H, L_\ell/H),   \label{gamma1}\\
                                       \|\Imaps{\ell}{+}{\ell-1}{+}\| & = \hgamma(kH, (L_\ell-\delta_\ell)/H, L_{\ell}/H ).  \label{gamma2}
\end{align}
\end{corollary}
\begin{proof}
We outline how to prove \eqref{rho1}; the proofs of \eqref{rho2}-\eqref{gamma2} are similar.
Following the discussion in \S\ref{sec:notation},
 the definition of $\Imaps{\ell}{-}{\ell-1}{+}g$ involves a homogeneous Helmholtz problem
  on $\Omega_\ell$, which has  length $L_\ell$ and height $H$. \ig{Using  an affine transformation
with scaling factor $1/H$, we transform this   to  a Helmholtz problem on the (canonical)  domain with length $L_\ell/H$ and height $1$ with wavenumber  $kH$.}
  The required  impedance data comes from evaluating in the right-facing direction at the
  interior interface situated at  position $\delta_\ell/H$ on the canonical domain, yielding \eqref{rho1}.   
\end{proof}

Up to now  we have developed the theory with general $L_\ell$ and $\delta_\ell$ to emphasise that these can vary with $\ell$. To reduce technicalities in the remainder of the theory, we introduce the simplifying notation.
\begin{align}  \rho & = \max_\ell \Big\{ \hrho(kH, \delta_\ell/H, L_\ell/H),  \,\, \hrho(kH, \delta_{\ell+1}/H, L_\ell/H)\Big\},  \label{simplify1} \\
  \gamma & = \max_\ell \Big\{ \hgamma(kH,(L_\ell -  \delta_\ell)/H, L_\ell/H),  \,\, \hgamma(kH, (L_\ell -\delta_{\ell+1})/H, L_\ell/H)\Big\}. 
\label{simplify2} \end{align} 
We make this slight abuse of notation to avoid introducing additional symbols for the maxima above.

\begin{lemma}\label{lm:TLU}
Let $\rho$ and $\gamma$ be defined as in \eqref{simplify1}, \eqref{simplify2},
and let $\|\cdot\|_{1,k,\partial}$ be  as in \eqref{eq:pseudo-energy1}.
Then,  
\begin{align}
  \|\bcL\bcU \bv \|_{1,k,\partial} \le \rho \|\bcU \bv \|_{1,k,\partial}\quad &\text{and}\quad \|\bcU\bcL\bv\|_{1,k,\partial} \le \rho\|\bcL \bv \|_{1,k,\partial}  \quad \text{for all} \  \bv \in \bbU, 
\label{48b}\\
  \|\bcL^2 \bv \|_{1,k,\partial} \le \gamma  \|\bcL\bv \|_{1,k,\partial}\quad &\text{and} \quad \|\bcU^2\bv \|_{1,k,\partial} \le \gamma \|\bcU \bv \|_{1,k,\partial} \quad \text{for all} \  \bv \in \bbU,\label{48a} \\
\|\bcL\bv \|_{1,k,\partial} \le \sqrt{\gamma^2+\rho^2}\|\bv \|_{1,k,\partial}\quad &\text{and}\quad \|\bcU\bv\|_{1,k,\partial} \le \sqrt{\gamma^2+\rho^2}\|\bv \|_{1,k,\partial} \quad  \text{for all} \  \bv \in \bbU_0. \label{48c} 
\end{align}
\end{lemma}

\begin{proof}  To prove the first estimate  in \eqref{48b}, we use  \eqref{products}, the bound \eqref{cor_T21new} 
(recalling that $\Imaps{\ell, \ell+1}{}{\ell+1,\ell}{}=\Imaps{\ell}{+}{\ell+1}{-}$), \eqref{rho2},
  and  \eqref{simplify1}
        to obtain 
$$
\begin{aligned}
\|\bcL\bcU \bv\|_{1,k,\partial}^2 &\ = \   \sum_{\ell = 1}^{N-1}  \N{\cT_{\ell+1,\ell}\cT_{\ell,\ell+1} \, v_{\ell+1}  }_{1,k,\partial \Omega_{\ell+1}}^2 
\ \le \ \max_{\ell = 1, \ldots, N-1} \|\Imaps{\ell}{+}{\ell+1}{-}   \|^2\,  \sum_{\ell=1}^{N-1}    \|\cT_{\ell,\ell+1} \, v_{\ell+1} \|_{1,k,\partial \Omega_{\ell}}^2 \\
&\leq \rho^2 \sum_{\ell = 1}^{N-1}  \|\cT_{\ell,\ell+1} \, v_{\ell+1} \|_{1,k,\partial \Omega_\ell}^2 \ = \ \rho^2 \|\bcU\bv\|_{1,k,\partial}^2\,  .
\end{aligned}
$$
The remaining estimates in \eqref{48b}, \eqref{48a}  are proved similarly.
We now focus on the first estimate in \eqref{48c} (the proof of the second one is similar).   Using the definition of $\bcL$, 
the definition \eqref{eq41}-\eqref{eq43} of $\cT_{\ell+1,\ell}$, 
the definition \eqref{eq:norm2} of the norm $\|\cdot\|_{1,k,\partial \Omega_{\ell+1}}$,
and the fact that $\chi_\ell = 1$ on $\Gamma_{\ell+1}^{-}$ and $\chi_\ell$ vanishes on $\Gamma_{\ell+1}^{+}$, we obtain 
\beq\label{eq:RAM1}
   \Vert \bcL\bv \Vert_{1,k,\partial}^2= \sum_{\ell=1}^{N-1} \Vert \cT_{\ell+1, \ell} \, v_{\ell}  \Vert_{1,k,\partial \Omega_{\ell+1}}^2
    \ = \ \sum_{\ell =1 }^{N-1}  \Vert (-\partial_x -\ri k) v_{\ell}  \Vert_{L^2(\Gamma_{\ell+1}^-)}^2.    
   \eeq
Since $v_{\ell}\in U_0(\Omega_{\ell})$, we can write $v_{\ell} = v_{\ell}^+ +v_{\ell}^-$ with components $v_{\ell}^\pm \in U_0(\Omega_{\ell})$ satisfying  $$\left( (\partial_x - \ri k) v_{\ell}^{+} \right)\big|_{\Gamma_{\ell}^+} = 0, \quad \text{and} \quad  \left( (-\partial_x - \ri k) v_{\ell}^{-} \right)\big|_{\Gamma_{\ell}^-} = 0.$$
Then observe that $\|v_\ell\|^2_{1,k,\partial \Omega_{\ell}}= \|v_\ell^+\|^2_{1,k,\partial \Omega_{\ell}} + \|v_\ell^-\|^2_{1,k,\partial \Omega_{\ell}}$,
and, by Corollary \ref{coro:imp},
\begin{align}\nonumber
\Vert (-\partial_x -\ri k) v_{\ell}  \Vert_{L^2(\Gamma_{\ell+1}^-)}  &\ \le\  \Vert (-\partial_x -\ri k) v_{\ell}^-   \Vert_{L^2(\Gamma_{\ell+1}^-)}  + \Vert (-\partial_x -\ri k) v_{\ell}^+  \Vert_{L^2(\Gamma_{\ell+1}^-)} \\ \nonumber
&\ \le\  \rho \|(\partial_x - \ri k) v_{\ell}^-\|_{L^2(\Gamma_{\ell}^+)} + \gamma \|(-\partial_x - \ri k)v_{\ell}^+\|_{L^2(\Gamma_{\ell}^-)}  \\ 
                                                                     &\ = \  \rho \|v_{\ell}^-\|_{1,k,\partial \Omega_{\ell}} + \gamma \|v_{\ell}^+\|_{1,k,\partial \Omega_{\ell}}
                                                                       \ \le \
                                                \sqrt{\gamma^2+\rho^2}\|v_{\ell}\|_{1,k,\partial \Omega_{\ell}}.\label{eq:RAM2}
\end{align}
Combining \eqref{eq:RAM1} and \eqref{eq:RAM2} yields
  \begin{align*}
   \Vert \bcL\bv \Vert_{1,k,\partial}^2   & \ \le \ \sum_{\ell=1}^{N-1}\sqrt{\gamma^2+\rho^2}\|v_{\ell}\|_{1,k,\partial \Omega_{\ell}}^2\ \le \   \sqrt{\gamma^2+\rho^2} \Vert \bv \Vert_{1,k,\partial}^2.
        \end{align*}
\end{proof}

The following two results are most useful when       $\rho$ is controllably small and $\gamma$ is  bounded independently of the important parameters; this situation is motivated by the fact that, in 1-d, $\rho = 0$ and $\gamma = 1$.
  
\begin{theorem}[Estimate of  $\cT^N$] \label{thm:TLU}
If  the number of subdomains $N \geq 2$, then,  for any $\bv\in \bbU_0,$
\begin{align}\label{eq:TLU}
\|\bcT^N  \bv\|_{1,k,\partial} \ \le \ 2\sqrt{\gamma^2+\rho^2}\left[(\gamma+\rho)^{N-1} -\gamma^{N-1}\right]  \|  \bv\|_{1,k,\partial} ,
\end{align}
where $\rho, \gamma$ are defined in \eqref{simplify1} and \eqref{simplify2}. 
\end{theorem}
\begin{proof} 
We use Theorem \ref{thm:sumsumLU} with $n=N$, so the $j=0$ term in \eqref{eq:dagger} vanishes. We now claim
  that,  for each $p \in \cP(N,j)$ with  $j \in \{1, \ldots, N-1\}$,  and  for any $\bv \in \bbU_0$,
  \begin{align} \label{eq:notproved} \Vert p(\bcL, \bcU)\bv \Vert_{1,k,\partial } \ \leq \ \sqrt{\gamma^2 + \rho^2} \, \rho^j\,  \gamma^{N-1-j} \|  \bv\|_{1,k,\partial}  .\end{align}
  \ig{We prove \eqref{eq:notproved} in the case $j = 1$,  where  $p(x,y) = x^{s_0} y^{s_1}$ with $s_0 + s_1 = N$;  the case of higher $j$ is obtained by induction.  Then  
    using  Lemma \ref{lm:TLU}  freely,
    \begin{align}  \Vert p(\bcL, \bcU) \bv\Vert_{1,k,\partial}  & = \Vert \bcL^{s_0} \bcU^{s_1} \bv \Vert_{1,k,\partial} \leq \gamma^{s_0-1}  \Vert \bcL \bcU^{s_1} \bv \Vert_{1,k,\partial} \leq \rho \, \gamma^{s_0-1}   \Vert  \bcU^{s_1} \bv \Vert_{1,k,\partial} \nonumber \\
                                                              & \leq \rho\,  \gamma^{s_0 + s_1 - 2}    \Vert \bcU \bv \Vert_{1,k,\partial} \leq \sqrt{\gamma^2 + \rho^2} \, \rho \, 
                                                                \gamma^{N-2}    \Vert \bv \Vert_{1,k,\partial}\, .   \nonumber \end{align}}
  Hence, combining  \eqref{eq:notproved} and \eqref{eq:composite},
   $$ \Vert \bcT^N\bv \Vert_{1,k,\partial}  
   \  \leq  \ 2 \sqrt{\gamma^2 + \rho^2} \, \left[ \sum_{j=1}^{N-1} \left(\begin{array}{c} N-1\\j \end{array} \right) \,  \rho^j\,  \gamma^{N-1-j} \right] \|\bv\|_{1,k,\partial}  , $$
   and an application of the Binomial Theorem gives   \eqref{eq:TLU}. 
\end{proof} 

\begin{corollary}[Estimate of  $\cT^N$, useful for  $\rho$ small]\label{coro:TLU}  Assume $\rho \leq \rho_0\leq \gamma$ \ig{and $N \geq 3$}. For any $\bv\in \bbU_0,$
\begin{align} \label{eq:rough}
\|\bcT^N \bv  \|_{1,k,\partial } \ \le \ \left(\left[2\sqrt{2}    \gamma^{N-1} (N-1)\right] \rho    \ + \ C(N,\gamma) \rho^2  \right) \| \bv \|_{1,k,\partial}, 
\end{align}
where $C(N,\gamma) :=\sqrt{2} (N-1)(N-2) \gamma(\gamma+\rho_0)^{N-3}$. 
Thus, if $\rho $ is small (relative to $\gamma$ and $N$),  then $\bcT^N$ is a contraction.
\end{corollary}

   \begin{proof}[Proof of Corollary \ref{coro:TLU}]
By Theorem \ref{thm:TLU}, Taylor's theorem, and the fact that $\rho\le \gamma,$ 
\begin{align*}
  \|\bcT^N \bv \|_{1,k,\partial } 
  \le 2 \sqrt{2} \gamma \left((N-1) \gamma^{N-2}\rho + \frac{(N-1)(N-2)}{2} (\gamma + \rho)^{N-3} \rho^2\right)\|\bv \|_{1,k,\partial },
\end{align*}
where we have used the fact that the function $x\mapsto (\gamma+x)^{N-3}$ is increasing on $[0,\rho_0]$ to bound the Taylor-theorem remainder.
  \end{proof}

Corollary \ref{coro:TLU} provides  an estimate for $\Vert \cT^N\Vert_{1,k,\partial}$ that is first order in $\rho$. An estimate with a higher order in $\rho$ can be obtained by considering higher powers of $\cT$.

\begin{corollary} [Estimate of higher powers of $\cT^N$] \label{coro:higherrho}
For $s\geq 1$, and $\bv \in \bbU_0$, 
$$\|\bcT^{sN} \bv  \|_{1,k,\partial } \ \le \ 2 \sqrt{\gamma^2 + \rho^2} \, \left[ \sum_{j=s}^{sN-1} \left( \begin{array}{c} sN-1 \\ j \end{array} \right) \gamma^{sN-1-j} \rho^j \right] \,  \| \bv \|_{1,k,\partial}.  
$$
  \end{corollary} 

  \begin{proof}
\ig{The proof uses estimate  \eqref{eq:composite} with $n = sN$.}  Consider any  monomial $p \in \cP(sN , j)$ with $j \leq s-1$. This is a monomial of order $sN$ with $j \leq s-1$ transitions from $x$ to $y$ or from $y$ to $x$. Thus it must contain at least one string of length $\geq N$ without jumps. (For example, any   $p \in \cP(2N,1)$ must contain one string of length $\geq N$ without a jump.) Hence, using \eqref{eq:nilpotent}, 
    $$ p(\bcL, \bcU) = 0 \quad \text{for all} \quad p \in \cP(sN , j) \quad \text{when } \quad j \leq s-1.$$
    and thus the result follows in a similar way to that in the proof of Theorem \ref{thm:TLU}. \end{proof}

\subsubsection{Estimating  the canonical map $ \rI_{-+}$ using semiclassical analysis}
\label{subsec:SCA}

Theorem \ref{thm:TLU} and Corollaries \ref{coro:TLU} and \ref{coro:higherrho} show that convergence of the iterative method improves as $\rho$ gets smaller. We now describe results from \cite{LaSp:21} that give sharp bounds on the large-$k$ limit of $\rho$ (with other parameters, such as $\delta$, fixed).

In the canonical domain (Figure \ref{fig:hat}) for the strip decomposition, the impedance boundary conditions on the top and bottom sides are due to the (outer) impedance boundary condition \eqref{eq:BC}, and the impedance boundary conditions on the left and right sides are due to the (inner) impedance boundary conditions imposed by the domain-decomposition algorithm.

For simplicity, \cite{LaSp:21} considers the case when the (outer) boundary condition \eqref{eq:BC} is replaced by a condition
that the solution is ``outgoing'' (in a sense made precise by the notion of the \emph{wavefront set}); i.e., that no outgoing rays hitting $ \pwOmega$ are reflected. Studying this situation therefore focuses on the effect of the impedance boundary conditions coming from the domain decomposition, and ignores the effect of any high-frequency reflections from absorbing boundary conditions on $\pwOmega$ (see 
\cite{GaLaSp:20} for a precise description of these reflection effects).
The outgoing condition replacing \eqref{eq:BC} is, in some sense, the ideal absorbing boundary condition at high frequency on $\pwOmega$. Since perfect matched layers approximate the outgoing condition exponentially well at high frequency \cite{GaLaSp:21}, we expect that the results of \cite{LaSp:21} will also hold when the outgoing condition is replaced by perfectly matched layers (and this is work in progress).
  
 The paper \cite{LaSp:21} considers the following two model problems. 

\emph{Model Problem 1:} the canonical problem specified in Definition \ref{def:imp_unit} with outgoing conditions on the top and bottom, impedance data posed on the left, and zero impedance data on the right (i.e., that discussed above), and 

\emph{Model Problem 2:}
 the canonical problem with outgoing conditions on the top, bottom, and right sides, and impedance data posed on the left. 

Model Problem 2 is the canonical problem for the strip-decomposition algorithm applied with two subdomains when the global problem is \eqref{eq:Helm} with outgoing boundary conditions. The reason for considering this further simplification is that in Model Problem 1 a ray moving from left to right can still be reflected an infinite number of times,
and the reflection coefficient on $\Gamma^-$ depends on the data; 
thus  an upper bound for general impedance data in this situation is more challenging to prove.

\paragraph{Upper and lower bounds on $\| \rI_{-+}\|$ for Model Problem 2.}

Let $\wOmega$ be the canonical domain of Figure \ref{fig:hat}, so that 
$\Gamma^- := \{0\} \times [0,1].$ Let $\Gamma_\delta := \{\delta\} \times [0,1]$ and we define 
$\Gamma_{\rm out}:= \partial {\wOmega} \backslash \Gamma^-$ (the subscript ``out'' indicates that this part of the boundary has the ``outgoing'' condition on it).
Given $g\in L^2(\Gamma^-)$, let $u$ be the solution to 
\begin{equation} \label{eq:model_cell}
\begin{cases}
(\Delta + k^2) u = 0 \quad \text{ in } \{ x_1>0 \} \\
(-\partial_{x} -\ri  k) u = g \quad \text{ on } \Gamma^-, \\
\text{$u$ is outgoing near $\Gamma_{\rm out}$},
\end{cases}
\end{equation}
where the outgoing condition near $\Gamma_{\rm out}$ is defined in terms of the wavefront set, as will
be explained in  \cite{LaSp:21}.
In analogy with \eqref{eq:maps}, $\rI_{-+}: L^2(\Gamma^-) \rightarrow L^2(\Gamma_\delta)$ is defined by 
\beq\label{eq:David_imp_def}
  \rI_{-+} g:= \partial_{x} u - \ri k u \quad \ton \Gamma_\delta.
\eeq

\begin{theorem}[Upper and lower bounds on $\|\rI_{-+}\|$ for Model Problem 2 from \cite{LaSp:21}]\label{th:EandD}
Let
\beq\label{eq:thetamax}
\theta_{\rm max} := \arctan (\delta^{-1}).
\eeq
Then, for any $\epsilon>0$, there exists $k_0(\epsilon)>0$ such that, for all $k\geq k_0$,
\beq\label{eq:David1}
\N{\rI_{-+}}_{L^2(\Gamma^-)\rightarrow L^2(\Gamma_\delta)} \leq 
\frac{1-\cos \theta_{\rm max}}{1 + \cos \theta_{\rm max}} + \epsilon.
\eeq
Furthermore, for any $0 < \epsilon' < \theta_{\rm max}$,
\beq\label{eq:David2}
\lim_{k\tendi} \N{\rI_{-+}}_{L^2(\Gamma^-)\rightarrow L^2(\Gamma_\delta)} \geq 
 \frac{1-\cos (\theta_{\rm max} - \epsilon')}{1 + \cos (\theta_{\rm max}-\epsilon')}.
\eeq
\end{theorem}

Observe that there exist $C_1, C_2>0$ such that
\beq\label{eq:thetabound}
C \delta^{-2}\leq C_1 (\theta_{\rm max})^2 
\leq 
\frac{1-\cos (\theta_{\rm max})}
{1+\cos (\theta_{\rm max})}
\leq 
C_2 (\theta_{\rm max})^2 \leq C \delta^{-2}
\eeq
and thus Theorem \ref{th:EandD} shows that $\lim_{k\tendi} \|\rI_{-+}\|_{L^2(\Gamma^-)\rightarrow L^2(\Gamma_\delta)}$ is bounded above and below by multiples of $(\theta_{\max})^2$, and hence multiples of $\delta^{-2}$,
where we recall that   $\delta$ is the distance of  $\Gamma_\delta$ from $\Gamma^-$.

\textbf{The idea behind Theorem \ref{th:EandD}.} 
The tools of semiclassical/microlocal analysis decompose solutions of PDEs in both frequency and space variables.
 These tools show that, at high-frequency, Helmholtz solutions  
propagate along 
the rays of geometric optics, in the sense that the wavefronts are perpendicular to the ray direction. 
The ideas behind Theorem \ref{th:EandD} can therefore be understood by first looking at the impedance-to-impedance map for plane-wave solutions of the Helmholtz equation (since these are simple Helmholtz solutions travelling along rays),
ignoring the fact that these do not satisfy the outgoing condition on all of $\Gamma_{\rm out}$, and thus are not solutions of Model Problem 2.

Let $u$ be a plane-wave in $\mathbb{R}^2$ with direction $(\cos\theta,\sin\theta)$ (i.e., propagating at angle $\theta$ to the horizontal), i.e.,
\beqs
u(x,y) = \exp\big(\ri k (x \cos\theta + y \sin \theta)\big).
\eeqs
Then
\begin{align*}
(-\partial_{x} - \ri k )u|_{\Gamma^-}&= \ri k (-\cos\theta-1) \exp\big(\ri k  y \sin\theta\big), \\
(\partial_{x} - \ri k )u|_{\Gamma_\delta}&= \ri k (\cos\theta-1) \exp\big(\ri k (\delta \cos\theta + y \sin\theta)\big),
\end{align*}
so that, for this class of $u$, 
\beq\label{eq:ratio}
\frac{
\N{(\partial_{x} - \ri k )u}_{L^2(\Gamma_\delta)}
}{
\N{(-\partial_{x} - \ri k )u}_{L^2(\Gamma^-)}
}
=
\frac{1-\cos\theta}{1+\cos\theta}.
\eeq
We now use \eqref{eq:ratio} as a heuristic for the behaviour of the impedance-to-impedance map on solutions of Model Problem 2 travelling on rays at angle $\theta$ to the horizontal.
Since the solution of Model Problem 2 is outgoing on $\Gamma_{\rm out}$, anything reaching $\Gamma_\delta$ must arrive on a ray emanating from $\Gamma^-$ and hitting $\Gamma_\delta$, and the maximum angle such rays can 
have with the horizontal satisfies $\tan \theta_{\max} =\delta^{-1}$; see Figure \ref{fig:hat}.
The right-hand side of \eqref{eq:ratio} is largest when $\theta=\theta_{\max}$, with this expression then (modulo the presence of $\epsilon$ and $\epsilon'$) the right-hand sides of \eqref{eq:David1} and \eqref{eq:David2}. 

The arguments in \cite{LaSp:21} use these ideas in a rigorous way; for example, to prove the lower bound \eqref{eq:David2}, we take a sequence of data $(g(k))_{k>0}$ where the Helmholtz solutions it creates are concentrated at high frequency in a beam coming from one point of $\Gamma^-$ and traveling in one direction $(\cos \theta, \sin \theta)$, and we take $\theta$ to be arbitrarily close to $\theta_{\rm max}$. The notion of concentration at high frequency is understood in a rigorous way using so-called \emph{semiclassical defect measures}; see \cite[\S9.1]{LaSpWu:19a} for an informal overview of these, and \cite[Chapter 5]{Zw:12}, \cite{Bu:02}, \cite{Mi:00}, \cite{GaSpWu:20}, \cite{GaLaSp:21}.

Finally, we highlight that these ray arguments and angle considerations are similar to those in \cite[\S5]{GaHa:05} used to optimise boundary conditions in domain decomposition for the wave equation.
  
\subsubsection{Estimating higher order products of  $\cL$ and $\cU$}\label{sec:composite}

The estimates in Theorem \ref{thm:TLU} and Corollaries  \ref{coro:TLU} and \ref{coro:higherrho} use  Lemma \ref{lm:TLU} repeatedly to bound $\Vert \cT^n\Vert_{1,k,\partial}$  in terms of powers of  $\rho$ and $\gamma$. For example, to bound the term
 $
 \bcL^{s}\bcU 
 $
for an integer $s > 0$, the argument in Theorem \ref{thm:TLU}  uses  \eqref{48b}--\eqref{48c} to obtain
\begin{equation}\label{eq:compositeLU}
\| \bcL^{s}\bcU  \|_{1,k,\partial} \ \le \ \gamma \| \bcL^{s-1}\bcU  \|_{1,k,\partial}  \ \le \ \gamma^2 \| \bcL^{s-2}\bcU  \|_{1,k,\partial}\ \le\  \cdots \ \le \   \sqrt{\gamma^2 + \rho^2} \, \gamma^{s-1} \,  \rho .
\end{equation}
Thus  if $\rho$ is  controllably small,  Corollary \ref{coro:TLU} implies   power contractivity for  $\cT$. The use of Corollary \ref{coro:TLU} is illustrated in  
Experiment \ref{Expt1} below, which  shows that the convergence rate of the domain decomposition method improves  as $\rho$ decreases.
However, we expect that estimates like  that in Corollary \ref{coro:TLU}  {are not} in general sharp.
In particular, looking at the case $k = 80$ in  Figure \ref{fig:error} and Table \ref{tb:composite-imp1} of \S \ref{sec:numerical} we see a case when $\rho \approx 0.15$, but 
the method converges effectively for $N = 4,8,16$, even though \eqref{eq:rough} grows linearly in $N$.
Thus, we expect that sharper results  
may be possible  by
bounding composite maps such as $\bcL^{s} \bcU$ directly, rather than estimating each of their factors, as in
\eqref{eq:compositeLU}. 
In fact, in \cite{LaSp:21}, ray arguments are used to give insight into the behaviour of these composite maps in the $k\tendi$ limit, and these arguments do indeed indicate that the compositions of the maps behave better than the products of the norms of the   individual components. 

To illustrate the use of composite maps,  we consider  the dominant ($j=1$) 
term in 
\eqref{eq:composite}: 
\begin{align}  2 (N-1) \max_{p \in \cP(N,1)} \Vert p(\bcL, \bcU) \Vert_{1,k,\partial} . \label{eq:first_order} 
\end{align} 
 One of the terms appearing inside  the maximum corresponds to  $p(\bcL, \bcU) = \bcL^{N-1} \bcU$. 
This operator is blockwise very sparse; for $N\geq 2$ all its nonzero blocks lie on the $(N-2)$th diagonal below the main diagonal (see 
\eqref{products} for the case $N=2$). The  $(N,2)$th element of $\bcL^{N-1}\bcU$ is     
\begin{align} \label{eq:sentry} 
\cT_{N,N-1}\cT_{N-1,N-2}\cdots\cT_{3,2}\cT_{2,1}\cT_{1,2} = \left(\prod_{j=1}^{N-1} \cT_{j+1,j}  \right) \cT_{1,2},
\end{align} 
where the operator product is understood as concatenated on the left  as the counting index $j$ increases. 

Rewriting \eqref{eq:impit2new}   using the notation \eqref{notstrip}, we see that, for any $s$, 
$$ \imp_{N} (\cT_{N,N-1} \cT_{N-1, N-2} z_{N-2} ) = \Imaps{N-1}{-}{N}{-}  \, \imp_{N-1} ( \cT_{N-1,N-2} z_{N-2}).  $$
A straightforward  induction argument {then shows that} 
\begin{align} 
\label{eq:induction} \imp_{N} \left(\left(\prod_{j=1}^{N-1} \cT_{j+1,j}  \right) \cT_{1,2}z_2\right)\  =\  \left(\left(\prod_{j=2}^{N-1} \Imaps{j}{-}{j+1}{-}\right) \Imaps{1}{+}{2}{-}\right)  \, 
\imp_1 ( \cT_{1,2} z_{2}) .  
\end{align}
In Experiment \ref{Expt3} we use \eqref{eq:induction} to  compute the norm of the composite operator $\bcL^{N-1} \bcU $
directly.

\section{Finite-element approximations}
\label{sec:fem}

In this section we describe how we use finite-element computations to illustrate our theoretical results.
Due to space considerations, we restrict here to a description of  algorithms and brief remarks  on  finite-element convergence; more details are in  \cite{GoGrSp:21}.

For any domain $\Omega$, let  ${T}^h$ be a  family of  shape-regular   
meshes  on   $\Omega$ with mesh diameter $h \rightarrow 0$. We assume each mesh resolves the boundaries of all subdomains.
Let ${V}^h$ be an $H^1$-conforming nodal finite-element space of polynomial degree $p$ defined with respect to
${T}^h$. For any  subset (domain or surface) $\Lambda$ that is resolved by $T^h$, we define 
$\mathrm{V}^h(\Lambda) = \{ w_h\vert_\Lambda : w_h \in V^h\}$. We let  $ {N}(\Lambda)$ denote the set of
nodes of the space $V^h$ that lie in $\Lambda$.

\subsection{The iterative method}
\label{subsec:iterative} 
Here we describe the computation of finite-element approximations of the iterates $u^n$ defined in \eqref{eq21} -- \eqref{star}.
With   $a$   as   in \eqref{eq:sesq1},  and for any $F \in H^1(\Omega)'$, 
we consider  finding $u \in H^1(\Omega)$ satisfying
\begin{align} \label{eq:weakfe}a(u,v) \ =\ F(v) \quad \text{for all} \quad v \in H^1(\Omega);
 \end{align} 
this includes the weak form of \eqref{eq:Helm}, \eqref{eq:BC} as a special case.
To discretize \eqref{eq:weakfe}, we define $\cA_h:V^h \mapsto (V^h)'$ and $F_h \in V_h'$ by
$
  (\cA_h u_h)(v_h) : = a(u_h,v_h)$ \text{and}  $F_h(v_h) :=  F(v_h)$  for  $u_h , v_h \in V^h$.   
  The  finite-element solution $u_h\in V^h$ of  \eqref{eq:weakfe} satisfies
  \begin{align} \label{eq:weakfeh} \cA_h u_h = F_h . \end{align}

  To formulate the discrete version of  \eqref{eq21} -- \eqref{star} on each subdomain $\Omega_\ell$,
  we introduce the  local space
$V^h_\ell: = V^h(\Omega_\ell)$,   
 and define the local operators $\cA_{h,\ell}: V^h_\ell \rightarrow (V^h_\ell)'$ by
  $(\cA_{h,\ell}u_{h,\ell})(v_{h,\ell}) : = a_\ell(u_{h,\ell}, v_{h,\ell}),$
with $a_\ell$ as defined in \eqref{eq:sesq2}.  
We also introduce  prolongations $\cR_{h,\ell} ^\top, \tcR_{h, \ell}^\top : V_\ell^h \rightarrow V^h$ defined for all $v_{h,\ell} \in V^h_\ell$ by
\begin{equation}\label{nodewise}
(\cR_{h,\ell} ^\top  v_{h,\ell})(x_j)  = \left\{ \begin{array}{ll} v_{h,\ell} (x_j) & \quad {x_j \in N(\overline{\Omega_\ell})}, \\ 
                                          0 & \quad \text{otherwise, }\end{array} \right.  \quad \text{and} \quad \tcR_{h,\ell}^\top v_{h,\ell}  = \cR_{h,\ell}^\top (\chi_\ell v_{h,\ell}). 
\end{equation}
Note that the  extension  $\cR_{h,\ell}^\top v_{h,\ell} \in V^h $ is defined {\em nodewise}: it  coincides with $v_{h,\ell}$ at  nodes in  $\overline{\Omega_\ell}$ and vanishes at nodes in $\Omega\backslash \overline{\Omega_\ell}$. Thus $\cR_{h,\ell} ^\top  v_{h,\ell} \in H^1(\Omega)$  is  a  finite-element approximation of the operator of extension by zero,  even though the (true) extension by zero does not, in general, map $H^1(\Omega_\ell)$ to $H^1(\Omega)$. We define the restriction operator $\cR_{h,\ell}: V_h' \rightarrow V_{h, \ell}'$  by duality, i.e., for all
$F_h \in V_h'$,  $$ (\cR_{h,\ell}F_h)(v_{h,\ell}) := F_{h}(\cR_{h,\ell}^\top v_{h,\ell}), \quad v_{h, \ell} \in V_\ell^h .$$ 

It is shown in \cite{GoGaGrSp:21} that a natural discrete analogue of \eqref{eq21}-\eqref{star} is
\begin{align} \label{eq:discT}
  u_{h, j}^{n+1} \ : = \  u_h^n|_{\Omega_j}  +  \cA_{h,j}^{-1} \cR_{h,j} (F_h - \cA_h  u_h^n) \quad \text{for} \quad j = 1, \ldots , N, \quad \quad  n= 1,2,\ldots,
  \end{align} 
  \begin{align} \label{eq:update}
\text{where} \qquad    u_h^n = \sum_\ell {\tcR_{h,\ell}^\top} u_{h,\ell}^n \quad \text{for } \ n = 0,1, \ldots. 
     \end{align}

       The algorithm \eqref{eq:discT}, \eqref{eq:update}  is derived in  \cite{GoGaGrSp:21} as a finite-element approximation of \eqref{eq21}--\eqref{star}.   
  In fact \eqref{eq:discT}, \eqref{eq:update} is equivalent to the well-known  Restricted Additive Schwarz method with impedance transmission condition (also known as WRAS-H \cite{KiSa:07} and  ORAS \cite[Definition 2.4]{DoJoNa:15} and \cite{StGaTh:07}). 

  Moreover, since  $u_h$  is the exact solution of \eqref{eq:weakfe}, we have, trivially, 
\beq\label{eq:trivial}
u_h \vert_{\Omega_j} = u_h \vert_{\Omega_j} + \cA_{h,j}^{-1} \cR_{h,j}(F_h - \cA_h u_h).
\eeq
 The error is then
  $\mathbf{e}_{h}^n := (e_{h,1}^n, e_{h,2}^n,\cdots,e_{h,N}^n)^\top,$ where $ e_{h,\ell}^n = u_h|_{\Omega_\ell} - u_{h,\ell}^n.$
      Subtracting  \eqref{eq:discT} from \eqref{eq:trivial}, 
       we obtain the error equation
    \begin{align} \label{eq:erroreq}
  e_{h, j}^{n+1} \ : = \  e_h^n|_{\Omega_j}  -  \cA_{h,j}^{-1} \cR_{h,j}  \cA_h  e_h^n \quad \text{for} \quad j = 1, \ldots , N,  \quad \text{where} \quad    e_h^n := \sum_\ell {\tcR_{h,\ell}^\top} e_{h,\ell}^n  . 
     \end{align} 
     The two expressions in \eqref{eq:erroreq} can be combined and written in  the operator matrix form:  \begin{align} \label{eq:disc_err} \be_h^{n+1} = \cT_h \be_h^n, \end{align}  providing a finite element analogue  of \eqref{eq:errit}. The matrix form of $\cT_h$  is discussed in \cite[\S5]{GoGaGrSp:21}.
   
  In \S \ref{sec:numerical} we plot error histories for this method. To do this, we need to choose a suitable norm in which to measure the error.
  Since ${e}_{h,\ell}^n \approx {e}^n_\ell \in U_0(\Omega_\ell)$ {(defined in Definition \ref{def:U})}, 
  it is natural to try to analyse $e_{h,\ell}^n$ in a finite-element  analogue of $U_0(\Omega_\ell)$. In fact, one can show   that, for each $n$,  
$$
e_{h,\ell}^n \in V^h_{\ell,0} \ :=\ \{v_{h,\ell}\in V^h_\ell~:~ a_\ell(v_{h,\ell}, w_{h,\ell}) = 0 \text{ for any }
w_{h,\ell}\in V^h_\ell\cap H^1_0(\Omega_\ell) \},
$$
which indicates that the error is `discrete Helmholtz harmonic'. Therefore we define the norm: 
   \begin{align}\nonumber
     \|v_{h,\ell}\|_{V_{\ell, 0}^h } \ &:= \sup_{w_{h,\ell} \in V_\ell^h, w_{h,\ell} |_{\partial \Omega_\ell} \neq 0}
                                                     \frac{\left|a_\ell(v_{h,\ell} ,w_{h,\ell}) \right|}{ \|w_{h,\ell}\|_{L^2(\partial \Omega_\ell)}}. \end{align}
           This is a norm for $h$ sufficiently small because  the sesquilinear  form $a_\ell$ satisfies a discrete inf-sup condition on $V^h(\Omega_\ell) \times V^h(\Omega_\ell)$  (with $h$-independent constant) \cite[Theorem 4.2]{MeSa:10}.
    The  norm of the error  vector $\mathbf{e}_{h}^n$ is then given by 
  \begin{align}
\|\mathbf{e}_{h}^n\|_{\mathbb{V}_0^h}& = \left(\sum_\ell\|e^n_{h,\ell}\|_{V_{\ell,0}^h}^2 \right)^{1/2}, \quad \text{where} \quad \mathbb{V}_0^h := \prod_\ell V_{\ell,0}^h  . \label{eq:pseudo-energy2}
\end{align}

\subsection{The impedance-to-impedance maps} 
\label{subsec:Impedance} 

We now describe the  computation of the canonical impedance-to-impedance  maps
$\rI_{s,t} $, defined on the canonical domain $\wOmega$ in Figure \ref{fig:hat}, for any $s,t \in \{ -,+\}$. {We emphasise that} this computation is used only to verify the theory of this paper, and is not needed in the implementation of the domain decomposition solver.

To construct  finite-element approximations of these  maps,  
we first derive  a  variational {problem} satisfied by them.   
To do this, we introduce the space  $V(\wOmega)$,  defined as the completion of $C^\infty(\overline{\wOmega})$ in the norm $\|\cdot\|_{V(\wOmega)}:=
 \left(\N{v}_{L^2(\wOmega)}^2 + \N{v}^2_{L^2(\partial \wOmega)}\right)^{1/2}. $
 Then we define the   sesquilinear form
    \beq\label{eq:alpha}
      \alpha(u,v) :=  -(\Delta u + k^2 u , v)_{\wOmega}  \ + \ \langle \partial u/\partial n - \ri k u, v\rangle_{\partial \wOmega}  \quad \text{for all} \ u \in U({\wOmega}), \ v \in V(\wOmega).  
      \eeq
      This  {form}   arises when  considering problem \eqref{eq:Helm}, \eqref{eq:BC} in strong (classical) form.  When $v {\in H^1(\wOmega)}$,    \eqref{eq:alpha} simplifies,  via  Green's first identity \cite[Lemma 4.3]{Mc:00},
      to
    \begin{align} \label{eq:equiv} \alpha(u,v) = a(u,v) \quad \text{for all}\   u \in U(\wOmega), \ v \in H^1(\wOmega),
    \end{align}
where $a$ denotes the sesquilinear form \eqref{eq:sesq1} defined on $\wOmega$.
    With $t\in \{+,- \}$ and   $v_t \in H^1(\Omega_t)$, let $\cR_t^\top v_t \in V(\wOmega)$ denote
    the function  
    that coincides with $v_t$ on $\Omega_t$ and is  zero elsewhere on $\wOmega$. Another  application of
    Green's first identity yields  the following 
    result.
            \begin{proposition}[Variational formulation of impedance-to-impedance map] \label{prop:alphaa} For  $s,t \in \{-,+\}$,  let $g \in L^2(\Gamma^s)$, and let  $u_s \in U_0(\wOmega)$ be the Helmholtz-harmonic function  with impedance data $g$ on $\Gamma^s$ and zero elsewhere. Then  
    \begin{equation}\label{eq:imp-map-eq}
      \langle \rI_{s,t} g, v_{t} \rangle_{\Gamma_\delta} = a_t (u_s,v_t) - \alpha(u_s,\cR_t^\top v_{t}), \quad \text{for all}    \ v_t \in H^1(\Omega_t),
    \end{equation}
where $$a_t(v,w) = \int_{\Omega_t} (\nabla v \cdot \nabla \overline{w} - k^2 v \overline{w}) -  \ri k \int_{\partial \Omega_t} v \overline{w}. $$ 
   \end{proposition}

   Motivated by  \eqref{eq:equiv} and  \eqref{eq:imp-map-eq},    we  define a finite-element approximation $\rI_{h,s,t}: L^2(\Gamma^s) \rightarrow V^h(\Gamma_\delta)$ to the map $\rI_{s,t}$ as follows.  Analogously to \eqref{nodewise},  for any $v_h \in V^h(\Gamma_\delta)$, we define its node-wise zero extension to all of $V^h({\wOmega})$ by
$$
\cR_{\Gamma_\delta,h}^\top v_{h} (x)= \left\{
\begin{aligned}
v_{h} (x)\quad &\text{for all} \quad  x \in {N}(\overline{\Gamma_\delta}), \\
0\quad &\text{for all}\quad   x \in {N}(\overline{\wOmega}\backslash \overline{\Gamma_\delta}). 
\end{aligned}\right.
$$
Note that  $\cR_{\Gamma_\delta,h}^\top v_{h}\in V^h \subset H^1(\wOmega)$  but is supported  only on the union of all elements of the mesh $T^h$
that touch $\Gamma_\delta$.
Using this,    we   define $\rI_{h,s,t} $ by the variational {problem}  
\begin{equation}\label{eq:discrete-imp-map}
\langle \rI_{h,s,t}  g, v_{h} \rangle_{\Gamma_\delta} = a_t (u_{h,s}, \cR_{\Gamma_\delta,h}^\top v_{h}) - a(u_{h,s},\cR_{\Gamma_\delta,h}^\top v_{h}) \quad \text{for all} \quad  v_{h}\in \mathrm{V}^h(\Gamma_\delta),
\end{equation}
where $u_{h,s} \in V^h(\wOmega)$ is  the standard finite-element approximation of the function $u_s$ (from  Proposition   \ref{prop:alphaa}),  obtained by  solving the homogeneous Helmholtz problem on  $\wOmega$ with impedance data $g$ on $\Gamma^s$ and zero elsewhere.

Note that several approximations have been made here. First, in {going from} \eqref{eq:imp-map-eq} {to \eqref{eq:discrete-imp-map}}, the test function $v_t \in H^1(\Omega_t) $ has been replaced by
$ v_{h} \in V^h(\Gamma_\delta)$ on the left-hand side and $\cR_{\Gamma_\delta,h} ^\top v_h$ on the right-hand side. Moreover the formula \eqref{eq:equiv}, which requires $u \in U(\wOmega)$, has been formally applied here with $u$ replaced by $u_{s,h} \in V^h \not \subset U(\wOmega)$. Despite these `non-conforming' approximations, it  can be
shown (with details in \cite{GoGrSp:21}) 
that, with $\Vert \cdot \Vert$ denoting the operator norm from $L^2(\Gamma^s) \rightarrow L^2(\Gamma_\delta)$, the following convergence result holds. 
\begin{corollary}[Convergence of discrete maps as $h \rightarrow 0$]\label{cor:main} 
  $$ \Vert \rI_{s,t} - \rI_{h,s,t} \Vert  \rightarrow 0 \ \quad \text{as} \ h \rightarrow 0. $$
\end{corollary}
\noindent Thus, 
the computations of $\Vert \rI_{h,s,t}\Vert$, given in \S \ref{sec:numerical}, 
are reliable approximations of    $\Vert \rI_{l,j}\Vert$.

A key point in the computation is the realisation that, for any $g \in L^2(\Gamma^s)$, $\rI_{h,s,t} g = \rI_{h,s,t} g_h$, with $g_h$ denoting the $L^2$-orthogonal projection of $g$ onto $V^h(\Gamma^s)$. (This is because the finite-element solution of the Helmholtz problem only `sees' the impedance data through  its $L^2$ moments against the finite-element basis functions.)
The operator  $\rI_{h,s,t}$ thus acts only on finite-dimensional spaces, and its norm can be computed by solving an appropriate matrix eigenvalue problem.       
In   \S \ref{sec:numerical} this is done using the code {\tt SLEPc}, within the finite-element package {\tt FreeFEM++}.

 \section{Numerical experiments }\label{sec:numerical}
 In this section, we verify the theoretical results 
in Theorems \ref{thm:TLU} and  \ref{th:EandD} and  Corollaries
  \ref{cor:composite}, \ref{coro:TLU}, \ref{coro:higherrho}
 using the finite-element approximations described  in \S \ref{sec:fem}. 
    We also perform some extra experiments that provide insight into  the performance of the iterative method in situations not covered by  the  theory.
 All  experiments are on rectangles,  the domain is discretized using a  uniform triangular mesh with  diameter  $h$, and we  use the Lagrange conforming element of degree 2. 
 {We use}  mesh diameter $h \sim k^{-5/4}$,  which is sufficient to ensure a bounded relative error as $k$ increases \cite[Corollary 5.2]{DuWu:15}.  The experiments are implemented using the package {\tt FreeFEM++} \cite{hecht2019freefem++}.
    
\subsection{Numerical illustration of   our theory} 

In this subsection we consider  the 2-d strip domain 
as in Notation \ref{def:geom}. The global domain $\Omega$ has height $H=1$ and length $L_\Omega$. 
  For the domain decomposition, we divide $\Omega$ into  $N$ equal non-overlapping rectangular domains and then extend
  each subdomain by adding to it neighbouring elements of distance $\leq r L_\Omega / N$ away, where $r>0$ is a parameter. 
  Thus the interior subdomains have length
$L=  (1+2r) L_\Omega / N$, while the end subdomains have length  $(1+r) L_\Omega / N$. The global overlap size is $\delta = 2r L_\Omega / N$. 
In the first two experiments we examine how  the convergence rate 
{depends on} the  parameters $\rho,\gamma$,   defined in \eqref{simplify1}, \eqref{simplify2} and \eqref{def:rho_gamma}.

\begin{experiment}\mythmname{{Computation of $\rho$ and $\gamma$ and} convergence {of the iterative method} as $\rho$ decreases} \label{Expt1}
      Corollaries \ref{coro:TLU} and \ref{coro:higherrho}  suggest that the convergence rate should 
      improve  as  $\rho$ decreases, and  Theorem \ref{th:EandD} 
      suggests that the large-$k$ limit of $\rho$ should decrease as $\delta$ increases.
      
Table \ref{tb:imp-left2right-aspect} {gives}  values of  $\rho$ and $\gamma$ as functions of  $k$, $\delta$ and   $L$
    as defined in \eqref{def:rho_gamma}. These are  computed  using  the method outlined in \S \ref{subsec:Impedance}. Here $r$ is chosen so that $\delta = L/3$. 
The top part of Table \ref {tb:imp-left2right-aspect} shows
that  $\rho$ decreases as $L$ increases,
as suggested by  Theorem \ref{th:EandD}. 

For fixed $k$, the observed decay rate of $\rho$ is slightly faster than  $\mathcal{O}(\delta^{-1})$.
The bottom part of this table  {shows the} corresponding values of $\gamma$.
Here  ${\gamma}\leq 1$, somewhat smaller than the upper bound  predicted by Lemma \ref{lm:imp}.   There is a very modest growth of the values of $\rho$ and $\gamma$ as $k$ increases, for each fixed $L$; given the lower bound in Theorem \ref{th:EandD}, we expect that the values of $\rho$ in Table \ref{tb:imp-left2right-aspect}  are in the preasymptotic regime for $k\tendi$.

  Figure \ref{fig:errorN3} 
  {shows} the corresponding convergence of the iterative method 
  {for} $N = 3$ subdomains and $\delta = L/3$  on a sequence of domains of increasing global length
$\mathrm{L_\Omega} =  4,8,16 $ (blue, black and red lines respectively); here the 
length of each subdomain, $L$, is also doubling for each experiment.

To   obtain the relative error in the iterative method, we solve the problem \eqref{eq:weakfeh} with  right-hand side $F=0$,
so that the   finite-element solution is $u_h = 0$ and the relative error is simply
\begin{align}
  \label{eq:relerr} \frac{\Vert  \mathbf{e}_h^n\Vert_{\mathbb{V}^h_0} }{\Vert   \mathbf{e}_h^1\Vert_{\mathbb{V}^h_0}}  = \frac{ \Vert   \mathbf{u}_h^n\Vert_{\mathbb{V}^h_0 } }{ \Vert   \mathbf{u}_h^1\Vert_{\mathbb{V}^h_0 }},   \end{align}
{where $\|\cdot\|_{\mathbb{V}^h_0}$ is} defined in  \eqref{eq:pseudo-energy2}.  The nodal values of the starting guess  $u_h^0\in V^h$ were chosen to be  uniformly  distributed in  the unit disc in the complex plane. The relative error \eqref{eq:relerr} was computed with respect to the first iterate $ \mathbf{u}_h^1 \in \mathbb{V}_0^h$, because  the initial guess $u_h^0$ is not in this space. 

\begin{table}[H]
\setlength\extrarowheight{2pt} 
\centering
\scalebox{0.9}{
\begin{tabularx}{\textwidth}{CC|CCCCC}
\hline
\hline
&${k} \backslash {L}$ & $1$  &  $2$	& $4$  &  $8$	&$16$  		\\ 
\hline 
\multirow{4}{*}{${\rho}({k}, \frac{1}{3} {L},  {L})$} &$10$	&0.169	&0.0863	&0.0385	&0.0153	&0.00952		\\
&$20$&0.190	&0.0997	&0.0382	&0.0175	&0.00909		\\
&$40$&0.234	&0.116	&0.0434	&0.0205	&0.00884		\\
&$80$&0.284	&0.148	&0.0557	&0.0231	&0.0115	\\
\hline
	\hline
\end{tabularx}}

\vspace{0.2cm}

\scalebox{0.9}{
\begin{tabularx}{\textwidth}{CC|CCCCC}
\hline
\hline
&${k} \backslash {L}$ & $1$  &  $2$	& $4$  &  $8$	&$16$  		\\ 
\hline 
\multirow{4}{*}{${\gamma}({k}, \frac{2}{3} {L},  {L})$} &$10$	&0.958	&0.834	&0.641	&0.382	&0.135		\\
&$20$&0.999	&0.982	&0.896	&0.786	&0.603		\\
&$40$&0.999	&0.999	&0.990	&0.943	&0.883		\\
&$80$&1.000	&1.000	&0.999	&0.995	&0.970	\\
\hline
	\hline
\end{tabularx}}
\caption{Numerical computation of ${\rho}({k},L/{3}, L )$ and ${\gamma}({k},2L/{3}, L )$ for increasing  ${L}$, $h=80^{-{5}/{4}}$}
\label{tb:imp-left2right-aspect}
\end{table}
\begin{figure}[H]
    \centering
    \begin{subfigure}[t]{0.33\textwidth}
        \centering
        \includegraphics[width=.95\textwidth]{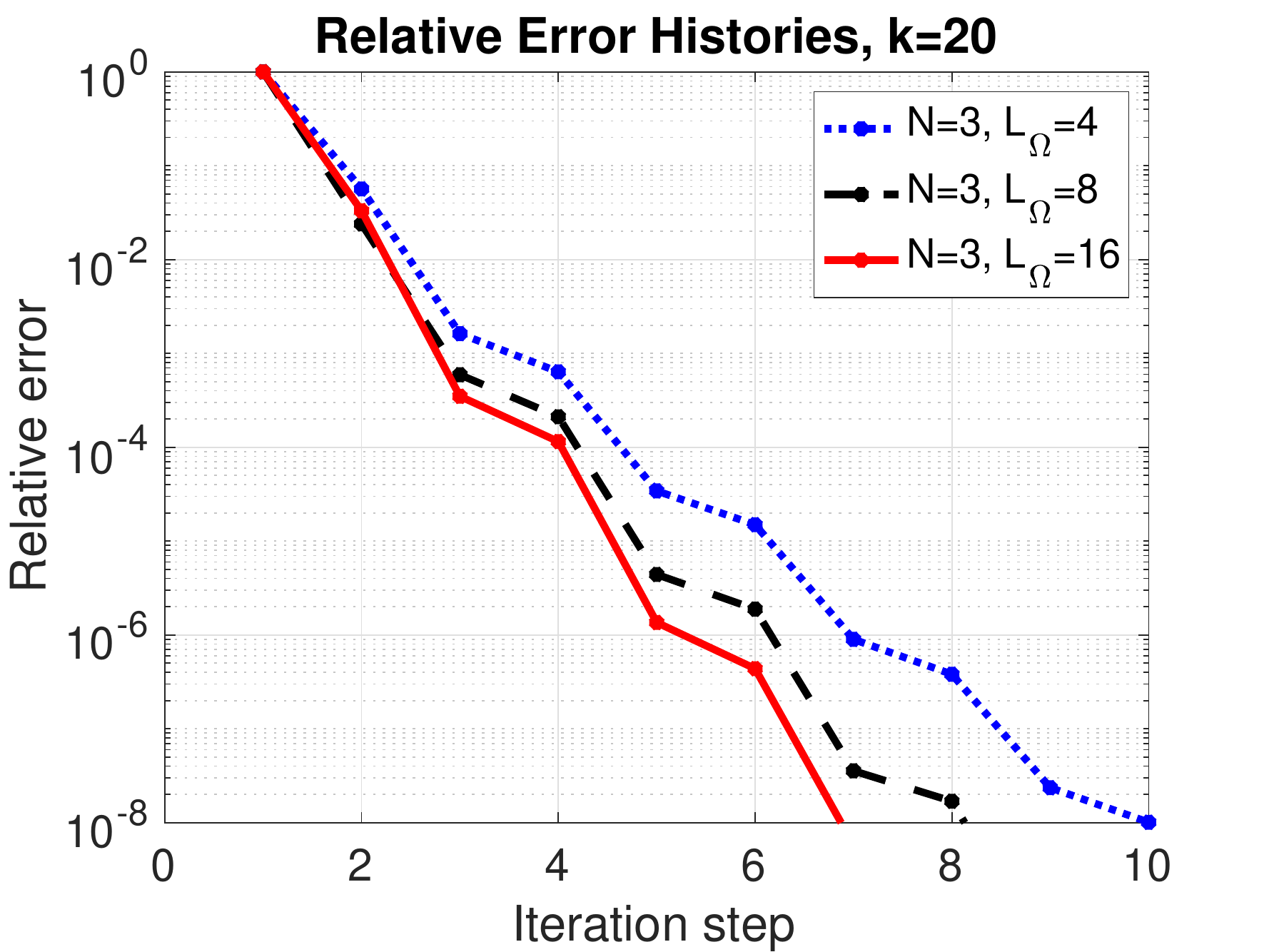}
        \caption{$k=20$\label{fig:errork20N3}}
         \end{subfigure}%
     \hfill
    \begin{subfigure}[t]{0.33\textwidth}
        \centering
        \includegraphics[width=.95\textwidth]{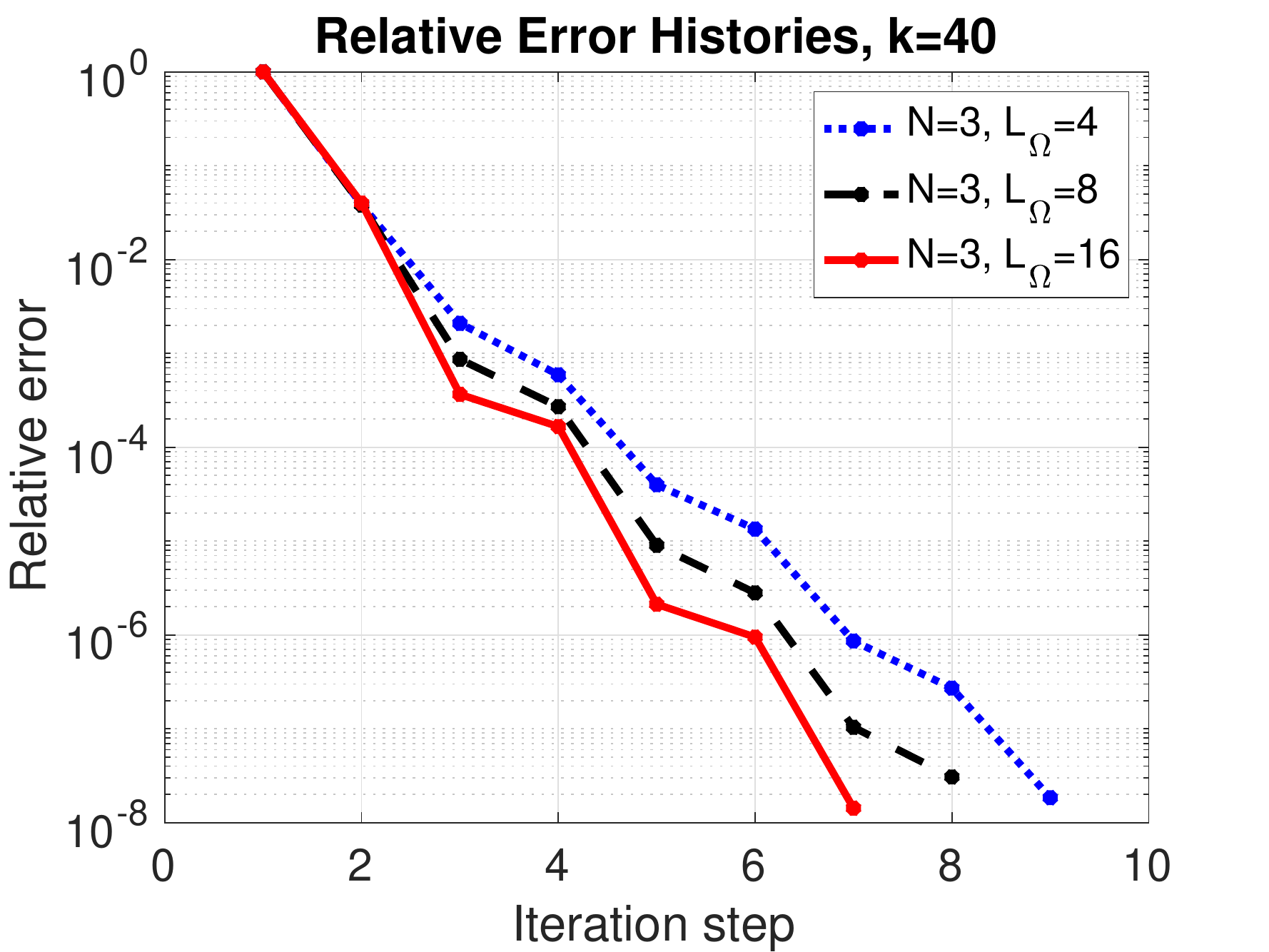}
        \caption{$k=40$\label{fig:errork40N3}}
        \end{subfigure}%
           \begin{subfigure}[t]{0.33\textwidth}
        \centering
        \includegraphics[width=.95\textwidth]{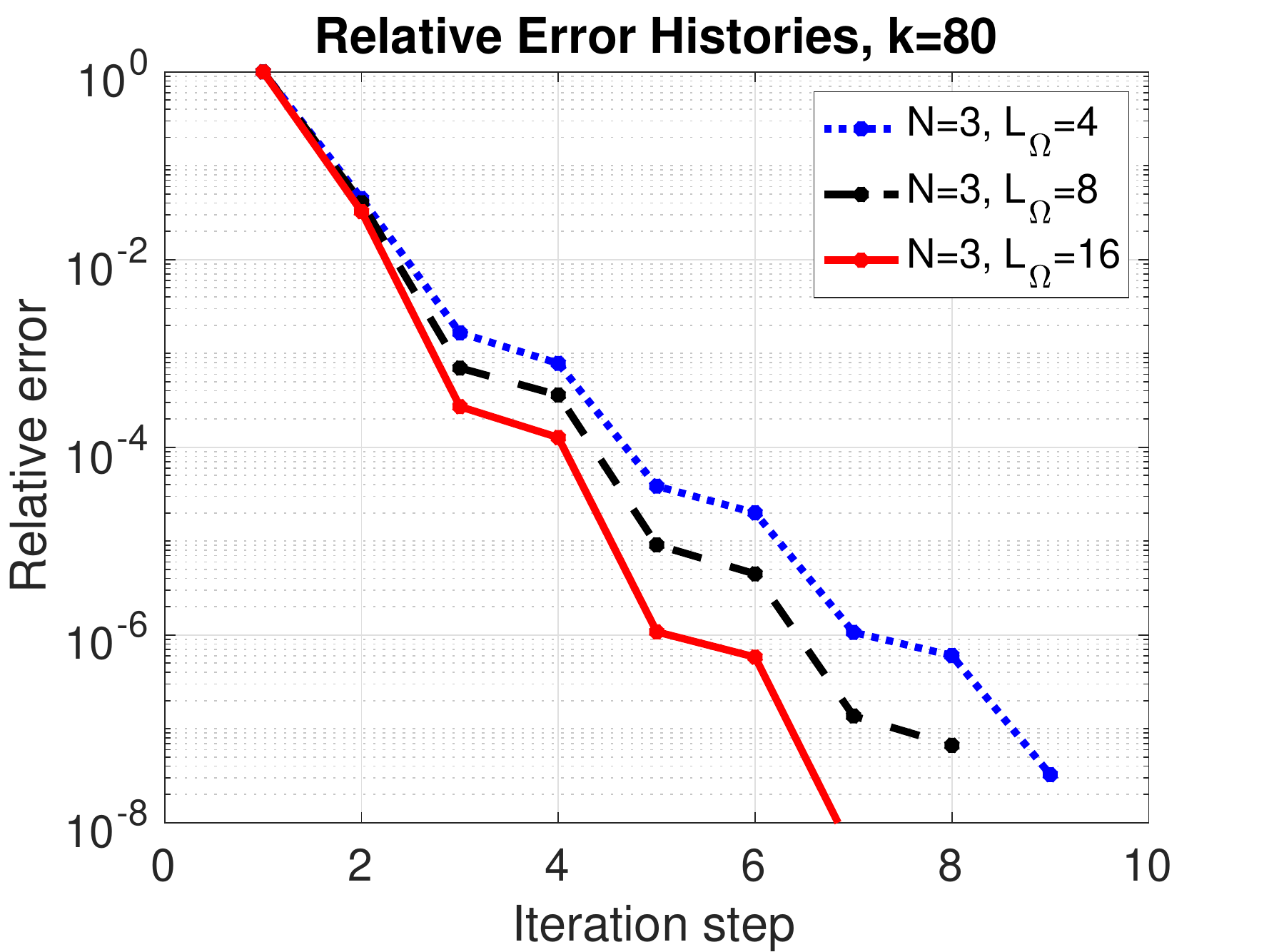}
        \caption{$k=80$ \label{fig:errork80N3}}
         \end{subfigure}%
           \caption{Relative error histories of the iterative method with 3 strip-type subdomains}\label{fig:errorN3}
\end{figure}

Figure \ref{fig:errorN3} shows that the convergence rate improves 
 as $L$ and hence $L_\Omega$ increases. This is 
 consistent with Corollary \ref{coro:TLU}, which shows that with $N$ fixed and $\gamma$ bounded, the iterative method is power contractive for small enough $\rho$. 
 The convergence rate is apparently unaffected by increasing  $k$, a bit better than  expected from the $k$-dependence of $\rho$ in Table \ref{tb:imp-left2right-aspect}. 
 \end{experiment}

The next experiment  investigates the effect of letting the number of subdomains $N$ grow. In this case,  Corollary \ref{coro:TLU} guarantees contractivity of  $\cT^N$
only for small  enough $N$.
However we  see that in fact the iterative method continues to work well as $N$ grows. The explanation for this is 
that, as discussed in \S\ref{sec:composite}, the composite impedance-to-impedance maps are better behaved than the individual ones; this is illustrated in
Experiment \ref{Expt3} {below}.

\begin{experiment}[Dependence on $N$]  \label{Expt2}
  We repeat the experiments in Figure \ref{fig:errorN3} but instead of $N=3$ (i.e, $3$ subdomains) we use  $N=4,8,16$. For each $N$, we  choose $\mathrm{L_\Omega}$  so  that the sizes of the subdomains and  overlaps do not depend on  $N$,  and thus   $\rho$ and $\gamma$ remain fixed as $N$ grows.
  The subdomain length is $L = 2$ and the overlap is $\delta = L/3$. 
In Figure \ref{fig:error} we plot the relative error histories for $ k= 20,40,80$.

  The relative error histories show  a  sudden reduction of the error after each batch of $N$ steps, and, after each such reduction,   the convergence rate appears to be 
  higher than before.
  This can be partially explained by  Corollary \ref{coro:higherrho}; {indeed,}   
  as the number of iterations  $n$ passes through $sN$ for $s = 2,3, \ldots$, {the}  order of the
  estimate for the norm of $\cT^{n}$ increases from $\mathcal{O}(\rho^{s-1})$ to $\mathcal{O}(\rho^s)$. However this  explanation can not be completely rigorous because the coefficient of the powers of $\rho$ in Corollary \ref{coro:higherrho} also grows with $N$. To understand the behaviour of the iterative method better we need to consider composite maps,
  which is the purpose of  Experiment \ref{Expt3}.

\begin{figure}[H]
    \centering
    \begin{subfigure}[t]{0.33\textwidth}
        \centering
        \includegraphics[width=.95\textwidth]{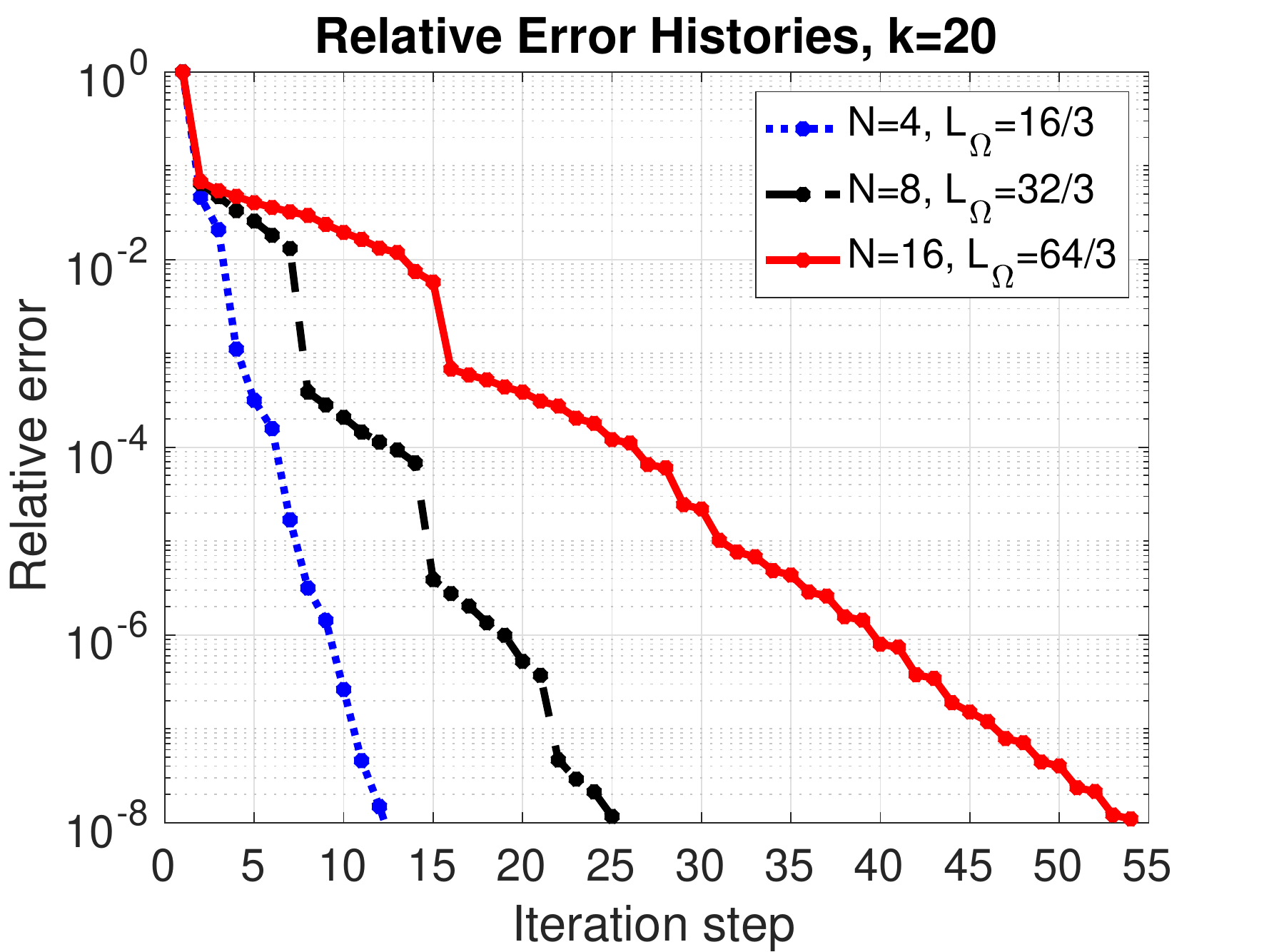}
        \caption{k=20\label{fig:errork20}}
         \end{subfigure}%
     \hfill
    \begin{subfigure}[t]{0.33\textwidth}
        \centering
        \includegraphics[width=.95\textwidth]{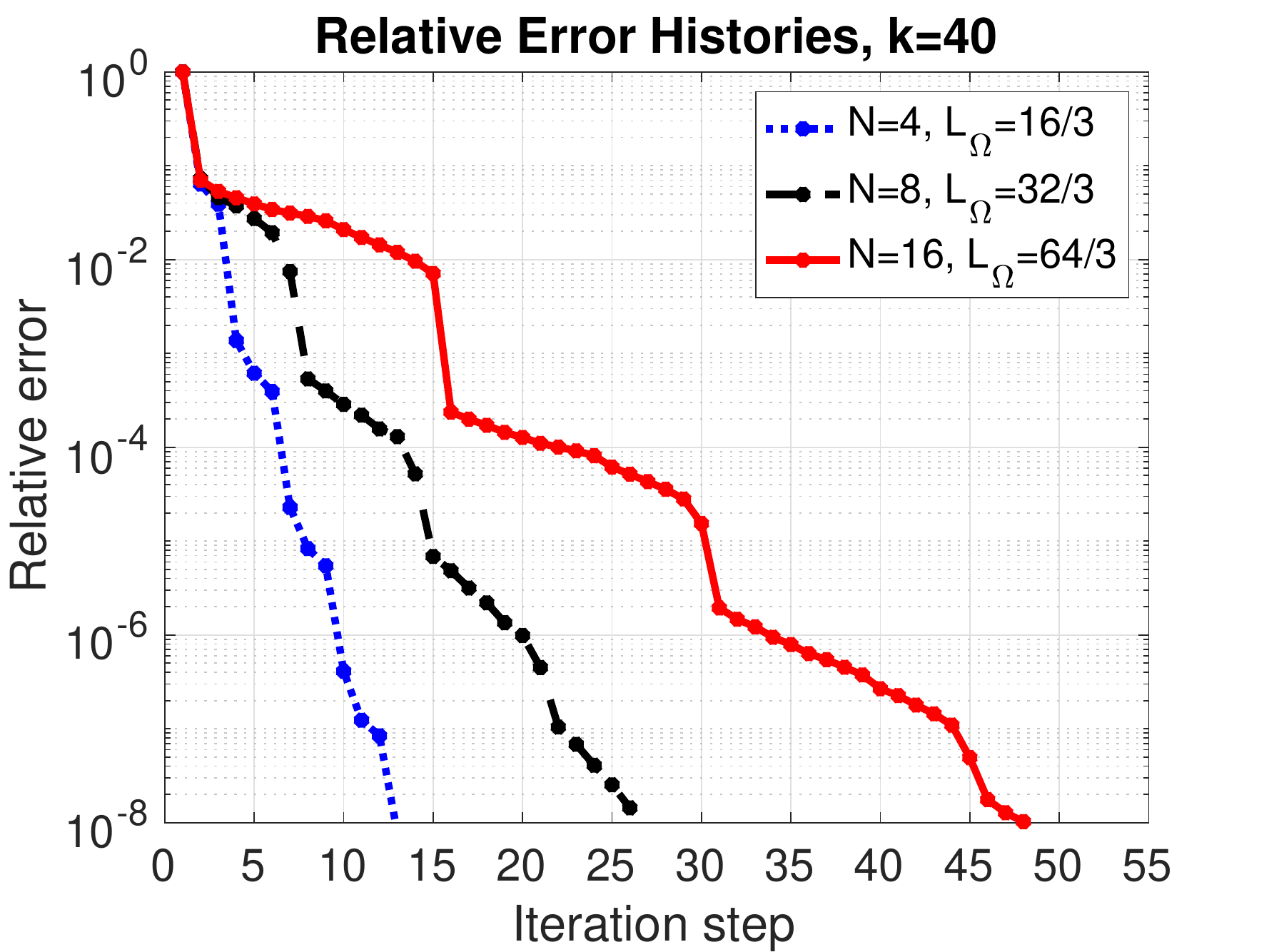}
        \caption{k=40 \label{fig:errork40}}
        \end{subfigure}%
           \begin{subfigure}[t]{0.33\textwidth}
        \centering
        \includegraphics[width=.95\textwidth]{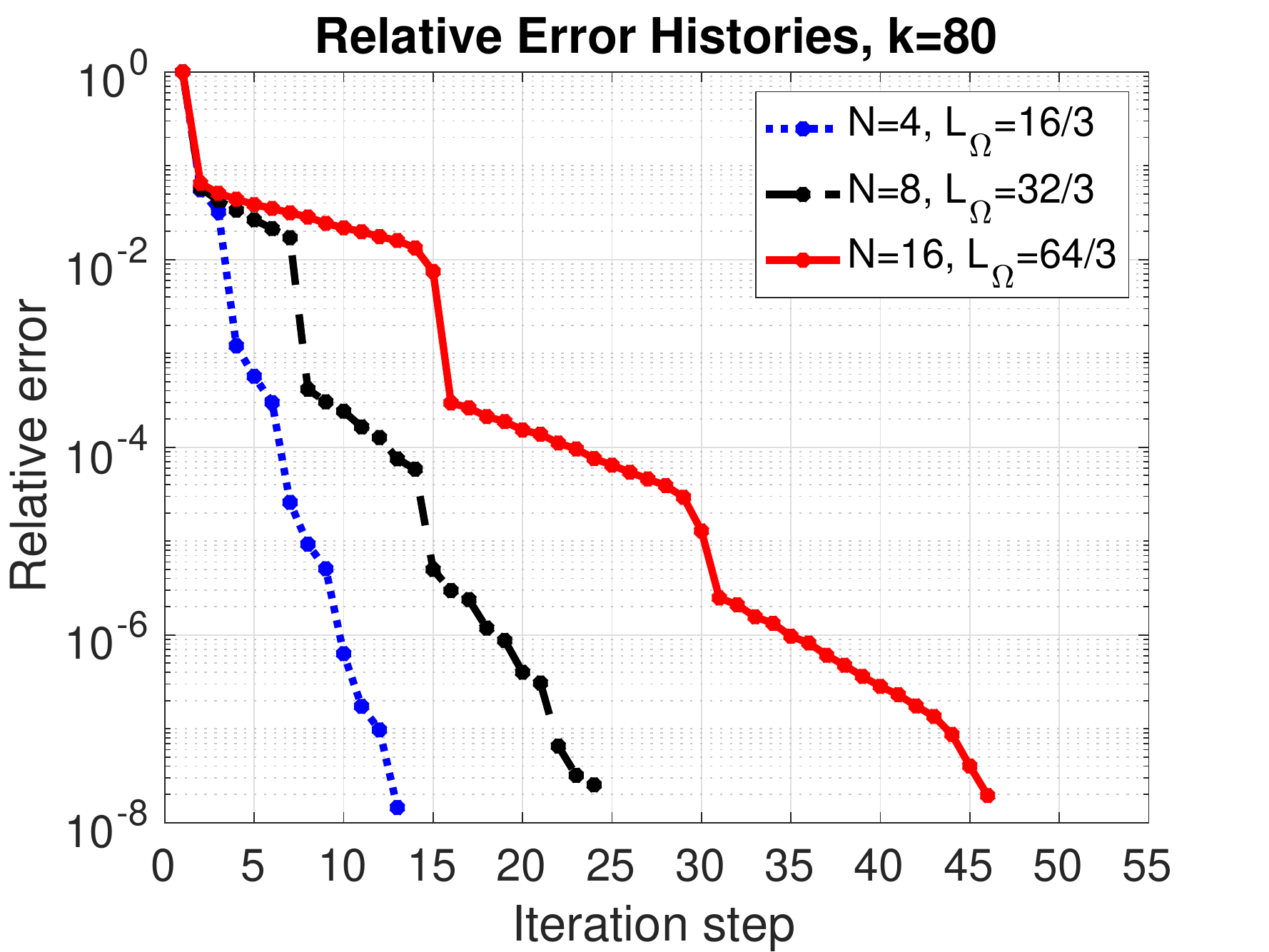}
        \caption{k=80 \label{fig:errork80}}
         \end{subfigure}%
           \caption{Relative error histories of the iterative method with many strip-type subdomains }\label{fig:error}
\end{figure}

Before that, Table \ref{tab:new1} {gives} the average number of iterations needed to reach a relative error of $10^{-6}$ for each of the scenarios  depicted in Figure   \ref{fig:error}, computed over 50 random starting guesses.  This table clearly indicates that the number of  iterations needed to obtain a fixed error tolerance is roughly  ${\mathcal{O}}(N)$ as $N$ grows.    We also observe modest improvement in the iteration numbers as $k$ increases; similar results were  seen in \cite[Table 3]{GrSpZo:20}.

\begin{table}[H]
\begin{center}
\scalebox{0.8}{
  \begin{tabularx}{\textwidth}{C|CCC}
\hline
\hline
$k \backslash N$ & $4$    & $8$	& $16$		\\ 
\hline 
$20$ & 6.00      & 12.34      & 25.12   		\\
$40$ & 5.58    & 10.16 &   16.96 		\\
$80$ &    4.44   & 8.00 & 15.88	\\

\hline
	\hline
      \end{tabularx}
    }
    \end{center} 
\caption{Average number of iterations to reach a relative error of $10^{-6}$ in Figure  \ref{fig:error}}
\label{tab:new1}
\end{table}

\end{experiment}

\begin{experiment}[Robustness to $N$ explained via composite  maps]  
\label{Expt3}  

As discussed in \S\ref{sec:composite}, the dominant term in \eqref{eq:composite} with $n=N$ is the $j=1$ term \eqref{eq:first_order}
  The goal of this experiment is to show that the behaviour of  \eqref{eq:first_order} is better than that predicted by estimating its norm by the product of the norms of its components (as in  \eqref{eq:compositeLU}).
Following  \eqref{eq:induction}, 
for $ N = 4, 8, 16$, $L=2$ and $\delta =L/3$, 
we  compute 
\begin{align}\label{eq:composit-imp-norm}
{\zeta}_N  & : =  2(N-1) \left \Vert \left(\prod_{j=2}^{N-1} \Imaps{j}{-}{j+1}{-}\right) \Imaps{1}{+}{2}{-}\right \Vert_{L^2(\Gamma_1^+) \ra L^2(\Gamma_{N}^-) } ,
\end{align}
{and use this as a proxy for \eqref{eq:first_order}, with this replacement justified by \eqref{eq:induction}, and the fact that $\bcL^{N-1}\bcU$ is a representative element of $\{ p(\bcL, \bcU) : p \in \cP(N,1)\}$.}
The results in Table \ref{tb:composite-imp1} show that $\zeta_N$ remains small and bounded
as $N$ increases. Although we have here computed only one  term in \eqref{eq:first_order},  
this gives some explanation why the convergence {rate} of the iterative method remains stable as $N$ increases, as observed 
in Figure  \ref{fig:error} and Table \ref{tab:new1}.

\begin{table}[H]
\setlength\extrarowheight{2pt} 
\centering
\scalebox{0.9}{
\begin{tabularx}{\textwidth}{C|CCCC}
\hline
\hline
 $   {k}$    &10   &  20	&40&80	\\ 
\hline
${\zeta}_2 = 2 \rho$& 		$1.74$e$-1$& 	$1.95$e$-1$&	$2.32$e$-1$&	$2.98$e$-1$	\\
$ {\zeta}_4$&	$4.06$e$-2$&	$9.28$e$-2$&	$1.20$e$-1$&  $1.33$e$-1$	\\
$ {\zeta}_8$&	$3.32$e$-2$&	$8.52$e$-2$&	$1.28$e$-1$&  $1.14$e$-1$		\\
$ {\zeta}_{16}$&$8.86$e$-3$&	$1.08$e$-1$&	$1.12$e$-1$&	$1.35$e$-1$	\\
\hline
	\hline
\end{tabularx}}
\caption{Norm of the composite impedance-to-impedance map \eqref{eq:composit-imp-norm}, $\delta=L/3$ 
}\label{tb:composite-imp1}
\end{table}
\end{experiment}

For the most efficient parallel implementations, the overlap $\delta$ should be as small as possible.
In our final experiment for the strip domain we {therefore} study the dependence of the convergence of the iterative method on the overlap parameter.

\begin{experiment}[Dependence on  overlap]
  \label{Expt4} 

  In this experiment we fix $ k = 40$ 
   and repeat Experiment \ref{Expt2}, with $N = 4,8, 16$,  comparing  the previous overlap choice {$\delta = L/3$  with   $\delta = L/6$ and $ 2h$.} 
   Here, the length of the global domain $L_\Omega= N(L-\delta)$ is chosen so that $L=2$, i.e., the subdomains have length $2$. 
   In Figure \ref{fig:error-overlap} we plot the relative error histories.
  These {histories} indicate that for small $N$ there is quite a big difference in performance between $\delta = 2h$ and the other two choices of $\delta$. However, as  $N$ increases the difference between the three choices of overlap becomes less pronounced.  With $N=16$ we again see clearly the `staircase' form of the error decay, as in Experiment \ref{Expt2}.
  
  \begin{figure}[H]
    \centering
    \begin{subfigure}[t]{0.33\textwidth}
        \centering
        \includegraphics[width=.95\textwidth]{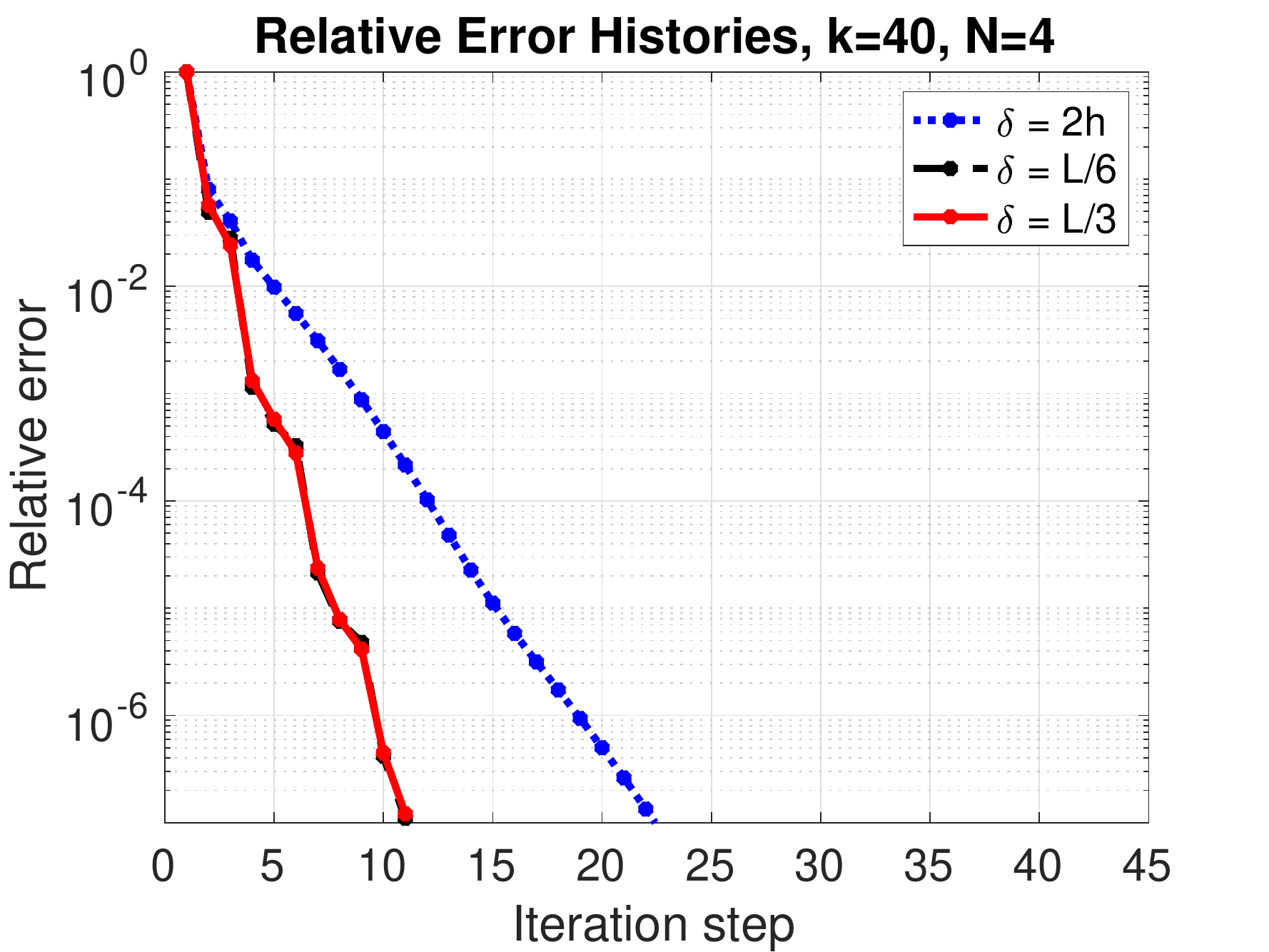}
        \caption{$N=4$\label{fig:errork80N4}}
         \end{subfigure}%
     \hfill
    \begin{subfigure}[t]{0.33\textwidth}
        \centering
        \includegraphics[width=.95\textwidth]{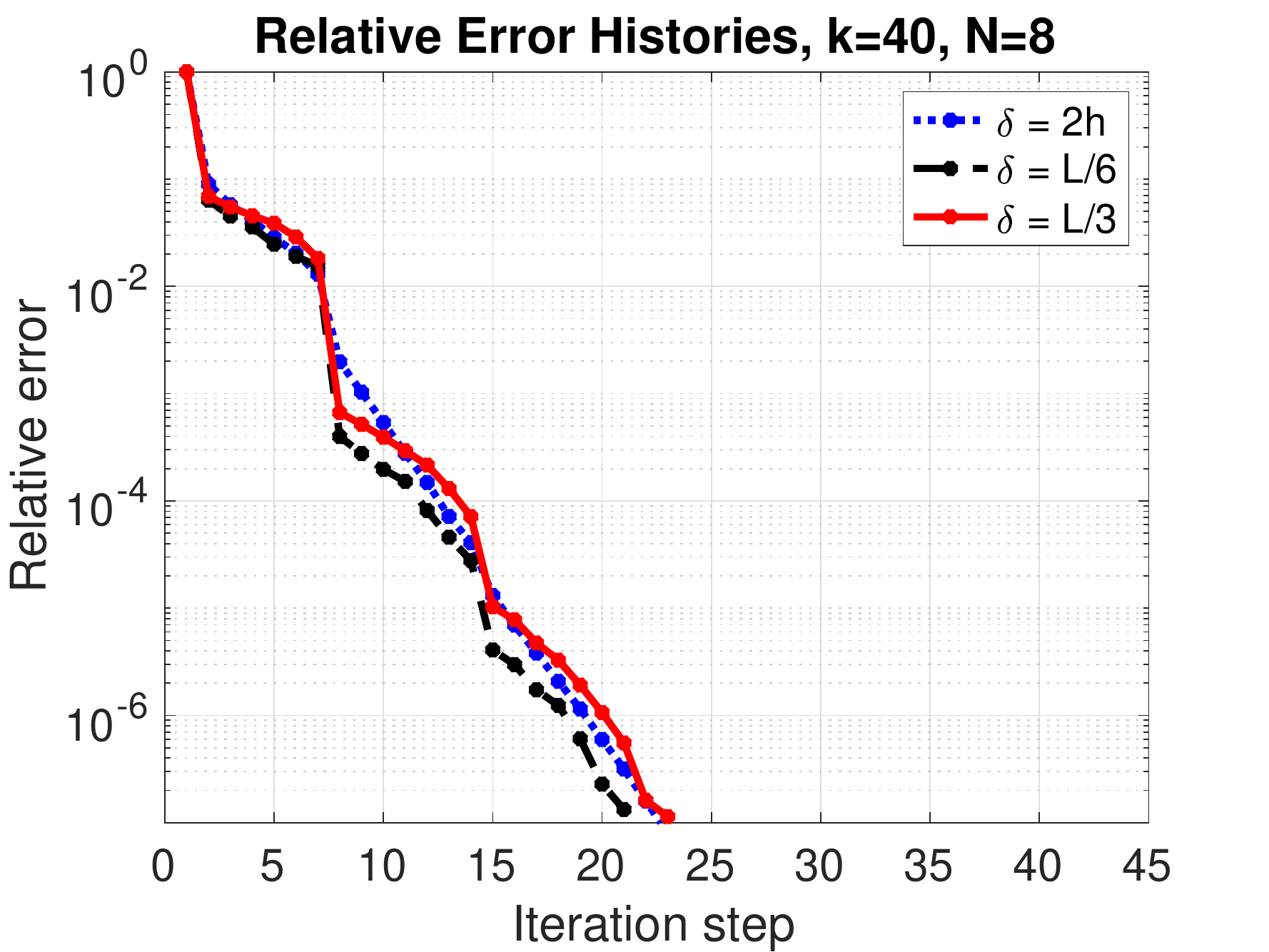}
        \caption{$N=8$ \label{fig:errork80N8}}
        \end{subfigure}%
           \begin{subfigure}[t]{0.33\textwidth}
        \centering
        \includegraphics[width=.95\textwidth]{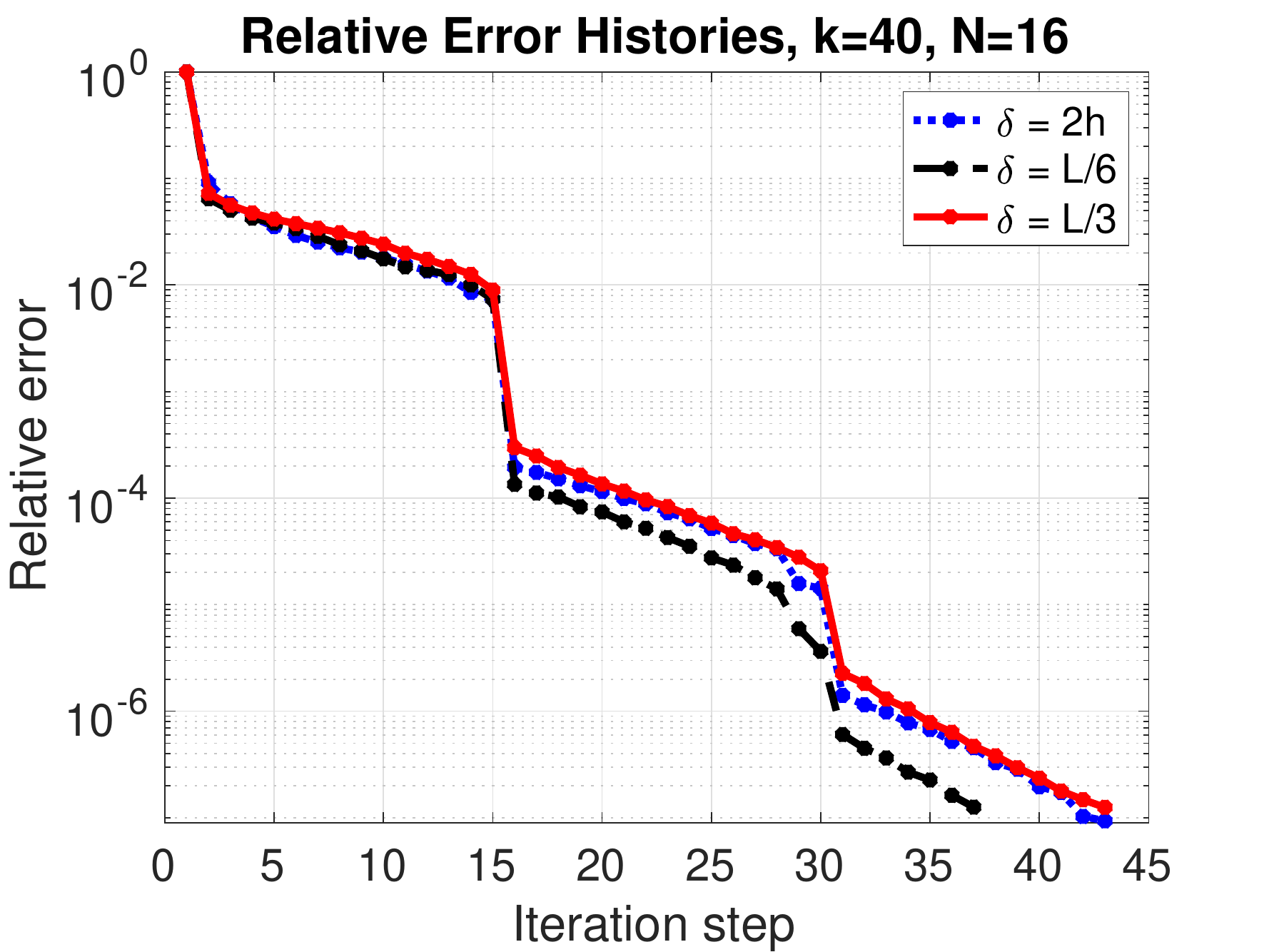}
        \caption{$N=16$ \label{fig:errork80N16}}
         \end{subfigure}%
           \caption{Relative error histories of the iterative methods with different overlaps }\label{fig:error-overlap}
\end{figure}

To give some heuristic explanation for Figure \ref{fig:error-overlap}, 
Tables \ref{tb:composite-imp2} and \ref{tb:composite-imp3} {provide} the analogous results to  Table \ref{tb:composite-imp1} for the new choices of overlap. 
As $N$ and $k$ increase, 
{the different choices of overlap all give similar values of $\zeta_N$}, thus explaining the competitiveness of the small  overlap method in this case.  

\begin{table}[H]
\setlength\extrarowheight{2pt} 
\centering
\scalebox{0.9}{
\begin{tabularx}{\textwidth}{C|CCCC}
\hline
\hline
 $   {k}$    &10   &  20	&40&80	\\ 
\hline
${\zeta}_2 = 2\rho$& 		$3.38$e$-1$& 	$3.60$e$-1$&	$4.60$e$-1$&	$5.64$e$-1$	\\
$ {\zeta}_4$&	$4.46$e$-2$&	$6.72$e$-2$&	$1.06$e$-1$&  $1.08$e$-1$	\\
$ {\zeta}_8$&	$2.82$e$-2$&	$8.88$e$-2$&	$1.02$e$-1$&  $1.06$e$-1$		\\
$ {\zeta}_{16}$&$4.48$e$-3$&	$9.92$e$-2$&	$4.66$e$-2$&	$7.86$e$-2$	\\
\hline
	\hline
\end{tabularx}}
\caption{Norm of the composite impedance-to-impedance map \eqref{eq:composit-imp-norm}, $\delta=L/6$ 
}\label{tb:composite-imp2}
\end{table}

\begin{table}[H]
\setlength\extrarowheight{2pt} 
\centering
\scalebox{0.9}{
\begin{tabularx}{\textwidth}{C|CCCC}
\hline
\hline
 $   {k}$    &10   &  20	&40&80	\\ 
\hline
${\zeta}_2 = 2\rho$& 		$7.68$e$-1$& 	$1.12$e$0$&	$1.40$e$0$&	$1.60$e$0$	\\
$ {\zeta}_4$&	$8.32$e$-2$&	$6.34$e$-2$&	$9.54$e$-2$&  $1.07$e$-1$	\\
$ {\zeta}_8$&	$4.44$e$-2$&	$7.00$e$-2$&	$8.48$e$-2$&  $7.30$e$-2$		\\
$ {\zeta}_{16}$&$4.98$e$-3$&	$6.92$e$-2$&	$2.94$e$-2$&	$8.34$e$-2$	\\
\hline
	\hline
\end{tabularx}}
\caption{Norm of the composite impedance-to-impedance map \eqref{eq:composit-imp-norm}, $\delta=2h$ 
}\label{tb:composite-imp3}
\end{table}
\end{experiment}

As discussed in \S\ref{subsec:Impedance}, 
the parameters $\rho, \gamma$ are computed {above} by the finite-element method, and 
Corollary \ref{cor:main} {ensures} that these {approximations} converge to the true values of $\rho, \gamma$ as $h \rightarrow 0$.
In practice, we compute $\rho, \gamma$ with
$h \sim k^{-5/4}$,  which is sufficient for ensuring a bounded error for the   Helmholtz problem as $k $ increases {by \cite[Corollary 5.2]{DuWu:15}}.
The following experiment 
{shows}  that this choice of $h$ also leads to accurate {computation of the}  impedance-to-impedance maps.

\begin{experiment}[Accuracy of the impedance-to-impedance map computation]\label{Expt5} 
  We compute  $\rho$  on the canonical domain ${\wOmega}$ depicted  in Figure
    \ref{fig:hat}, with $L=1$ and $\delta = 1/3$ for increasing $k$.
  In Table \ref{tb:imp-2algs}, we list computed values for  ${\rho}({k},{1}/{3},1)$
  (i.e., the norm of the left-to-right impedance-to-impedance map  - see \eqref{def:rho_gamma}), using mesh sizes  $h$,
  chosen as decreasing multiples of $k^{-5/4}$.  
  A `brute force' computation of
  the impedance-to-impedance map by numerical differentiation of the finite-element solution 
  {gave} almost identical results to those given in Table \ref{tb:imp-2algs}; we {therefore} conclude that the  computation of $\rho$ is sufficiently
  accurate when $h ={k}^{-{5}/{4}}$.   

\begin{table}[H]
\setlength\extrarowheight{2pt} 
\centering
\scalebox{0.9}{
\begin{tabularx}{\textwidth}{C|CCCC}
\hline
\hline
\multirow{1}{*}{${k} \backslash h$}& \multicolumn{1}{c}{${k}^{-{5}/{4}}$}& \multicolumn{1}{c}{$\frac12 {k}^{-{5}/{4}}$ }& \multicolumn{1}{c}{$\frac13 {k}^{-{5}/{4}}$}& \multicolumn{1}{c}{$\frac14 {k}^{-{5}/{4}}$} \\ 
\hline 
$10$	&0.171		&0.171		&0.171		&0.171	\\
$20$&0.188		&0.189		&0.190		&0.191	\\
$40$&0.235		&0.234		&0.236		&0.236	\\
\hline
	\hline
\end{tabularx}}
\caption{Numerical computation of ${\rho}({k},{1}/{3},1)$ 
}\label{tb:imp-2algs}
\end{table}

\end{experiment}

\subsection{Domain decompositions that are not of strip type}

\begin{experiment}[Unit square with uniform checkerboard decomposition] \label{Expt6} 
We partition  \\$\Omega := (0,1)^2$ into $N^2$ non-overlapping equal subsquares each with side length  $H = 1/N$, and then extend to an overlapping cover    by adding to each subdomain  neighbouring elements that have  distance  $\leq \delta$ 
from its boundary (so the actual overlap is $2\delta$). Tables \ref{tb:checkerboard1}-\ref{tb:checkerboard3} give {the} iteration counts   for the
    method \eqref{eq21}--\eqref{star} required to achieve a reduction of $10^{-6}$ in the 
Euclidean norm of the relative residual (with zero right-hand side and starting from a random initial guess), with overlap parameter $\delta = H/4, H/10$, and $h$, respectively. 
    We also give {(in brackets in each table)} the number of iterations needed by the  corresponding GMRES-accelerated iteration (that is GMRES using the  \ig{`ORAS' preconditioner implicit in \eqref{eq:discT}, \eqref{eq:update} -- see also \cite{GoGaGrSp:21}} ) to obtain a relative residual of $10^{-6}$.

The number of iterations of the iterative method grows  somewhere between $\mathcal{O}(N)$ and $\mathcal{O}(N^2)$, where $N^2$ is the number of subdomains.
    In contrast,  the number of GMRES iterations seems to grow somewhat more slowly -  close to $\mathcal{O}(N)$.   It appears that while the iterative method is not effective as a solver when the subdomains do not contain enough wavelengths, the GMRES method continues to function acceptably and is robust or even improving as $k$ increases. (Related theory and observations are given around \cite[Table 3]{GrSpZo:20}.)
    In contrast to the strip domain case, we also observe  that reducing the overlap does significantly affect the performance of both the iterative method and GMRES.

\begin{table}[H]
\setlength\extrarowheight{2pt} 
\centering
\scalebox{0.9}{
\begin{tabularx}{\textwidth}{C|CCC}
\hline
\hline
 $ k\backslash N\times N$ & $2\times2$  &  $4\times4$	 & $8\times8$ 	\\ 
\hline 			
40&	5 (5)&	14 (13)&	45 (28)\\
80&	5 (5)&	13 (12)&	29 (25)\\
120&	5 (5)&	12 (11)&	41 (24)\\
160&	4 (4)&	11 (10)&	29 (23)\\
\hline
	\hline
\end{tabularx}}
\caption{Checkerboard: Iteration counts for the iterative method (GMRES),  $\delta=H/4$
 }\label{tb:checkerboard1}
\end{table}

\begin{table}[H]
\setlength\extrarowheight{2pt} 
\centering
\scalebox{0.9}{
\begin{tabularx}{\textwidth}{C|CCC}
\hline
\hline
 $ k\backslash N\times N$ & $2\times2$  &  $4\times4$	 & $8\times8$ 	\\ 
\hline 			
40&	7 (6)&	32 (21)&	200+ (35)\\
80&	6 (6)&	19 (21)&	200+ (37)\\
120&	6 (6)&	18 (20)&	200+ (35)\\
160&	5 (5)&	18 (20)&	76 (30)\\
\hline
	\hline
\end{tabularx}}
\caption{ Checkerboard: Iteration counts for the iterative method (GMRES), $\delta=H/10$
 }\label{tb:checkerboard2}
\end{table}

\begin{table}[H]
\setlength\extrarowheight{2pt} 
\centering
\scalebox{0.9}{
\begin{tabularx}{\textwidth}{C|CCC}
\hline
\hline
 $ k\backslash N\times N$ & $2\times2$  &  $4\times4$	 & $8\times8$ 	\\ 
\hline 			
40&	16 (12)&	27 (21)&	101 (39)\\
80&	21 (14)&	30 (22)&	200+ (42)\\
120&	25 (16)&	34 (23)&	200+ (43)\\
160&	30 (17)&	36 (24)&	200+ (42)\\
\hline
	\hline
\end{tabularx}}
\caption{ Checkerboard: Iteration counts for the iterative method (GMRES), $\delta=h$
 }\label{tb:checkerboard3}
\end{table}

\end{experiment}

\begin{experiment}[Uniform square with domain decomposition via METIS]\label{Expt7}   
{We consider the same set-up as in}  Experiment \ref{Expt6}, {but} instead of using the checkerboard domain decomposition,
  we use METIS to generate a non-overlapping domain decomposition -- see Figure \ref{fig:metis} for {plots for} $N=4,16,64$ subdomains -- and then we extend to an overlapping cover in the same way as before.  Tables \ref{tb:metis1}-\ref{tb:metis3} give iteration counts  for  the iterative method (and GMRES in brackets) for each choice of $\delta$.  The iteration counts 
  {behave similarlly} to those given in Experiment \ref{Expt6}. Again, we notice the almost-robustmess of the
  GMRES method as $k$ increases for all choices of $\delta$, and particularly 
  for  generous overlap.

\begin{figure}[H]
    \centering
    \begin{subfigure}[t]{0.33\textwidth}
        \centering
        \includegraphics[width=.95\textwidth]{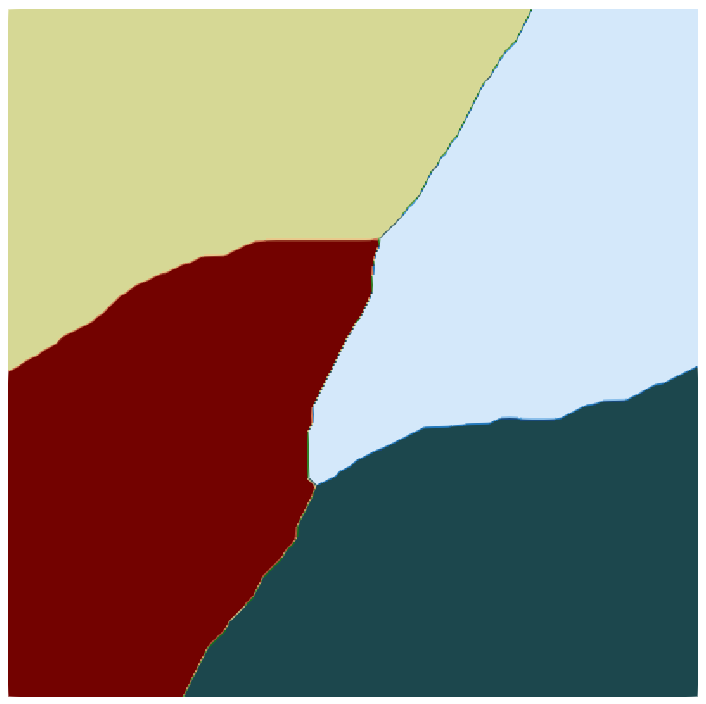}
        \caption{4 subdomains}\label{fig:metis4}
         \end{subfigure}%
     \hfill
    \begin{subfigure}[t]{0.33\textwidth}
        \centering
        \includegraphics[width=.95\textwidth]{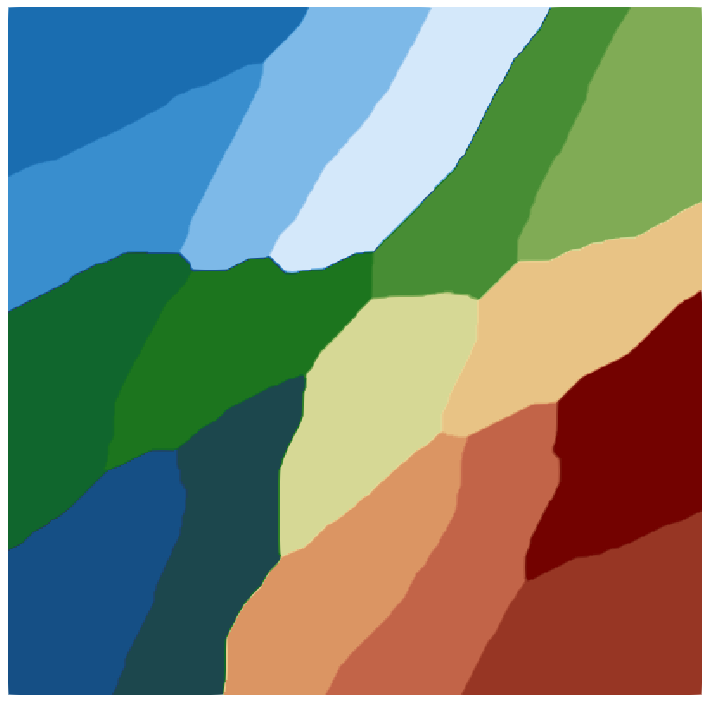}
        \caption{16 subdomains}\label{fig:metis16}
        \end{subfigure}%
           \begin{subfigure}[t]{0.33\textwidth}
        \centering
        \includegraphics[width=.95\textwidth]{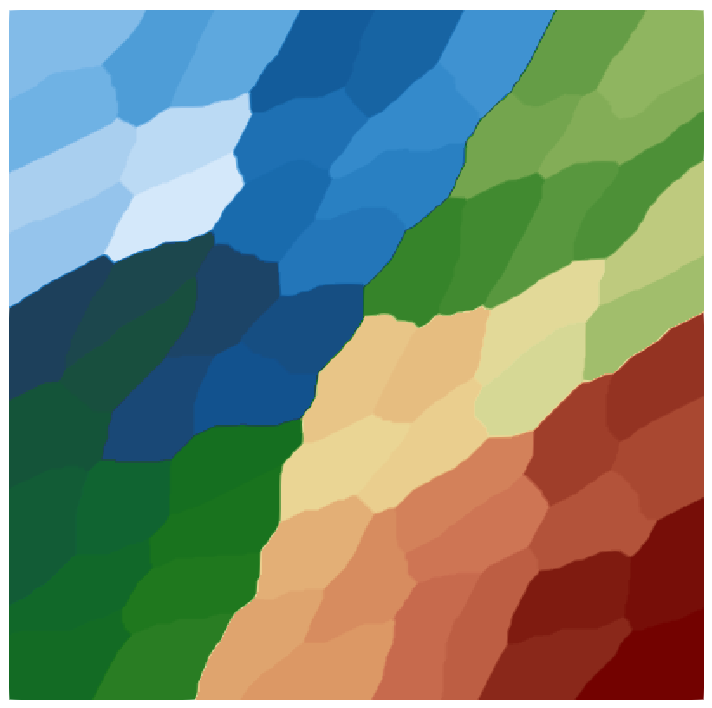}
        \caption{64 subdomains}\label{fig:metis64}
         \end{subfigure}%
           \caption{METIS non-overlapping domain decompositions}\label{fig:metis}
\end{figure}

\begin{table}[H]
\setlength\extrarowheight{2pt} 
\centering
\scalebox{0.9}{
\begin{tabularx}{\textwidth}{C|CCC}
\hline
\hline
 $ k\backslash N$ & $4$  &  $16$	 & $64$ 	\\ 
\hline 			
40&	8 (7)&	20 (17)&	73 (39)\\
80&	7 (7)&	19 (17)&	57 (37)\\
120&	6 (6)&	17 (16)&	41 (33)\\
160&	6 (6)&	16 (15)&	40 (33)\\
\hline
	\hline
\end{tabularx}}
\caption{Number of iterations of the iterative method and GMRES counts (in brackets),  METIS domain decomposition for the unit square, $\delta=H/4$
 }\label{tb:metis1}
\end{table}

\begin{table}[H]
\setlength\extrarowheight{2pt} 
\centering
\scalebox{0.9}{
\begin{tabularx}{\textwidth}{C|CCC}
\hline
\hline
 $ k\backslash N$ & $4$  &  $16$	 & $64$ 	\\ 
\hline 			
40&	9 (9)&	27 (21)&	109 (45)\\
80&	9 (9)&	24 (21)& 	159 (47)\\
120&	8 (8)&	24 (20)&	104 (43)\\
160&	8 (7)&	23 (20)&	104 (41)\\
\hline
	\hline
\end{tabularx}}
\caption{Number of iterations of the iterative method, METIS domain decomposition for the unit square, $\delta=H/10$
 }\label{tb:metis2}
\end{table}

\begin{table}[H]
\setlength\extrarowheight{2pt} 
\centering
\scalebox{0.9}{
\begin{tabularx}{\textwidth}{C|CCC}
\hline
\hline
 $ k\backslash N$ & $4$  &  $16$	 & $64$ 	\\ 
\hline 			
40&	20 (14)&	33 (25)&	86 (48)\\
80&	27 (17)&	33 (26)&	86 (53)\\
120&	28 (18)&	36 (27)&	82 (51)\\
160&33 (20)&	200+ (30)& 200+ (53)	\\
\hline
	\hline
\end{tabularx}}
\caption{Number of iterations of the iterative method, METIS domain decomposition for the unit square, $\delta=h$
 }\label{tb:metis3}
\end{table}

  \end{experiment}

  \subsection{Possible approach to the theory for the checkerboard case}
  \label{subsec:possible}

  To finish the paper  we include some  remarks on how the fundamental general  theory developed here could be
  extended to get convergence estimates for  more general decompositions. Here we only (and rather tentatively) discuss  the case of a $2\times 2$ checkerboard decomposition depicted in Figure \ref{fig:checker}.
More   general decompositions could be approached by extending this reasoning, although, we also admit such reasoning
  will become increasingly complex for more general domain decompositions.

Consider the unit square divided into four quarters, yielding non-overlapping subdomains labelled 1,2,3,4 in Figure \ref{fig:checker}. These are extended to  overlapping subdomains $\Omega_\ell$, $\ell = 1,2,3,4$. We depict only $\Omega_1$ (in red) and $\Omega_2$ (in blue) in the figure.
  \begin{figure}[H]
    \centering
\vspace{-2cm} 
    \includegraphics[scale=0.4]{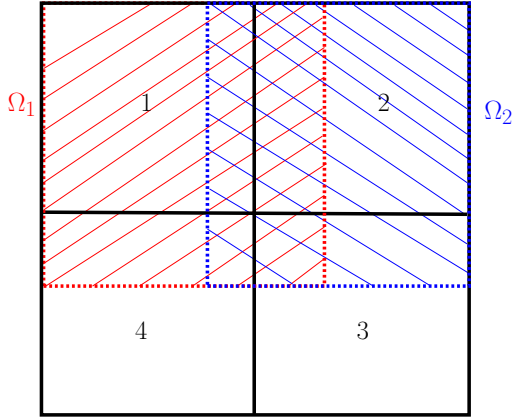}
    \vspace{-3cm}
    \caption{A simple checkerboard decomposition }\label{fig:checker} 
\end{figure} 
All the results of this paper up to the end of \S \ref{sec:fixedpoint} apply to this geometry.  With four subdomains, it is natural to look at the fourth power of the propagation operator $\cT$. It can easily be shown by induction that
  \begin{align} \label{eq:T4} 
  (\cT^4)_{j,\ell} = \sum_{\boldsymbol{m}   \in \cR(j,\ell,4)} \cT_{j,m_1}\cT_{m_1,m_2}\cT_{m_2,m_3}\cT_{m_3,\ell},\end{align} 
where $\cR(j,\ell,4)$ denotes the collection of all paths from domain $j$ to domain $\ell$ passing through three other domains en route.  Then  $\boldsymbol{m}=(m_1,m_2,m_3,m_4,m_5)$ represents a typical  path
\begin{align} \label{eq:sequence} j = m_1 \rightarrow m_2 \rightarrow m_3 \rightarrow m_4 \rightarrow m_5 = \ell\end{align}
from $\Omega_j$ to $\Omega_\ell$, with
  $m_1 \not = m_2$, $m_2 \not = m_3$, $m_3 \not = m_4$, and $m_4 \not = m_5$.

  The first thing to note is that many of the component products in   \eqref{eq:T4} have already been  investigated earlier in this paper.
  For example, consider the product $\cT_{1,2}\cT_{2,1}$; by Theorem \ref{thm:main_T2} we see that, as $k$ gets larger then the norm of this product is well-estimated by the norm of the impedance-to-impedance map
  $\cI_{\Gamma_{2,1}\rightarrow \Gamma_{1.2}}$.  This map operates only in the $x-$direction and was analysed  in \S  \ref{subsec:canonical}. In fact it corresponds to the left-to-right map defined by (a) finding the Helmholtz-harmonic function on $\Omega_2$ with data on $\partial \Omega_2 \cap \Omega_1$ and (b)  evaluating the right-facing impedance data of the solution on  $\partial \Omega_1\cap \Omega_2$. This map was analysed in \S \ref{sec:conv} and shown to have norm of   order $\mathcal{O}(\rho)$. Similar remarks apply to  $\cT_{2,1}\cT_{1,2}$, $\cT_{1,4}\cT_{4,1}$, etc. However the `diagonal switches' $\cT_{1,3}\cT_{3,1}$, $\cT_{3,1}\cT_{1,3}$, $\cT_{1,4}\cT_{4,1}$ and $\cT_{4,1}\cT_{1,4}$ are genuinely two-dimensional and the properties of the corresponding impedance maps would need to be studied numerically.

  More generally we could conjecture that if the sequence in \eqref{eq:sequence} visits the same subdomain twice then the norm of the corresponding term in \eqref{eq:T4} is small in size. This conjecture would have  to be examined numerically, but would then give estimates  for  all off-diagonal terms in \eqref{eq:T4}.    

  Finally, one would have to analyse the cycles which appear in the diagonal terms in \eqref{eq:T4}  e.g. the term corresponding to  $1 
   \rightarrow 2 \rightarrow 3 \rightarrow 4 \rightarrow 1$.

  This further analysis is outside the scope of the present paper. However, while  it is clear from experiments that higher powers of $\cT$ are still contractive in the checkerboard (and more general) cases, the convergence profile (at least in certain norms \cite{GoGaGrSp:21}) does not exhibit the same jumps when one passes from $n = sN$ to $n = (s+1)N$ as are present in the strip domain case (see Figures 
\ref{fig:errorN3}, \ref{fig:error}). 

\bigskip

\paragraph{Acknowledgements.} 
We gratefully acknowledge support from the UK Engineering and Physical Sciences Research Council  Grants EP/R005591/1 (DL and EAS) and   EP/S003975/1
(SG, IGG, and EAS).  This research made use of the Balena High Performance Computing (HPC) Service at the University of Bath.

{\footnotesize
\bibliographystyle{plain}
\bibliography{combined.bib}
}
\section*{Appendix}

\emph{Proof of Proposition \ref{prop:sumsum}}.  We prove \eqref{eq:prod} by induction. Clearly it holds for $n = 1$.
  Assuming it holds for any $n \geq 1$, then
  \begin{align} \label{eq:prop1} (x + y)^{n+1} \ =  
  &\  (x+y) \left(\sum_{j=1}^{n-1} \sum_{p\in \cP(n,j)}p(x,y) + x^n +y^n\right)
  \nonumber \\
   = & \
   \sum_{j=1}^{n-1} \left( \sum_{p\in \cP(n,j,x)} + \sum_{p \in \cP(n,j,y) } \right) x p(x,y) \ + \ xy^n \\
   & + \ \sum_{j=1}^{n-1} \left( \sum_{p\in \cP(n,j,x)} + \sum_{p \in \cP(n,j,y) } \right) y p(x,y) \
                                                      + \ yx^n \label{eq:prop2}\\
                                                    & + x^{n+1} + y^{n+1}, \label{eq:prop3}   \end{align}
          where $\cP(n,j,x), \cP(n,j,y)$ denote, respectively, the monomials of the
             form \eqref{eq:monomx}, \eqref{eq:monomy}. In this notation,
             \begin{align} \cP(n+1, j, x) & = \left( x \cP(n,j,x)\right) \cup \left( x \cP(n,j-1,y)\right) \quad \text{for} \quad 1 \leq j \leq n-1, \label{eq:prop4} \\ \cP(n+1,n,x) &= x\cP(n,n-1,y).\label{eq:prop5}  \end{align}
             Hence the term \eqref{eq:prop1} equals
             \begin{align}    & \quad \sum_{j=2}^{n-1} \left( \sum_{p\in \cP(n,j,x)} + \sum_{p \in \cP(n,j-1,y) } \right) x p(x,y) \       + \sum_{p \in \cP(n,1,x)} x p(x,y) + xy^n + \sum_{p \in \cP(n,n-1,y)} x p(x,y) \nonumber \\
               & = \ \sum_{j=2}^{n-1} \left( \sum_{p\in \cP(n,j,x)} + \sum_{p \in \cP(n,j-1,y) } \right) x p(x,y) \       + \left( \sum_{p \in \cP(n+1,1,x)}   +  \sum_{p \in \cP(n+1,n,x)}\right)   p(x,y) \nonumber \\
               & = \ \sum_{j=1}^n \sum_{p \in \cP(n+1,j,x)} p(x,y),
             \label{eq:prop6}\end{align}
           where in the second step we used both \eqref{eq:prop4} with $j=1$ and \eqref{eq:prop5}. A similar argument shows the term \eqref{eq:prop2} can also be written in the form \eqref{eq:prop6}, but with the sum over $\cP(n+1,j,x)$ replaced by the sum over $\cP(n+1,j,y)$.            Putting these results together with \eqref{eq:prop1} -- \eqref{eq:prop3} shows that \eqref{eq:prod} holds for $n+1$.

           The proof of \eqref{eq:count} also uses induction on $n$. Note that $\# \cP(1,0) = 2$, so the result holds for $n = 1$. If it holds for $n$ then it holds for $n+1$ by observing (analogously to \eqref{eq:prop4}), that
           $$ \# \cP(n+1, j) = \# \cP(n,j) + \# \cP(n,j-1) ,$$
           and then using elementary properties of the binomial coefficient. 
\qed

\end{document}